\pdfoutput=1
\documentclass[reqno]{amsart}

\usepackage{amsmath,amssymb,amsthm,verbatim, csquotes,mathrsfs}
\usepackage{hyperref}
\usepackage{xcolor}
\usepackage{esint}
\usepackage{mathtools}
\usepackage{graphicx}
\usepackage{float}
\usepackage{caption}
\mathtoolsset{showonlyrefs}
\usepackage{geometry}

\newtheorem{theorem}{Theorem}[section]
\newtheorem{lemma}[theorem]{Lemma}
\newtheorem{proposition}[theorem]{Proposition}
\newtheorem{definition}[theorem]{Definition}
\newtheorem{corollary}[theorem]{Corollary}
\newtheorem{remark}[theorem]{Remark}

\newtheorem{example}[theorem]{Example}

\def\om{\omega}
\def\Om{\Omega}
\def\p{\partial}
\def\ep{\epsilon}
\def\de{\delta}

\def\S{{\Sigma}}
\def\<{\langle}
\def\>{\rangle}

\def\na{\nabla}

\providecommand{\abs}[1]{\lvert#1\rvert}
\providecommand{\Abs}[1]{\left\lvert#1\right\rvert}

\providecommand{\norm}[1]{\lVert#1\rVert}

\newcommand{\mbH}{\mathbb{H}}

\newcommand{\mbN}{\mathbb{N}}

\newcommand{\mbP}{\mathbb{P}}

\newcommand{\mcC}{\mathcal{C}}

\newcommand{\mcF}{\mathcal{F}}

\newcommand{\mcH}{\mathcal{H}}

\newcommand{\mcR}{\mathcal{R}}

\newcommand{\mfn}{\mathbf{n}}

\newcommand{\mfH}{\mathbf{H}}

\newcommand{\mfR}{\mathbf{R}}
\newcommand{\mfS}{\mathbf{S}}
\newcommand{\mfV}{\mathbf{V}}
\newcommand{\mfRV}{\mathbf{R}\mathbf{V}}

\newcommand{\rd}{{\rm d}}

\newcommand{\wsc}{\overset{\ast}{\rightharpoonup}}

\newcommand{\bsmu}{\boldsymbol\mu}
\newcommand{\bstau}{\boldsymbol\tau}
\newcommand{\bseta}{\boldsymbol\eta}

\newcommand{\ra}{\rightarrow}

\newcommand{\eq}[1]{\begin{equation}\begin{alignedat}{2} #1 \end{alignedat}\end{equation}}

\numberwithin{equation} {section}

\begin{document}

	
\title[Varifolds with capillary boundary]{Varifolds with capillary boundary}
\date{\today}

\author[Wang]{Guofang Wang}
	\address[G.W]{Mathematisches Institut\\Universit\"at Freiburg\\Ernst-Zermelo-Str.1\\79104\\Freiburg\\ Germany}
	\email{guofang.wang@math.uni-freiburg.de}

\author[Zhang]{Xuwen Zhang}
	\address[X.Z]{Mathematisches Institut\\Universit\"at Freiburg\\Ernst-Zermelo-Str.1\\79104\\Freiburg\\ Germany}
\email{xuwen.zhang@math.uni-freiburg.de}

\begin{abstract}
In this paper we introduce and study a new class of varifolds in $\mfR^{n+1}$ of arbitrary dimensions and co-dimensions, which satisfy a Neumann-type boundary condition characterizing capillarity.
The key idea is to introduce a Radon measure on a subspace of the trivial Grassmannian bundle over the supporting hypersurface
as a generalized boundary with prescribed angle, which plays a  role as a measure-theoretic capillary boundary.
We show
several structural properties, monotonicity inequality, boundary rectifiability, classification of tangent cones, and integral compactness for such varifolds under reasonable conditions.
This Neumann-type boundary condition fits very well in the context of curvature varifold and Brakke flow, which we also discuss.

\

\noindent {\bf MSC 2020: 49Q15, 53C42}\\
{\bf Keywords:} varifold, capillarity, rectifiability, compactness, curvature varifold, Brakke flow\\
\end{abstract}

\maketitle
\tableofcontents
\section{Introduction}

In this paper, we are interested in a new varifold formulation of capillary submanifolds, which are submanifolds (with possibly prescribed mean curvature) in a given container that meet the supporting hypersurface at a prescribed contact angle.
In the co-dimension $1$ case these hypersurfaces arise as critical points of the {Gauss free energy}, and
its study from mathematical perspective 
goes back at least to Thomas Young \cite{Young1805} in 1805.
For  a historical overview we refer to a nice survey by Finn \cite{Finn99} and also his book \cite{Finn86}.
The mathematical framework is as follows: given a domain (open, connected set) $\Om\subset\mfR^{n+1}$ of class $C^1$ with boundary $S\coloneqq\p\Om$ as a container and functions $h\in C^1(\overline\Om),\beta\in C^1(S,(0,\pi))$, the \textit{Gauss free energy} for an open subset $E\subset\Om$ is defined as
\eq{\label{free_energy}
\mcF_\beta(E;\Om)
=\mcH^n(\p E\cap\Om)-\int_{\p\Om\cap\p E}\cos\beta(x)\rd\mcH^n+\int_Eh(x)\rd\mcH^{n+1}(x).
}
When 
everything is sufficiently regular, the well-known \textit{Young's law} states that a critical point of $\mcF_\beta$, with or without volume constraint, is a capillary hypersurface, 
which admits mean curvature  $h$ (plus a constant if volume-constrained) and meets $S$ at a contact angle prescribed by the function $\beta$.
We emphasize that the second term on the right in \eqref{free_energy}, which is usually called the {\it wetting energy}, determines the contact angle.
{%
For simplicity, we omit the third term in \eqref{free_energy} in this paper.

When $\cos \beta\equiv 0$ the free energy reduces to the relative perimeter, and the critical points give rise to \textit{free boundary hypersurfaces}, which meet the container orthogonally.
In the variational point of view, it is easy to check that when we consider $\mcF_{\frac\pi2}$ among all hypersurfaces with boundary moving { freely} on $S$, then we have Young's law. We call it the {\it free boundary phenomenon}.
Note that the relative perimeter
can be replaced by the area functional of submanifolds of any dimensions and any co-dimensions and hence co-dimension-$1$ is not essential to characterize free boundary phenomenon from the variational point of view.
However, when $\cos \beta \not \equiv 0$, to our knowledge a general capillary submanifold of higher co-dimension cannot be expressed as a critical point of certain functionals, because one can not find a suitable wetting energy.
}

Very recently, the existence, regularity, and geometric properties of capillary hypersurfaces have greatly interested differential geometers and geometric analysts.
Here we only mention some recent progress on this topic.
Using the Min-Max method, the existence of capillary minimal or CMC hypersurfaces in compact $3$-manifolds with boundary has been shown independently by De Masi-De Philippis \cite{DeMP21} and Li-Zhou-Zhu \cite{LZZ21}.
An Allard-type boundary regularity ($\ep$-regularity) for capillary hypersurfaces is studied independently by Wang \cite{Wang24} and De Masi-Edelen-Gasparetto-Li \cite{DEGL24}, under different conditions and using significantly different approaches.
These mentioned progress are mainly built on a varifold formulation of capillary hypersurfaces, which captures the geometric variational essence of such hypersurfaces and was first studied by Kagaya-Tonegawa \cite{KT17}. In this formulation, the area $\mcH^n (\p E\cap \Omega) $ and the wetting energy $\int_{\p \Omega \cap\p E}\cos\beta (x) d\mcH^n $ are characterized 
in 
terms of  two different varifolds $V$ and $\cos\beta W$, 
and the corresponding free energy is given by the mass of the varifold $V-\cos\beta W$.
By dealing with the varifold $V-\cos\beta W$, one is in fact studying a free boundary problem (see \cite{DeMP21,LZZ21}), which has a significant advantage in the study of the co-dimension one case.
However,  for the higher co-dimensional case, there is no suitable wetting energy as mentioned above.

In this paper, we introduce a new Neumann-type boundary condition for varifolds, which gives a non-variational characterization of capillarity (prescribed mean curvature in the interior and prescribed contact angle on the boundary) for hypersurfaces.
In particular, this generalizes easily to submanifolds of arbitrary dimensions and co-dimensions, in contrast to the geometric variational characterization.
We recall that after the celebrated work \cite{Allard72}, Allard studied the boundary behavior of varifolds of arbitrary dimensions and co-dimensions in \cite{Allard75} (see also Section \ref{Sec-2-Allard-vfld-bdry} below), which can be viewed as a Dirichlet-type boundary condition for varifolds.

We first define the following subspace of the trivial Grassmannian bundle over the supporting hypersurface, which provides all possible choices of tangent planes at boundary points of any regular capillary submanifolds.

\begin{definition}[Capillary bundle]\label{Defn:G_m,beta(S)}
\normalfont
For a possibly unbounded domain $\Om\subset\mfR^{n+1} (n\geq2)$ of class $C^1$ with boundary $S\coloneqq\p\Om$, denote by $\nu^S$ the inwards pointing unit normal field along $S$.
Let $\beta:S\to (0,\pi)$ be a $C^1$-function on $S$, which will be used to prescribe the contact angle at $x\in S$. 
For any $m\in\mbN$ with $m\leq n$ and any $x\in S$, we define $G_{m,\beta}(x)$ to be the collection of $m$-planes $P\in G(m,n+1)$ satisfying
\begin{enumerate}
    \item [($i$)]
\eq{\label{eq:P^m(nu^S)}
    \abs{P(\nu^S(x))}=\sin\beta(x),
}
 \item [($ii$)]$\left(P\cap T_xS\right)\perp P(\nu^S(x))$,
    \end{enumerate}
where $P(v)$ denotes the orthogonal projection of a vector $v$ onto the plane $P$.

By collecting all such $G_{m,\beta}(x)$ for $x\in S$, we obtain a subspace of $G_m(S)=S\times G(m,n+1)$, denoted by $G_{m,\beta}(S)$, endowed with the subspace topology, and  call it \textit{capillary bundle}.
\end{definition}

For any $(x,P)\in G_{m,\beta}(S)$ we define a unit vector field
\eq{\label{defn:n(x,P)}
\mfn(x,P)
\coloneqq\frac{P(\nu^S(x))}{\abs{P(\nu^S(x))}},
}
which is well-defined thanks to $\sin\beta>0$ and in fact $\mfn\in C^1(G_{m,\beta}(S),\mfR^{n+1})$.
It is clear by definition
\eq{\label{eq:n(x,Q),nu^S}
\left<\mfn(x,P),\nu^S(x)\right>
=\sin\beta(x).
}
Condition ($ii$) can be also written as
$P=(P\cap T_xS)\oplus\mfn(x,P)$.

Our Neumann boundary condition is
inspired by the definition of classical \textit{varifolds}.
Precisely, since the trivial Grassmannian bundle provides all possible choices of the tangent planes for any regular submanifolds, Radon measures on the trivial Grassmannian bundle are then naturally the measure-theoretic submanifolds, which gives us the nowadays well-known notion, \textit{varifolds}.
In the same spirit we define the measure-theoretic boundary of any capillary submanifolds to be Radon measures supported on the capillary bundle as follows.
Note however, to state it we unavoidably need to first introduce some necessary terminologies from Geometric Measure Theory, which we opt to present in Section \ref{Sec:2}.

\begin{definition}[Varifolds with capillary boundary]
\label{Defn:vfld-prescribed-bdry}
\normalfont
Let $\Om,\beta$ be given as in Definition \ref{Defn:G_m,beta(S)}.
Let $V$ be an $m$-varifold on $\overline\Om$ and  $\Gamma$  a Radon measure on $G_{m,\beta}(S)$.
We say that \textit{$V$ has a prescribed contact angle $\beta$ intersecting $S$ along $\Gamma$} if for all $\varphi\in\mathfrak{X}_t(\Om)$ with compact support, there exists a \textit{generalized mean curvature vector} $\mfH\in L^1_{loc}(\mu_V)$, with $\mfH(x)\in\mfR^{n+1}$ and $\mfH(x)\in T_xS$ for $\mu_V$-a.e. $x\in S$, such that
\eq{\label{defn:1st-vairation-intro}
\de V(\varphi)
=-\int_{\overline\Om}\left<\mfH(x),\varphi(x)\right>\rd\mu_V(x)-\int_{G_{m,\beta}(S)}\left<\mfn(x,P),\varphi(x)\right>\rd\Gamma(x,P),
}
where $G_{m,\beta}(S)$ and the corresponding $\mfn(x,P)$ are defined  in Definition \ref{Defn:G_m,beta(S)}.
We call such $V$ a \textit{varifold with capillary boundary (or capillary varifold)} and the corresponding $\Gamma$ the \textit{capillary boundary varifold with respect to $V$ and $\beta$}, or in short, the boundary varifold of $V$.
The class of all such $V$ is denoted by ${\bf V}^m_\beta(\Om)$.
\end{definition}

This idea is not only suitable for defining measure-theoretic capillary boundary in the context of varifolds, but also useful for curvature varifolds and Brakke flow, see below.
For the case $\beta\equiv\frac \pi 2$, \textit{curvature varifolds with orthogonal boundary} was recently introduced
by Kuwert and M\"uller (\cite{KM22}), which partly motivates our work.

To have a closer look at Definition \ref{Defn:vfld-prescribed-bdry},
we use standard disintegration to write $\Gamma=\sigma_\Gamma\otimes\Gamma^x$, where $\sigma_\Gamma$ is called the \textit{generalized boundary measure} of $V$ and $\Gamma^x$ is a Radon probability measure on $G_{m,\beta}(x)$ for $\sigma_\Gamma$-a.e. $x$.
At the points where $\Gamma^x$ is well-defined
we define a \textit{generalized inwards pointing co-normal} to $V$  artificially to be the average
\eq{\label{defn:n_V-intro}
\mfn_V(x)
\coloneqq\int_{G_{m,\beta}(x)}\mfn(x,P)\rd\Gamma^x(P).
}
The vector field $\mfn_V$ is approximately continuous with respect to $\sigma_\Gamma$ at $\sigma_\Gamma$-a.e. point, and $\norm{\mfn_V}_{L^\infty(\sigma_\Gamma)}\leq1$ but in general not unit,
which will be discussed  in Section \ref{Sec:3}. Using \eqref{defn:n_V-intro}, one can express \eqref{defn:1st-vairation-intro} as 
\eq{\label{defn:1st-vairation-intro2}
\de V(\varphi)
=-\int_{\overline\Om}\left<\mfH(x),\varphi(x)\right>\rd\mu_V(x)-\int_{S}\left<\mfn_V(x),\varphi(x)\right>\rd\sigma_\Gamma(x).
}

For $\beta\equiv\frac\pi2$, it is easy to check that the last integral is nothing but zero, since in this case $\mfn_V$ defined by \eqref{defn:n_V-intro} is just $\nu^S$, thanks to \eqref{eq:n(x,Q),nu^S}.
Therefore our notion agrees with the classical notion of \textit{free boundary varifolds}.
In this case, $\sigma_\Gamma$ is not a true boundary measure, since it plays the same role as $0$ measure.
For $\beta\neq\frac\pi2$, 
it is easy to check that any smooth submanifold in $\overline\Om$ that intersects $S$ at angle $\beta$ satisfies Definition \ref{Defn:vfld-prescribed-bdry} (see Example \ref{exam:submflds}). In this case the last integral is non-vanishing.
However, note that in the non-smooth case the last integral could again be vanishing, and the capillary varifold in the sense of Definition \ref{Defn:vfld-prescribed-bdry} then reduces to a classical free boundary varifold.
In this case $\sigma_\Gamma$ is again not a true boundary measure.
A nice example indicating the  fact is provided in Example \ref{exam:planes}, where two half-planes with the same capillary angle meet along a common boundary from opposite directions so that the tangential parts of the two planes that contribute to the capillary structure cancel with each other.
This is therefore called \textit{degenerate capillary phenomenon}, and we point out that this also occurs in the geometric variational formulation, because the geometric variational formulation of capillary varifold, which we record in Definition \ref{Defn:vfld-co-dim-1} below, is also satisfied by the example provided above (with $V,U$ therein chosen as the union of the two planes and the empty set/supporting hypersurface, respectively).

\subsection{Main results}

\subsubsection{Boundary rectifiability}
Our first main result concerns boundary rectifiability of the ``at most $(m-1)$-dimensional'' part of the generalized boundary measure $\sigma_\Gamma$, defined as the restriction of $\sigma_\Gamma$ to the points with strictly positive lower $(m-1)$-density,
that is,
\eq{\label{defn:E_ast}
\sigma_{\ast\Gamma}
=\sigma_\Gamma\llcorner E_\ast,\text{ where }
E_\ast\coloneqq\{x:0<\Theta^{m-1}_\ast(\sigma_\Gamma,x)\}.
}

We will show that the further restriction of $\sigma_{\ast\Gamma}$ to a specific subset of $S$
is $(m-1)$-rectifiable.
Before that we recall, the boundary rectifiability has been established for free boundary varifolds by De Masi \cite{DeMasi21} in light of the blow-up analysis performed in \cite{DePDeRG18}; and for the so-called $(\beta,S)$-capillary varifolds, which are in fact free boundary varifolds that contain capillary information, by De Masi-Edelen-Gasparetto-Li in \cite{DEGL24}.

To illustrate the main differences between these results and 
our aim, we need Proposition \ref{Prop:1st-variation}, in which we prove that any $V\in{\bf V}^m_\beta(\Om)$ is of bounded first variation and hence we could write the first variation formula of $V$ for not only tangential variation with respect to the supporting hypersurface but also for normal variation.
By virtue of this we obtain in \eqref{eq:1st-variation-V_beta} the \textit{tangential part} of the boundary measure which is exactly $\sigma_\Gamma$, and also the \textit{normal part} of the boundary measure, which we denote by $\sigma_V^\perp$.
In both \cite{DeMasi21} and \cite{DEGL24}, the varifolds under consideration are shown to have bounded first variation,
and due to their free boundary essence, the resulting boundary measure is just like our \textit{normal part} $\sigma_V^\perp$.
In view of this the difference is thus clear, 
namely, our goal is to show rectifiability of \textit{tangential part} of the boundary measure.

In fact, the specific subset of $S$ mentioned above is defined as follows.

\begin{definition}[Capillary boundary point]\label{Defn:capillary-bdry-point}
\normalfont
Let $V\in{\bf V}^m_\beta(\Om)$ with boundary varifold $\Gamma$.
Any point
$x_0\in{\rm spt}\sigma_\Gamma\cap\{x\in S:\cos\beta(x)\neq0\}$ is called a \textit{capillary boundary point} of $V$ if there exist a unit vector $\tau^V_{x_0}\in T_{x_0}S$, $c_{x_0}>0$ and $\ep_{x_0},\rho_{x_0}\in(0,1]$, such that
\begin{enumerate}
    \item [(C1)] (local non-degeneracy) For $\sigma_\Gamma$-a.e. $x\in B_{\rho_{x_0}}(x_0)$ there holds
\eq{\label{defn:genera-capi-bdry-point-1}
\left<\cos\beta(x)\mfn_W(x),\tau^V_{x_0}\right>\geq\ep_{x_0},
}
where $\mfn_W(x)\coloneqq \frac{T_{x}S(\mfn_V(x))}{\cos\beta(x)}$ is the generalized inwards pointing co-normal of $V$ with respect to $S$.
    \item [(C2)] For $\sigma_\Gamma$-a.e. $x\in B_{\rho_{x_0}}(x_0)$ there holds
\eq{\label{defn:genera-capi-bdry-point-2}
\Abs{\left<{x-x_0},\cos\beta(x)\mfn_W(x)\right>}
\leq c_{x_0}\abs{\cos\beta(x_0)}\abs{x-x_0}^2.
}
\end{enumerate}
We denote by $\mathscr{P}_{\rm cb}(\Gamma)$ the collection of all capillary boundary points, and a shorthand $\mathscr{P}_{\rm cb}$ will be used when there is no ambiguity.
\end{definition}
Our boundary rectifiability result then states as follows.
\begin{theorem}[Boundary rectifiability]\label{Thm:bdry-rectifiability}
Let $\Om\subset\mfR^{n+1}$ be a bounded domain of class $C^2$ and $\beta\in C^1(S,(0,\pi))$.
Let $V\in{\bf V}^m_\beta(\Om)$ with boundary varifold $\Gamma$, such that $\mfH\in L^p(\mu_V)$ for some $p\in(m,\infty)$, and $\Theta^m(\mu_V,x)\geq a>0$ for $\mu_V$-a.e. $x$.
Then for any $\sigma_\Gamma$-measurable set $A\subseteq \mathscr{P}_{\rm cb}(\Gamma)$, $\sigma_{\ast\Gamma}\llcorner A$ is $(m-1)$-rectifiable.
\end{theorem}

Let us justify the necessity of $\mathscr{P}_{\rm cb}(\Gamma)$ in Theorem \ref{Thm:bdry-rectifiability}:
First of all, as discussed after Definition \ref{Defn:vfld-prescribed-bdry}, regardless of $\beta=\frac\pi2$ or not, the capillary varifold $V$ might be in fact a free boundary varifold, in which case it makes no sense to talk about $\sigma_\Gamma$, as it plays the same role as $0$ measure in the variational structure \eqref{defn:1st-vairation-intro2}.
For this reason,
we need to exclude the set $\{x\in S:\beta(x)=\frac\pi2\}$, and also assume
Condition (C1), which infers that at the considered points $\sigma_\Gamma$ is a true boundary measure.
Second, our strategy of proof is to carry out blow-up analysis and then apply the Marstrand-Mattila rectifiability criterion as in \cite{DePDeRG18,DeMasi21}.
To be able to do so we assume Condition (C2), which requires that locally $\mfn_W$ (and hence $\mfn_V$) stays perpendicular to the generalized boundary ${\rm spt}\sigma_\Gamma$, just like \textit{co-normals} of regular submanifolds.
(C2) is essentially used to prove a monotonicity identity at boundary points (Proposition \ref{Prop:Mono-bdry}), and to show \eqref{eq:x-n_W(x_0)}, which is crucial for showing that the $(m-1)$-blow-up measure of $\sigma_\Gamma$ takes the form $\theta\mcH^{m-1}\llcorner L^{m-1}$ for some $\theta\in\mfR_+$ and $(m-1)$-dimensional linear space $L^{m-1}$, see Lemma \ref{Lem:sigma_Gamma_infty-characterization}, by virtue of which we could use the Marstrand-Mattila rectifiability criterion as said.
More discussions on Definition \ref{Defn:capillary-bdry-point} can be found in Section \ref{Sec:3} below.

Note that Theorem \ref{Thm:bdry-rectifiability} only partially describes the boundary varifold $\Gamma$.
For the disintegration $\Gamma=\sigma_\Gamma\otimes\Gamma^x$ we have not much information about $\Gamma^x$. The only thing we know, see Corollary \ref{Cor-n_V-perp-app.tangentspace}, is that the generalized inwards pointing co-normal to $V$ is perpendicular to the approximate tangent space of the $(m-1)$-rectifiable measure $\sigma_{\ast\Gamma}\llcorner A$.
This inspires us to consider the following subclass of varifolds of ${\bf V}^m_\beta(\Om)$.
\begin{definition}[Varifolds with rectifiable capillary boundary]\label{Defn:rectifiable-bdry}
\normalfont
Under the assumptions in Definition \ref{Defn:vfld-prescribed-bdry}, if the disintegration of $\Gamma$ takes the form:
\eq{
\Gamma
=\sigma_\Gamma\otimes\Gamma^x
=\theta\mcH^{m-1}\llcorner M\otimes\de_{P_\Gamma(x)},
}
where $\sigma_\Gamma=\theta\mcH^{m-1}\llcorner M$ is an $(m-1)$-rectifiable measure, $P_\Gamma$ is a map
with $P_\Gamma(x)\in G_{m,\beta}(x)$ for $\sigma_\Gamma$-a.e. $x$, such that
\eq{
P_\Gamma(x)\cap T_xS=T_xM,
}
then we call $\Gamma$ \textit{rectifiable boundary varifold} with respect to $V,\beta$.
We denote the class of all $V\in{\bf V}^m_\beta(\Om)\cap{\bf RV}^m(\overline\Om)$
with rectifiable boundary varifold $\Gamma$ by ${\bf RV}^m_\beta(\Om)$.

Of particular interest is the case when $\theta\equiv1$, and in this case we call $\Gamma$ \textit{multiplicity one rectifiable boundary varifold} with respect to $V,\beta$.
\end{definition}

\subsection{Classification of tangent cones}

As mentioned above, we study the blow-up process at capillary boundary points to prove Theorem \ref{Thm:bdry-rectifiability}.
A by-product of independent interest is the following classification of tangent cones at the capillary boundary points.
\begin{theorem}[Classification of tangent cones]\label{Prop:Classify-VarTan}
Let $\Om\subset\mfR^{n+1}$ be a bounded domain of class $C^2$ and $\beta\in C^1(S,(0,\pi))$.
Let
$V\in{\bf V}^m_\beta(\Om)\cap{\bf IV}^m(\overline\Om)$ with boundary varifold $\Gamma$, such that $\mfH\in L^p(\mu_V)$ for some $p\in(m,\infty)$.
Then for any $\sigma_\Gamma$-measurable set $A\subseteq \mathscr{P}_{\rm cb}(\Gamma)$,
there exists a set $F\subset S$ with
\eq{
\sigma_\Gamma(A\setminus F)=0,
}
such that
for every $x_0\in A\cap F\subset \mathscr{P}_{\rm cb}\cap F$ and for any $\mcC\in{\rm VarTan}(V,x_0)$, if
\eq{
\Theta^{m-1}_\ast(\sigma_\Gamma,x_0)>0,
}
and for some $0<\ep<\frac12$
\eq{
\Theta^m(\mu_V,x_0)\leq\frac12+\ep,
}
then, $\mcC$ is the induced varifold of some half-$m$-plane that lies in $T_{x_0}\overline\Om$.

If in addition $V\in{\bf RV}^m_\beta(\overline\Om)$ with multiplicity one rectifiable boundary varifold $\Gamma$ in the sense of Definition \ref{Defn:rectifiable-bdry} then $\mcC$ is the induced varifold of
$P_\Gamma(x_0)\cap T_{x_0}\overline\Om$.

\end{theorem}

Recall that a similar result is obtained in \cite[Theorem 3.16]{DEGL24}, under different geometric assumptions.
These results can be viewed as initial steps towards the Allard-type boundary regularity theorem.

\subsubsection{Integral compactness}
Our second main result concerns Allard-type compactness.
\begin{theorem}[Integral compactness]\label{Thm:compactness}
Let $\Om\subset\mfR^{n+1}$ be a bounded domain of class $C^2$, $\beta\in C^1(S,(0,\pi))$, and $p\in(1,\infty)$.
Suppose $\{V_j\}_{j\in\mbN}$ is a sequence of rectifiable $m$-varifolds such that $V_j\in{\bf V}^m_\beta(\Om)$ with boundary varifolds $\{\Gamma_j\}_{j\in\mbN}$.
If for each $j$,
\begin{enumerate}
    \item There exists a universal constant $a>0$ such that
\eq{
\Theta^m(\mu_{V_j},x)\geq a>0,\text{ }\mu_{V_j}\text{-a.e. }x.
}
    \item 
    $\mfH_j\in L^p(\mu_{V_j})$ and
\eq{\label{condi:uniform-bdd-H-mass}
\sup_{j\in\mbN}\left(\mu_{V_j}(\overline\Om)+\norm{\mfH_j}_{L^p(\mu_{V_j})}\right)<\infty.
}
    \item 
    There exist a dense set $E_j$ in ${\rm spt}\sigma_{\Gamma_j}$, universal constants $\ep_0,\rho_0>0$, such that Definition \ref{Defn:capillary-bdry-point} (C1) is satisfied by every $x\in E_j$ with corresponding unit vector $\tau^{V_j}_x\in T_{x}S$ locally in $B_{\rho_0}(x)$.

\end{enumerate}
Then after passing to a subsequence, there exists $V\in{\bf V}_\beta^m(\Om)\cap{\bf RV}^m(\overline\Om)$  with boundary varifold $\Gamma$, such that $V_{j_k}\ra V$ as varifolds, $\de V_{j_k}\ra\de V$, and $\Gamma_{j_k}\wsc\Gamma$ as Radon measures,
with
\begin{enumerate}
    \item [(i)] $\Theta^m(\mu_V,x)\geq a>0$, $\mu_V$-a.e. $x$.
    \item [(ii)] If $\{V_j\}_{j\in\mbN}\in{\bf V}^m_\beta(\Om)\cap{\bf IV}^m(\overline\Om)$, then $V\in{\bf V}^m_\beta(\Om)\cap{\bf IV}^m(\overline\Om)$.    
\end{enumerate}
In particular, for any $x_0\in S$, if there exist $a_{x_0},\ep_{x_0},\rho_{x_0}>0$ and a unit vector $\tau_{x_0}\in T_{x_0}S$, such that for each $j$, there hold
\begin{enumerate}
    \item [(4)] $\sigma_{\Gamma_j}(B_{\frac{\rho_{x_0}}2}(x_0))\geq a_{x_0}$.
    \item [(5)] $\left<\cos\beta(x)\mfn_{W_j}(x),\tau_{x_0}\right>\geq\ep_{x_0}$, $\sigma_{\Gamma_j}$-a.e. $x\in B_\rho(x_0)$.
    Namely, \eqref{defn:genera-capi-bdry-point-1} is satisfied uniformly.
\end{enumerate}
Then
\begin{enumerate}
    \item [(iii)] $\sigma_{\Gamma}( B_{\frac{\rho_{x_0}}2}(x_0))\geq a_{x_0}$.
    \item [(iv)] $\mfn_V$ is non-degenerate locally around $x_0$ in the sense that
\eq{
0<a_{x_0}\ep_{x_0}
\leq\int_{S\cap B_{\rho_{x_0}}(x_0)}\Abs{T_xS\left(\mfn_V(x)\right)}\rd\sigma_\Gamma(x).
}
\end{enumerate}
\end{theorem}
The convergence follows from the bounded first variation of $V\in{\bf V}^m_\beta(\Om)$ (Proposition \ref{Prop:1st-variation}) and Allard's integral compactness \cite{Allard72}, provided conditions ($1$)-($3$).
Note that ($3$) is trivially satisfied if $\Gamma_j=0$. Hence,
to study the local property of the boundary measure in the limit,
 we require condition ($4$), the local uniform lower bound of the boundary measures in the sequence, to guarantee the local non-degeneracy of the boundary measure in the limit, which is conclusion ($iii$).
Moreover, we further require condition ($5$), the local uniform non-degeneracy of the generalized inwards pointing co-normals with respect to a fixed tangential direction in the sequence, to exclude the local capillary degenerate phenomenon in the limit, which is conclusion ($iv$).

At first glance condition ($5$) seems to be a technical assumption that could be redundant for an optimal compactness result, as condition ($3$) already provides a local uniform non-degeneracy of the tangential part of $\mfn_{V_j}$.
To justify the necessity of condition ($5$), we revisit Example \ref{exam:planes} (for detailed discussions, see Example \ref{exam:planes-limit}), where we construct a union of two half-planes with the same capillary angle meeting along a common boundary from opposite directions that satisfies Definition \ref{Defn:vfld-prescribed-bdry}.
If we separate these two half-planes along their unit co-normals with respect to the supporting hyperplane respectively, then each separated half-plane is a smooth submanifold that satisfies Definition \ref{Defn:vfld-prescribed-bdry}, and of course condition ($3$) is satisfied by each of the separated half-planes.
It is easy to check that condition ($5$) is satisfied by at most one of these two separated half-planes.
Therefore, reversing this process we obtain Example \ref{exam:planes} as a convergence limit, which clearly violates conclusion ($iv$).
%


\subsection{Comparison with the geometric variational formulation}\label{Sec:1-3}

We first record the aforementioned geometric variational weak formulation of capillarity:
\begin{definition}[Geometric variational formulation]\label{Defn:vfld-co-dim-1}
\normalfont
Let $\Om\subset\mfR^{n+1}$ be a bounded domain of class $C^2$, $\beta\in C^1(S,(0,\pi))$.
Let $V$ be a rectifiable $n$-varifold on $\overline\Om$, and $U$ a $\mcH^n$-measurable set in $S$.
We say that \textit{$V$ has prescribed contact angle $\beta$ with $U$} if for any compactly supported $\varphi\in\mathfrak{X}_t(\Om)$, there exists a \textit{generalized mean curvature vector} $\mfH\in L^1(\mu_V)$ with $\mfH(x)\in\mfR^{n+1}$, and $\mfH(x)\in T_xS$ for $\mu_V$-a.e. $x\in S$, such that
\eq{\label{defn:1st-variation-KT17}
\de V(\varphi)-\int_U{\rm div}_S(\cos\beta\varphi)\rd\mcH^n
=-\int_{\overline\Om}\left<\mfH,\varphi\right>\rd\mu_V,
}
where ${\rm div}_S$ denotes the tangential divergence with respect to $S$.
\end{definition}
Here the set $U$ models for the \textit{wetting region} and the corresponding integral models for the first variation of the wetting energy.
Comparing \eqref{defn:1st-variation-KT17} with \eqref{defn:1st-vairation-intro}, we find that
$V$ in Definition \ref{Defn:vfld-co-dim-1} also satisfies Definition \ref{Defn:rectifiable-bdry} (a stronger form of Definition \ref{Defn:vfld-prescribed-bdry}), provided that the boundary of $U$ is regular enough (e.g., when $U$ is a set of finite perimeter).
Precisely, an easy application of De Giorgi's structure theorem shows that the integrals
\eq{
\int_U{\rm div}_S(\cos\beta\varphi)\rd\mcH^n,\quad\int_{G_{n,\beta}(S)}\left<\mfn(x,P),\varphi(x)\right>\rd\Gamma(x,P)
}
are equivalent for some suitably defined $\Gamma$ (see Proposition \ref{Prop:RV-codim-1}).
The latter integral
plays the role as the boundary term appearing in the first variation of a smooth capillary hypersurface.
For the former integral, note that  under specific situations $U$ can be shown to be a set of finite perimeter, see \cite{DEGL24}, but
in general this is not known.
However, without assuming embeddedness, the boundary of a capillary hypersurface needs not to  arise as the boundary of a wetting region (for example, capillary immersions of orientable surfaces into the Euclidean half-space, see \cite{WXZ24}).
In other words,
Definition \ref{Defn:vfld-prescribed-bdry} provides a direct way to characterize the capillary
boundary.

One nice application concerning this fact is that we can use our Neumann-type boundary condition to study curvature varifolds with capillary boundary.
Another nice application is to study evolution of surfaces in the sense of Brakke \cite{Brakke78}.
Both of these applications are new in the co-dimension-$1$ case, and can be carried out in any dimensions and co-dimensions.
%

\subsubsection{Curvature varifolds with capillary boundary}
In a pioneering work of Hutchinson \cite{Hutchinson86}, varifolds are equipped with second fundamental form in the measure-theoretic sense.
Such varifolds are called \textit{curvature varifolds}, and the idea of introducing \textit{weak second fundamental form} is in particular useful for studying geometric variational problems.
Later on, this idea is used by Mantegazza \cite{Mantegazza96} to define \textit{curvature varifolds with boundary}.

Recently,
Kuwert-M\"uller \cite{KM22} define the so-called \textit{curvature varifolds with orthogonal boundary} and use this notion to study the minimization problem of $L^p$-total curvature energy.
Bellettini studies weak second fundamental form in the context of \textit{oriented integral varifolds} and provides an alternative approach to the \textit{prescribed mean curvature problem} proposed and addressed by Sch\"atzle \cite{Schatzle01}.

Our contribution in this direction is built on the following definition.
\begin{definition}[Curvature varifolds with capillary boundary]\label{Defn:CV-capillary}
\normalfont
Let $\Om\subset\mfR^{n+1}$ be a bounded domain of class $C^2$, $\beta\in C^1(S, (0,\pi))$.
Let $V\in{\bf V}^m(\overline\Om)$, $B\in L^1(V)$ where $B(x,P)\in{BL}(\mfR^{n+1}\times\mfR^{n+1},\mfR^{n+1})$, and $\Gamma$  a Radon measure on $G_{m,\beta}(S)$.
We say that \textit{$V$ has weak second fundamental form $B$ and prescribed contact angle $\beta$ with $S$ along $\Gamma$} if for any test function $\phi\in C^1(\mfR^{n+1}\times\mfR^{(n+1)^2},\mfR^{n+1})$, there holds
\eq{\label{defn:CV-capillary}
\int
\left(D_P\phi\cdot B+\left<{\rm tr}B,\phi\right>+\left<\na_x\phi,P\right>\right)\rd V(x,P)
=-\int_{G_{m,\beta}(S)}\left< 
\mfn(x,P),\phi(x,P)\right>\rd\Gamma(x,P),
}
where $\mfn(x,P)$ is defined as \eqref{defn:n(x,P)} and $D_p\phi\cdot B$ is defined as \eqref{eq:D_Pphi,B}.
We call such $V$ a \textit{curvature varifold with capillary boundary} and the corresponding $\Gamma$ the \textit{boundary varifold} (with respect to $V,\beta$).
The class of all such curvature varifolds $V$ is denoted by ${\bf CV}^m_{\beta}(\Om)$.

\end{definition}
This is a stronger definition compared to Definition \ref{Defn:vfld-prescribed-bdry},
and in this case the degenerate capillary phenomenon does not occur since we are using test functions which also depend on the choice of $m$-planes.
In fact, we will show that our definition is compatible with Mantegazza's curvature varifolds with boundary \cite{Mantegazza96}, and the boundary varifold (in the sense of \cite{Mantegazza96}) is the vector-valued Radon measure $\p V(x,P)\coloneqq\mfn(x,P)\Gamma(x,P)$.
As observed by Mantegazza, $\p V$ carries much more information on the local structure of $V$, while the projection of $\p V$ to $\mfR^{n+1}$ could be even equal to $0$.

We collect the nice properties of curvature varifolds with capillary boundary in Section \ref{Sec:8}.
Compactness results for curvature varifolds are crucial to be exploited for geometric variational problems, and have been discussed intensively in \cite{Hutchinson86,Mantegazza96,Bellettini13,KM22}, in our case we have the following:

\begin{theorem}[Compactness for curvature varifolds]\label{Thm:compactness-CV}
Let $\Om\subset\mfR^{n+1}$ be a bounded domain of class $C^2$, $\beta\in C^1(S,(0,\pi))$.
Suppose $\{V_j\}_{j\in\mbN}$ is a sequence of (integral) $m$-varifolds such that $V_j\in{\bf CV}^m_\beta(\Om)$ with weak second fundamental forms $\{B_j\}_{j\in\mbN}$ and boundary varifolds $\{\Gamma_j\}_{j\in\mbN}$ in the sense of Definition \ref{Defn:CV-capillary}.
If
\begin{enumerate}
    \item $\mu_{V_j}(\overline\Om)\leq C$;
    \item For some fixed $p\in(1,\infty]$,
\eq{
\sup_{j\in\mbN}\{\norm{B_j}_{L^p(V_j)}^p\}\leq\Lambda<\infty.
}
\end{enumerate}
Then after passing to a subsequence, one has $V_{j_k}\ra V$ as varifolds, $\Gamma_{j_k}\wsc\Gamma$ as Radon measures.
Moreover, the limiting (integral) varifold and Radon measure $V,\Gamma$ satisfy Definition \ref{Defn:CV-capillary} with the same $\beta$ for some weak curvature $B\in L^p(V)$ such that
\eq{
\norm{B}_{L^p(V)}^p
\leq\liminf_{k\ra\infty}\norm{B_{j_k}}^p_{L^p(V_{j_k})}\leq\Lambda.
}
In fact, for every convex and lower semicontinuous function $f:\mfR^{(n+1)^3}\ra[0,+\infty]$, there holds
\eq{\label{ineq:liminf-convex-curvature}
\int f(B)\rd V
\leq\liminf_{k\ra\infty}\int f(B_{j_k})\rd V_{j_k}.
}
\end{theorem}
Theorem \ref{Thm:compactness-CV} could serve as the first step (initial compactness) for studying capillary geometric variational problems, which might have some further applications.
Here we also mention another application of our Neumann boundary condition in the context of curvature varifolds.
Following the ideas of Bellettini \cite{Bellettini23}, the second author studies and exploits the integral oriented curvature varifolds with capillary boundary to solve the capillary prescribed mean curvature problem in \cite{Zhang}.
\subsubsection{Brakke flow with capillary boundary}
The study of \textit{Brakke flow} and the classical \textit{mean curvature flow} has been an important topic in Geometric Analysis since the 1970's as it is one of the fundamental geometric evolution problems, for an overview of the history see e.g., \cite{Bellettini13,CM11,Ecker04,Giga06,Mantegazza11,Tonegawa19} and the references therein.

Recently, White \cite{White21} developed a theory of (hyper)surfaces with boundary moving by Brakke flow, in his work the boundary condition under consideration is in fact a Dirichlet-type boundary condition, since each flow hypersurface is associated with a rectifiable varifold satisfying Allard's definition \cite{Allard75}.
An $\ep$-regularity theorem is established in this framework by Gasparetto \cite{Gasparetto22}.
On the other hand, a notion of \textit{Brakke ﬂow with free boundary} (Neumann-type boundary condition) was originally written down by Mizuno and Tonegawa \cite{MT15}, who proved existence of co-dimension-1 free boundary Brakke ﬂows in convex barriers via the Allen–Cahn functional, see also \cite{Kagaya19}, and then further developed by Edelen \cite{Edelen20}, who established existence, regularity, as well as compactness for free boundary Brakke flow of arbitrary dimensions and co-dimensions.
In terms of capillarity,
a geometric variational characterization of co-dimension-$1$ {Brakke flow with capillary boundary} was introduced by Tashiro in \cite{Tashiro24}.

In light of Definition \ref{Defn:vfld-prescribed-bdry}, we define Brakke flow with capillary boundary for any dimensions and co-dimensions as follows:
\begin{definition}
\normalfont
Let $I\subset\mfR$ be some interval.
Given a bounded domain $\Om\subset\mfR^{n+1}$ of class $C^1$ with boundary $S\coloneqq\p\Om$, and a function $\beta\in C^1(S,(0,\pi))$.
We say a collection $\{\mu(t)\}_{t\in I}$ of Radon measures supported on $\overline\Om$ is an $m$-dimensional \textit{Brakke flow with capillary boundary} prescribed by $\beta$ in $\Om$ if the following holds:
\begin{enumerate}
    \item [($i$)] For a.e. $t\in I$, $\mu(t)=\mu_{V(t)}$ for some $V(t)\in{\bf V}^m_\beta(\Om)\cap{\bf IV}^m(\overline\Om)$ with square integrable generalized mean curvature $\mfH_{V(t)}$ in the sense of Definition \ref{Defn:vfld-prescribed-bdry}.
    \item [($ii$)] For any finite interval $(a,b)\subset I$, the Brakke’s inequality (see e.g., \cite[(4.1)]{Edelen20}) holds.
    Namely, for any non-negative test function $\phi\in C^1(\overline\Om\times[a,b],[0,\infty))$ with $\na_x\phi(\cdot,t)$ tangent to $S$ for all $t\in[a,b]$,
\eq{
\int\phi(\cdot,b)\rd\mu(b)
-\int\phi(\cdot,a)\rd\mu(a)
\leq\int_a^b\int-\Abs{\mfH}^2\phi+\left<\mfH,\na_x\phi\right>+\p_t\phi\rd\mu(t)\rd t. 
}
\end{enumerate}
\end{definition}
One can check that any classical mean curvature flow $(\S_t^m)_t$ with capillary boundary satisfies the above definition by taking $\mu(t)=\mcH^m\llcorner\S^m_t$.
In the co-dimension-$1$ case, we also refer to \cite{BK18,EG24,Kholmatov24,HL24,Tashiro24} for weak solutions to mean curvature flow with capillary boundary from geometric variational perspective.

%

\subsubsection{Allard-type regularity}
As mentioned above, the Allard-type boundary regularity for capillary hypersurfaces has been studied using geometric variational formulation \cite{DEGL24,Wang24}.
Let us now comment on how our Neumann-type boundary condition can be applied in this case.

In fact, as already mentioned at the beginning of Section \ref{Sec:1-3}, we will show in Proposition \ref{Prop:RV-codim-1} that, if $V$ has prescribed contact angle $\beta$ with a set of finite perimeter $U\subset S$ in the sense of Definition \ref{Defn:vfld-co-dim-1}, then $V\in{\bf RV}^n_\beta(\Om)$ with multiplicity one rectifiable boundary in the sense of Definition \ref{Defn:rectifiable-bdry}.
Therefore, we could apply Theorem \ref{Prop:Classify-VarTan} to replace the crucial \textit{multiplicity-one tangent half-plane} condition in Wang's Allard-type regularity theorem \cite{Wang24} by the geometric assumptions as follows.

\begin{theorem}\label{Thm:Allard}
Let $\Om\subset\mfR^{n+1}$ be a bounded domain of class $C^2$, $\beta\in C^1(S,(0,\pi))$.
Let
$V$ be an integral $n$-varifold on $\overline\Om$ which has bounded generalized mean curvature and prescribed contact angle $\beta$ with a set of finite perimeter $U\subset S$ in the sense of Definition \ref{Defn:vfld-co-dim-1}.
Then
\begin{enumerate}
    \item [(i)] $V\in{\bf RV}^n_\beta(\Om)\cap{\bf IV}^n(\overline\Om)$ with multiplicity one rectifiable boundary $\Gamma$ in the sense of Definition \ref{Defn:rectifiable-bdry}, where $\Gamma=\mcH^{n-1}\llcorner\p^\ast U\otimes\de_{P_U(x)}$ with
\eq{
P_U(x)
\coloneqq T_x\p^\ast U\oplus\left(\sin\beta(x)\nu^S(x)+\cos\beta(x)\nu_U(x)\right), \mcH^{n-1}\text{-a.e. }x\in\p^\ast U.
}
Here $\p^\ast U,$ and $\nu_U$ denote  the reduced boundary and measure-theoretic inner unit normal of $U$ respectively.
    \item [(ii)] For any $\sigma_\Gamma$-measurable set $A\subseteq \mathscr{P}_{\rm cb}(\Gamma)$, there exists a set $F\subset S$ with $\sigma_\Gamma(A\setminus F)=0$, such that for every $x_0\in A\cap F$,
    if
\eq{
\Theta^{m-1}_\ast(\sigma_\Gamma,x_0)>0,
}
and for some $0<\ep<\frac12$
    \eq{
    \Theta^n(\mu_V,x_0)\leq\frac12+\ep,
    }
    then ${\rm spt}\mu_V$ near $x_0$ is a $C^{1,\alpha}$-hypersurface with boundary for some $\alpha\in(0,1)$.
    Moreover, the tangent space of this regular hypersurface at $x_0$ meets $S$ with angle $\beta(x_0)$.
    %
\end{enumerate}
\end{theorem}

It is natural to ask, if one could establish an Allard-type boundary regularity theorem for varifolds with capillary boundary in the sense of Definition \ref{Defn:vfld-prescribed-bdry}.
In view of Theorem \ref{Prop:Classify-VarTan} we believe, with necessary modifications, the arguments in \cite{DEGL24} can be adapted to our framework to show the Allard-type boundary regularity theorem, at least for integral varifolds, in the co-dimension-$1$ case, under suitable assumption on density and that $\mfH\in{L^\infty(\mu_V)}$.

Since Allard-type regularity theorem has been shown for classical varifolds and free boundary varifolds under natural density assumptions and the mean curvature bounds $\norm{H}_{L^p(\mu_V)}$ for super-critical $p$ in \cite{Allard72,GJ86}.
It is natural to propose the following question:

\

{
\it 
For $V\in{\bf V}^m_\beta(\Om)$, of any dimensions and co-dimensions, with $\mfH\in L^p(\mu_V)$ for $p>m$,
under certain density assumptions,
is Allard-type boundary regularity result true on $\mathscr{P}_{\rm cb}(\Gamma)$?
}

\

\subsection{Brief outline of the paper}

Section \ref{Sec:2} is an introductory section about varifolds and basic facts which we need in this paper.

In Section \ref{Sec:3} we introduce the Neumann-type boundary condition of varifolds and prove several basic properties resulting 
from Definitions \ref{Defn:vfld-prescribed-bdry}, \ref{Defn:capillary-bdry-point}.
An important property is that any $V\in{\bf V}^m_\beta(\Om)$ has bounded first variation (Proposition \ref{Prop:1st-variation}).
Some examples are given for a better understanding of the definition.

Section \ref{Sec:4} is devoted to proving the boundary Monotonicity inequalities for $V\in{\bf V}^m_\beta(\Om)$ on $\mathscr{P}_{\rm cb}(\Gamma)$ when $p>m$ (Corollary \ref{Coro-monoto-p>m}).
Direct consequences are the existence of density and upper semi-continuity of density at capillary boundary points.

In Section \ref{Sec:5} we study the blow-up process of $V\in{\bf V}^m_\beta(\Om)$ on $\mathscr{P}_{\rm cb}(\Gamma)$.
Blow-up sequences are shown to possess very nice properties (Lemma \ref{Lem:blow-up-seq}).
In particular, the establishment of \eqref{eq:x-n_W(x_0)} is the key for our subsequent analysis of the blow-up process.
The first part of this section is devoted to proving that the blow-up measure of $\sigma_\Gamma$ is $(m-1)$-rectifiable (Lemma \ref{Lem:sigma_Gamma_infty-characterization}), which leads to the boundary rectifiability result.
Interesting by-products are the uniqueness of tangent cones (Theorem \ref{Prop:Classify-VarTan}), and the Allard-type boundary regularity (Theorem \ref{Thm:Allard}), as well as a boundary strong maximum principle (Proposition \ref{Prop:bdry-SMP}).

In Sections \ref{Sec:6} and \ref{Sec:7} we present the proof of Theorems \ref{Thm:bdry-rectifiability}, \ref{Thm:compactness} respectively.

In Section \ref{Sec:8} we discuss the basic properties of curvature varifolds with capillary boundary and prove Theorem \ref{Thm:compactness-CV}.

\

\noindent\textbf {Acknowledgements}
The authors would like to thank Professor  Ernst Kuwert for  helpful discussions, especially on curvature varifolds.
XZ would like to thank Kiichi Tashiro and Gaoming Wang for the interesting discussions.
\section{Preliminaries}\label{Sec:2}
We adopt the following basic notations throughout the paper.
\begin{itemize}
    \item $\mfR^{n+1}$ with $n\geq1$ denotes the Euclidean space, with  the Euclidean scalar product denoted by $\left<\cdot,\cdot\right>$ and the corresponding Levi-Civita connection denoted by $\na$.
    When considering the topology of $\mfR^{n+1}$, we denote by $\overline{E}$ the topological closure of a set $E$, by ${\rm int}(E)$ the topological interior of $E$, and by $\p E$ the topological boundary of $E$.
    \item $\bsmu_r$ is the homothety map $y\mapsto ry$;
    \item $\bstau_x$ is the translation map $y\mapsto y-x$;
    \item $\bseta_{x,r}$ is the composition $\bsmu_{1/r}\circ\bstau_x$, i.e., $y\mapsto\frac{y-x}r$;
    \item The \textit{tangent space} of any subset $A\subset\mfR^{n+1}$ at $p$, denoted by $T_pA$, is defined as the set of vectors $v\in\mfR^{n+1}$ satisfying: for every $\ep>0$, there exist $x\in A$, $r>0$, such that $\abs{x-p}<\ep, \abs{\bsmu_r\circ\bstau_p(x)-v}<\ep$.

    \item $\mcH^k$ is the $k$-dimensional Hausdorff measure on $\mfR^{n+1}$ and $\om_k$ is the $\mcH^k$-measure of $k$-dimensional unit ball;
    \item $B_r(x)$ is the closed ball in $\mfR^{n+1}$, centered at $x$ with radius $r>0$;
    \item For a Radon measure $\mu$ on $\mfR^{n+1}$:
    \begin{enumerate}
        \item ${\rm spt}\mu$ is the \textit{support} of the measure $\mu$.
        For $f:\mfR^{n+1}\ra\mfR^{n+1}$ proper, the \textit{push-forward} of $\mu$ through $f$ is the outer measure $(f)_\ast\mu$ defined by the formula
        \eq{
        \left((f)_\ast\mu\right)(E)
        =\mu\left(f^{-1}(E)\right),\quad E\subset\mfR^{n+1}.
        }
        See \cite[Section 2.4]{Mag12};
        \item For $x\in\mfR^{n+1}$ and $k\in\mbN$, the \textit{upper} and the \textit{lower $k$-densities} of $\mu$ at $x$ are given by
        \eq{
        \Theta^{\ast k}(\mu,x)=\limsup_{r\searrow0}\frac{\mu(B_r(x))}{\om_kr^k},\quad
        \Theta^{k}_\ast(\mu,x)=\liminf_{r\searrow0}\frac{\mu(B_r(x))}{\om_kr^k}.
        }
        If the above limits coincide, then we denote by $\Theta^k(\mu,x)$ this common value, which is called \textit{$k$-density} of $\mu$ at $x$;
        \item We say that $\tilde\mu$ is a \textit{k-blow-up} of $\mu$ at $x$ if there exists a sequence $r_j\searrow0$ such that
        \eq{
        \mu_j\coloneqq\frac1{r_j^k}(\bseta_{x,r_j})_\ast\mu\wsc\tilde\mu
        }
        as measures.
        The collection of $k$-blow-ups of $\mu$ at $x$ is denoted by ${\rm Tan}^k(\mu,x)$.
        If $\Theta^{\ast k}(\mu,x)<\infty$ then ${\rm Tan}^k(\mu,x)$ is non-empty by the Banach-Alaoglu Theorem.
        If $\Theta^k_\ast(\mu,x)>0$ then any $k$-blow-up of $\mu$ at $x$ is non-trivial.
        \item We say that $\mu$ is \textit{$k$-rectifiable} if there exists a $k$-rectifiable set $M$ and a positive function $\theta\in L^1_{\rm loc}(M,\mcH^k)$ such that $\mu=\theta\mcH^k\llcorner M$.
        A set $M$ is \textit{$k$-rectifiable} if it can be covered, up to a $\mcH^k$-negligible set, by a countable family of $C^1$, $k$-dimensional submanifolds of $\mfR^{n+1}$.
        One can also write $M$ as a disjoint union
        \eq{\label{eq:rectifiable-M}
        M=\bigcup_{j=0}^\infty M_j
        }
        where $\mcH^k(M_0)=0$, $M_j$ ($j\geq1$) is $\mcH^k$-measurable subset in $N_j$, with $N_j$ an embedded $C^1$, $k$-dimensional submanifold of $\mfR^{n+1}$, such that
        \eq{
        T_xM=T_xN_j,\quad\mcH^k\text{-a.e. }x\in M_j,
        }
        see e.g., \cite[Remark 11.7]{Simon83}.
    \end{enumerate}
    \item For a domain $\Om\subset\mfR^{n+1}$ of class $C^2$, we work with the following vector fields:
\eq{
\mathfrak{X}(\Om)
=&C^1(\overline\Om,\mfR^{n+1}),\\
\mathfrak{X}_t(\Om)
=&\left\{\varphi\in\mathfrak{X}(\Om):\varphi(x)\in T_xS, \forall x\in S\right\},\\
\mathfrak{X}_\perp(\Om)
=&\left\{\varphi\in\mathfrak{X}(\Om):\varphi(x)\in(T_xS)^\perp, \forall x\in S\right\},\\
\mathfrak{X}_0(\Om)
=&\left\{\varphi\in\mathfrak{X}(\Om):\varphi(x)=0, x\in S\right\},\\
\mathfrak{X}_c(\Om)
=&\left\{\varphi\in\mathfrak{X}(\Om):{\rm spt}\, \varphi\subset\subset\Om\right\}.
}
    \item For a bounded domain $\Om$ of class $C^2$, there exists $\rho_S>0$ (see e.g., \cite[Section 14.6]{GT01}) such that the \textit{normal exponential map}
\eq{
S\times[0,\rho_S)\ra U^+_{\rho_S}(S),\quad(x,\rho)\mapsto x+\rho\nu^S(x)
}
is a diffeomorphism.
Here $U^+_{\rho_S}$ is the one-sided tubular neighborhood of $S$.
We use $d_S$ to denote the signed distance function with respect to $S$ such that $d_S>0$ in $\Om$.
\end{itemize}

The following disintegration for Radon measures is well-known, see e.g., \cite{AFP00,AGS08}.
\begin{lemma}[Disintegration]
Let $\gamma$ be a Radon measure on $X\times Y$ with compact support, where $X,Y$ are metric spaces, and denote by $\pi:X\times Y\ra X$ the projection map.
Then $\mu_\gamma=(\pi)_\ast\gamma$ is a Radon measure, and there is a $\mu_\gamma$-a.e. uniquely determined family $(\gamma^x)_{x\in X}$ of Radon probability measures on $Y$, such that for any Borel function $\phi:X\times Y\mapsto[0,\infty]$ one has
\eq{
\int_{X\times Y}\phi\rd\gamma
=\int_X\int_Y\phi(x,y)\rd\gamma^x(y)\rd\mu_\gamma(x).
}
\end{lemma}
The following lemma will be used in the paper. 
\begin{lemma}[Marstrand-Mattila rectifiability criterion, {\cite[Theorem 5.1]{DeLellis08}}]\label{Lem:M-M-Reciti-Crit}
Let $k\in\mbN$ with $k\leq n+1$.
If a positive Radon measure on $\mfR^{n+1}$, say $\mu$, satisfies for $\mu$-a.e. $x\in\mfR^{n+1}$
\begin{enumerate}
    \item $0<\Theta^k_\ast(\mu,x)\leq\Theta^{\ast k}(\mu,x)<\infty$.
    \item Every $k$-blow-up of $\mu$ at $x$ takes the form $\theta\mcH^k\llcorner{L^k}$ for some $k$-dimensional linear subspace $L^k\subset\mfR^{n+1}$.
\end{enumerate}
Then $\mu$ is $k$-rectifiable.
\end{lemma}

\subsection{Varifolds}\label{Sec-2-vflds}

Let us begin by recalling some basic concepts of varifolds, we refer to \cite{Allard72,Simon83} for detailed accounts.

For $1\leq m\leq n+1$ we call $G(m,n+1)$ the \textit{Grassmannian} of the un-oriented $m$-palnes of $\mfR^{n+1}$.
For $U\subset\mfR^{n+1}$, $G_m(U)\coloneqq U\times G(m,n+1)$ denotes the \textit{trivial Grassmannian bundle} over $U$.
The space of \textit{$m$-varifolds} on $U\subset\mfR^{n+1}$, denoted by $\mfV^m(U)$, is the set of all positive Radon measures on the Grassmannian $G_m(U)$ equipped with the weak topology.
The \textit{weight measure} of a varifold $V\in\mfV^m(U)$, denoted by $\mu_V$, is the push-forward measure $(\pi)_\ast V$, where $\pi:G_m(U)\ra U$ is the canonical projection map.
For any Borel set $E\subset\mfR^{n+1}$, we denote by $V\llcorner E$ the {restriction} of $V$ to $E\times G(m,n+1)$. By \textit{support} of $V$ we mean ${\rm spt}\, \mu_V$, which is the smallest closed subset $B\subset\mfR^{n+1}$ such that $V\llcorner(\mfR^{n+1}\setminus B)=0$.
For any diffeomorphism $f:U\ra\mfR^{n+1}$, the continuous \textit{push-forward map} $f_\#:\mfV^m(U)\ra\mfV^m(f(U))$ is defined as in \cite[(39.1)]{Simon83}.
Note that this is not the push-forward of Radon measures introduced above,
therefore we adopt different notations.

Let $\Om$ be a domain of class $C^2$ and
let $V\in{\bf V}^m(\overline\Om)$. If $\varphi\in\mathfrak{X}(\Om)$ generates a one-parameter family of diffeomorphisms $\Phi_t$ of $\mfR^{n+1}$ with $\Phi_t(\overline\Om)\subset\overline\Om$ (at a point $x$ on $\p\Om$, one considers the tangent space $T_x\Om$ as the half $(n+1)$-space obtained by the blow-up of $\Om$ at $x$), then $(\Phi_t)_\#V\in{\bf V}^m(\overline\Om)$ and its \textit{first variation} with respect to $\varphi$ is, see \cite[(4.2), (4.4)]{Allard72},
\eq{\label{defn-1st-variationformula}
    \delta V(\varphi)
    \coloneqq\frac{\rd}{\rd t}_{|_{t=0}}\mu_{(\Phi_t)_\# V}(\mfR^{n+1})
    =\int_{G_m(\overline\Om)}{\rm div}_P\varphi(x)\rd V(x,P),
}
where ${\rm div}_P\varphi(x)=\sum_i\left<\na_{e_i}\varphi, e_i\right>$ and $\{e_1,\ldots,e_n\}\subset P$ is any orthonormal basis.
We say that $V\in{\bf V}^m(\overline\Om)$ has \textit{bounded first variation} if
\eq{
\sup\{\abs{\de V(\varphi)}:\varphi\in\mathfrak{X}(\Om),\abs{\varphi}\leq1\}<+\infty.
}

As in \cite[Definition 42.3]{Simon83}, we denote ${\rm VarTan}(V,x)$ to be the set of \textit{varifold tangents} of $V$ at  $x\in{\rm spt}\,\mu_V$.
By the compactness of Radon measures \cite[Theorem 4.4]{Simon83}, ${\rm VarTan}(V,x)$ is compact and non-empty provided that the upper density $\Theta^{\ast m}(\mu_{V},x)$ is finite.
Moreover, there exists a non-zero element in ${\rm VarTan}(V,x)$ if and only if $\Theta^{\ast m}(\mu_V,x)>0$, see \cite[3.4]{Allard72}.
\subsection{Rectifiable varifolds}
Given a $k$-rectifiable measure $\mu=\theta\mcH^k\llcorner M$, the naturally induced \textit{$k$-rectifiable varifold} is given by
\eq{
V=\theta\mcH^k\llcorner M\otimes\de_{T_xM},
}
where $T_xM$ is the \textit{approximate tangent space} of $M$ at $x$, which exists $\mcH^k$-a.e., $\de_{T_xM}$ is the \textit{Dirac measure} and $\theta$ is called the \textit{multiplicity function}.
If $\theta$ only takes integer values, then the rectifiable varifold $V$ is called \textit{integral varifold}.
We denote by $\mfRV^k(U)$ the set of \textit{rectifiable $k$-varifolds} in $U\subset\mfR^{n+1}$, and by ${\bf IV}^k(U)$ the set of \textit{integral $k$-varifolds} in $U\subset\mfR^{n+1}$.

\subsection{Allard's varifolds with boundary}\label{Sec-2-Allard-vfld-bdry}
We recall here Allard's notions for readers' convenience.

Given $2\leq m\leq n+1$, let $B$ be a smooth $(m-1)$-dimensional submanifold of $\mfR^{n+1}$, let $U$ be a tubular neighborhood of $B$ on which the distance function to $B$ is smooth, and let $V\in{\bf V}^m(U)$.
We could consider the boundary behavior of $V$ near $B$ if
\begin{enumerate}
    \item $\mu_V(U\cap B)=0$;
    \item $\norm{\de V}\llcorner(U\setminus B)$ is a Radon measure on $U$, where $\norm{\de V}$ is the total variation of the first variation $\de V$.
\end{enumerate}
This is somewhat a Dirichlet boundary condition of $V$ since it is shown in \cite[Section 3]{Allard75} that for any such $V$, $\norm{\de V}$ is a Radon measure on $U$, and the representing vector of $\de V$ along $B$ points normal to $B$.

\subsection{Prescribed contact angle condition}
We discuss Definition \ref{Defn:G_m,beta(S)} in the following.

Note that condition ($i$) infers that $P$ meets $S$ transversally since $\sin\beta(x)>0$, and hence we have ${\rm dim}\left(P\cap T_xS\right)=m-1$.
Condition ({ii}) is always satisfied when either $\beta(x)=\frac\pi2$, since in that case $P(\nu^S(x))=\nu^S(x)$;
or when $m=n$, since in that case $P$ and $T_xS$ are $n$-planes intersecting transversally along the common $(n-1)$-plane $(P\cap T_xS)$, so that the projection $P(\nu^S(x))$ is always perpendicular to this $(n-1)$-plane.
{
It is clear that from condition ($i$) 
we have
\eq{\label{eq:same-bundles}
G_{m,\beta}(S)
=G_{m,\pi-\beta}(S).
}
}

The fact that $\beta$ serves as the contact angle prescribing function is revealed by the following equivalent definition of $G_{m,\beta}(S)$.

\begin{remark}[Prescribed contact angle]
\normalfont
The function
$\beta$ in Definition \ref{Defn:G_m,beta(S)} is a function prescribing the contact angle in the following sense:
Fix a point $x\in S$, 
the relation
\eq{
\left<\xi,\nu^S(x)\right>=\cos\beta(x), \quad\xi\in\mfS^n
}
determines an $(n-1)$-dimensional sphere on the unit sphere $\mfS^n$, and for each point  $\xi$ on this sphere, we naturally identify it with an $n$-plane $P^n_\xi\in G(n,n+1)$ by the relation $\xi\perp P^n_\xi$. Let $G_{n,\beta}(x)$ be the collection of all such $n$-planes.
Clearly $P_\xi^n\cap T_xS$ is an $(n-1)$-plane.
In particular, once a $P^n_\xi$ is fixed we write
\eq{
\mfn(x,P^n_\xi)
\coloneqq\frac{P^n_\xi(\nu^S(x))}{\abs{P^n_\xi(\nu^S(x))}}.
}
Note that $\mfn$ is well-defined since $\beta(x)\in(0,\pi)$.

For a fixed $m\in\mbN$ with $m\leq n-1$, let $G_{m,\beta}(x)$ be the collection of $m$-planes spanned by $\mfn(x,P^n_\xi),e_1$, $\ldots,e_{m-1}$, where $P^n_\xi\in G_{n,\beta}(x)$,
$\{e_i\}$ is an orthonormal basis satisfying
\eq{
{\rm span}\{e_1,\ldots,e_{m-1}\}\subset \left(P^n_\xi\cap T_xS\right),\qquad\text{  }{\rm span}\{e_1,\ldots,e_{m-1}\}\perp\mfn(x,P^n_\xi).
}
$G_{m,\beta}(S)$ is then obtained by
collecting all such $G_{m,\beta}(x)$ for every $x\in S$.
For the case $m=1$ we simply have
\eq{
G_{1,\beta}(x)=\bigcup_{P^n_\xi\in G_{n,\beta}(x)}{\rm span}\{\mfn(x,P^n_\xi)\}.
}

\end{remark}

\section{Varifolds with capillary boundary 
}\label{Sec:3}

Let $\Om,\beta$ be given as in Definition \ref{Defn:G_m,beta(S)}.
Let $V\in{\bf V}^m_\beta(\Om)$ with boundary varifold $\Gamma$ in the sense of Definition \ref{Defn:vfld-prescribed-bdry}.
We begin with the following simple observations:
\begin{enumerate}
    \item From \eqref{eq:same-bundles} it is clear that, if $V$ has prescribed contact angle $\beta$ with $S$ along $\Gamma$, then $V$ also has prescribed contact angle $(\pi-\beta)$ with $S$ along the same boundary varifold $\Gamma$.
Therefore
\eq{
{\bf V}^m_\beta(\Om)
={\bf V}^m_{\pi-\beta}(\Om).
}
    \item For $\mfn_V$ defined through \eqref{defn:n_V-intro}, we have for $\sigma_\Gamma$-a.e. $x$
\eq{\label{eq:n_V,nu^S=sinbeta}
\left<\mfn_V(x),\nu^S(x)\right>
=\int_{G_{m,\beta}(x)}\left<\mfn(x,P),\nu^S(x)\right>\rd\Gamma^x(P)
=\sin\beta(x),
}
where we have used \eqref{defn:n_V-intro}, \eqref{eq:n(x,Q),nu^S}, and the fact that $\Gamma^x$ is a Radon probability measure.
From \eqref{defn:n_V-intro} we see that $\mfn_V$ is a $\sigma_\Gamma$-measurable vector field on $S$ with $\mfn_V\in L^\infty(S,\sigma_\Gamma)$,
and hence approximate continuous with respect to $\sigma_\Gamma$ at $\sigma_\Gamma$-a.e. points.
Note that $\mfn_V$ is in general {\it not} a unit vector.
In fact, it is easy to see that $\norm{\mfn_V}_{L^\infty(\sigma_\Gamma)}\leq1$.
    \item We define \textit{generalized inwards pointing co-normal} with respect to $S$ as
\eq{\label{defn:n_W}
\abs{\cos\beta(x)}\mfn_W(x)
&\coloneqq{\rm sgn}(\cos\beta(x))\int_{G_{m,\beta}(x)}T_xS\left(\mfn(x,P)\right)\rd\Gamma^x(P) \\
&
={\rm sgn}(\cos\beta(x))T_xS\left(\mfn_V(x)\right)
}
whence $\Gamma^x$ is a Radon probability measure on $G_{m,\beta}(x)$ and we adopt the notation ${\rm sgn}(0)=0$, so that when $\beta(x)=\frac\pi2$, \eqref{defn:n_W} reads as ``$0=0$''.
Clearly
\eq{
\left<\mfn_W(x),\nu^S(x)\right>=0.
}
Moreover, since \eqref{defn:1st-vairation-intro2} deals with $\varphi\in\mathfrak{X}_t(\Om)$, we can now rewrite it as:
\eq{
\de V(\varphi)
=&-\int_{\overline\Om}\left<\mfH(x),\varphi(x)\right>\rd\mu_V(x)-\int_S\left<\mfn_V(x),\varphi(x)\right>\rd\sigma_\Gamma(x)\\
=&-\int_{\overline\Om}\left<\mfH(x),\varphi(x)\right>\rd\mu_V(x)-\int_S\cos\beta(x)\left<\mfn_W(x),\varphi(x)\right>\rd\sigma_\Gamma(x),
}
where we have used the fact that $T_xS(\mfn_V(x))=\cos\beta(x)\mfn_W(x)$ by \eqref{defn:n_W}.
\end{enumerate}


\subsection{Basic properties}

\begin{proposition}\label{Prop:unit-vectors}
Let $\Om,\beta$ be as in Definition \ref{Defn:G_m,beta(S)}.
Let $V\in{\bf V}^m_\beta(\Om)$ with boundary varifold $\Gamma$.
If $\mfn_V(x_0)$ is well-defined, that is, $\Gamma^{x_0}$ is a Radon probability measure on $G_{m,\beta}(x_0)$,  for some $x_0\in{\rm spt}\,\sigma_\Gamma$, then $\abs{\mfn_V(x_0)}=1$ if and only if
\eq{
\mfn(x,P)
=\mfn_V(x_0),\text{ for }\Gamma^{x_0}\text{ -a.e. }P\in G_{m,\beta}(x_0).
}
Moreover, $\abs{\mfn_V(x_0)}=1$ is equivalent to $\abs{\mfn_W(x_0)}=1$, provided that $\cos\beta(x_0)\neq0$.
In particular, under the above conditions if $m=1$ then $\abs{\mfn_V(x_0)}=1$ if and only if
\eq{
\Gamma^{x_0}=\de_{P_0},\quad \text{ for }\mfn_V(x_0)=\mfn(x,P_0)\text{ with }P_0\in G_{1,\beta}(x_0).
}
\end{proposition}
\begin{proof}
Since $\mfn_V(x_0)$ and $\mfn(x,P)$ are unit, we have by virtue of \eqref{defn:n_V-intro}
\eq{
0
=1-\left<\mfn_V(x_0),\mfn_V(x_0)\right>
=&\int_{G_{m,\beta}(x_0)}\{1-\left<\mfn(x,P),\mfn_V(x_0)\right>\}\rd\Gamma^{x_0}(P)
\geq0,
}
which implies, for $\Gamma^{x_0}$-a.e. $P\in G_{m,\beta}(x_0)$
\eq{
\mfn(x,P)
=\mfn_V(x_0).
}
Reversely, it is clear that $\mfn_V(x_0)$ has unit length, thanks to \eqref{defn:n(x,P)}.

Moreover,
thanks to \eqref{eq:n_V,nu^S=sinbeta}, we know that $\abs{\mfn_V(x_0)}=1$ implies
\eq{
\Abs{T_{x_0}S(\mfn_V(x_0))}
=\sqrt{1-\sin^2\beta(x_0)}
=\abs{\cos\beta(x_0)},
}
when $\cos\beta(x_0)\neq0$ we have by virtue of \eqref{defn:n_W} that $\abs{\mfn_W(x_0)}=1$.
We may use a similar argument to show that $\abs{\mfn_W(x_0)}=1$ implies $\abs{\mfn_V(x_0)}=1$.

Finally, recall the definitions of $G_{1,\beta}(x_0)$, $\mfn_V(x_0)$,
we know that there exists a unique $1$-plane $P_0\in G_{1,\beta}(x_0)$ such that
\eq{
\mfn(x,P_0)
=\mfn_V(x_0).
}
It follows that
$\Gamma^{x_0}
=\de_{P_0}$ as measure on $G_{1,\beta}(x_0)$.
\end{proof}

\begin{proposition}\label{Prop:1st-variation}
Let $\Om\subset\mfR^{n+1}$ be a bounded domain of class $C^2$ and $\beta\in C^1(S,(0,\pi))$, let $V\in{\bf V}^m_\beta(\Om)$ with boundary varifold $\Gamma$.
Then $V$ has bounded first variation.
More precisely, there exist a $\mu_V$-measurable vector field $\widetilde H$ on $S$ which is orthogonal to $S$ for $\mu_V$-a.e. $x\in S$ with $\norm{\widetilde H}_{L^\infty(\mu_V)}$ depends only on the second fundamental form of $S$, and a positive Radon measure $\sigma_V^\perp$ on $S$, 
such that for all $\varphi\in\mathfrak{X}(\Om)$
\eq{\label{eq:1st-variation-V_beta}
\de V(\varphi)
=&-\int_{\overline\Om}\left<\mfH,\varphi\right>\rd\mu_V-\int_S\left<\widetilde H,\varphi\right>\rd\mu_V\\
&-\int_{S}\left<\nu^S,\varphi\right>\rd\sigma_V^\perp-\int_{S\cap\{x: \cos\beta(x)\neq0\}}\left<\mfn_V(x),\varphi^T(x)\right>\rd\sigma_\Gamma(x),
}
where $\varphi^T(x)=T_xS(\varphi(x)),\forall x\in S$.
The last integral can be also written as $$\int_S\left<\cos\beta\,\mfn_W,\varphi\right>(x)\rd\sigma_\Gamma(x).$$
Moreover,
\begin{enumerate}
    \item (normal part of the boundary measure): For any $x_0\in S$ and any $\rho\leq \rho_S$, there exists $c=c(\Om)>0$ such that the following local estimate and global estimate hold:
\eq{\label{esti:local-perp}
\sigma_V^\perp(B_{\frac{\rho}2}(x_0))
\leq\frac{c}\rho\mu_V(B_\rho(x_0))
+\int_{B_\rho(x_0)}\abs{\mfH}\rd(\mu_V\llcorner\Om),
}
\eq{\label{esti:global-perp}
\sigma_V^\perp(S)
\leq c\mu_V(\overline\Om)+\int_\Om\abs{\mfH}\rd\mu_V.
}
    \item (tangential part of the boundary measure): For any $x_0\in{\rm spt}\sigma_\Gamma$ at which ${\mfn_V}$ is approximate continuous with respect to $\sigma_\Gamma$ and satisfies
\eq{\label{condi:n_V}
\abs{\mfn_V(x_0)}
\geq\sqrt{\sin^2\beta(x_0)+\ep^2\cos^2\beta(x_0)}\text{ for some }\ep\in(0,1],
}
there exist some $C_0=C_0(\Om)>0,0<\rho_1<\frac{\rho_S}2$, such that for all $0<\rho<\rho_1$,
the following local estimate holds:
\eq{\label{esti:local-tangent}
\ep\abs{\cos\beta(x_0)}\sigma_\Gamma(B_{\frac{\rho}2}(x_0))
\leq\frac{C_0}{\rho}\mu_V(B_{2\rho}(x_0))+\sigma_\Gamma(B_{\rho}(x_0))
+\int_{B_{2\rho}(x_0)}\abs{\mfH}\rd\mu_V.
}
In particular, if
\eqref{defn:genera-capi-bdry-point-1} is satisfied by $x_0$ with the corresponding $\tau_{x_0}^V$, $\ep_{x_0},$ and $\rho_{x_0}$, then the local estimate improves as follows: for any $0<\rho<\min\{\rho_{x_0},\frac{\rho_S}2\}\}$,
\eq{\label{esti:local-tangent-improved}
\ep_{x_0}\sigma_\Gamma(B_{\frac{\rho}2}(x_0))
\leq\frac{C_0}{\rho}\mu_V(B_{2\rho}(x_0))
+\int_{B_{2\rho}(x_0)}\abs{\mfH}\rd\mu_V.
}
\end{enumerate}
\end{proposition}
\begin{remark}\label{Rem:Item-1}
\normalfont
A direct computation shows that
condition \eqref{condi:n_V} is equivalent to requiring that
\eq{\label{condi:n_W}
\abs{\cos\beta(x_0)\mfn_W(x_0)}
\geq\ep\abs{\cos\beta(x_0)},
}
which clearly holds when $\beta(x_0)=\frac\pi2$.
In this case, the local estimate \eqref{esti:local-tangent} says nothing as the LHS is identically $0$.
We remark that  the estimate blows up when $\ep<<1$, since in this case the generalized outer co-normal $\mfn_W$ has a rather small length. But in the ideal (smooth) case, $\mfn_W$ should have unit length.
\end{remark}
\begin{proof}[Proof of Proposition \ref{Prop:1st-variation}]
First we prove that $V$ has bounded first variation with respect to $\mathfrak{X}_t(\Om)$.
In fact, from \eqref{defn:1st-vairation-intro2} we deduce directly
\eq{
\abs{\de V(\varphi)}
\leq\left(\norm{\mfH}_{L^1(\mu_V)}+\sigma_\Gamma(S)\right)\abs{\varphi}_{C^0(\overline\Om)},\quad\forall\varphi\in\mathfrak{X}_t(\Om).
}
Since $S$ is compact we have $\sigma_\Gamma(S)<\infty$, the claimed fact follows.

Then we prove that $V$ has bounded first variation with respect to $\mathfrak{X}_\perp(\Om)$.
To see this we first note that $\mathfrak{X}_0(\Om)\subset\mathfrak{X}_t(\Om)$ so that we deduce from \eqref{defn:1st-vairation-intro2}
\eq{
\de V(\varphi)
=-\int_{\Om}\left<\mfH(x),\varphi(x)\right>\rd\mu_V(x),\quad\forall\varphi\in\mathfrak{X}_0(\Om).
}
Since $\mfH\in L^1(\mu_V)$, we could directly apply \cite[Theorem 1.1]{DeMasi21} to deduce that $V$ has bounded first variation with respect to $\mathfrak{X}_\perp(\Om)$, and there exists a positive Radon measure $\sigma_V^\perp$ on $S$ and a $\mu_V$-measurable vector field $\widetilde H$ on $S$ such that for any $\varphi\in\mathfrak{X}_\perp(\Om)$,
\eq{\label{eq:1st-variation-perp}
\de V(\varphi)
=-\int_\Om\left<\mfH,\varphi\right>\rd\mu_V-\int_S\left<\widetilde H,\varphi\right>\rd\mu_V-\int_S\left<\nu^S,\varphi\right>\rd\sigma_V^\perp,
}
where $\widetilde H$ is orthogonal to $S$ for $\mu_V$-a.e. $x\in S$, $\widetilde H\in L^\infty(S,\mu_V)$ and the $L^\infty$-norm depends only on the second fundamental form of $S$.
In the co-dimension-$1$ case, $\widetilde H(x)={\rm tr}(h^S(x))$ is just the mean curvature of $S$.
Moreover, the local estimate \eqref{esti:local-perp} and the global estimate \eqref{esti:global-perp} hold.

Thus we have shown that $V$ has bounded first variation, and hence thanks to Radon-Nikodym Theorem we can write: for any $\varphi\in\mathfrak{X}(\Om)$
\eq{\label{eq:1st-variation-Radon-Niko}
\de V(\varphi)
=-\int_{\overline\Om}\langle\widehat H,\varphi\rangle\rd\mu_V-\int\left<\eta,\varphi\right>\rd\sigma_V,
}
for some $\widehat H\in L^1(\mu_V)$ and some
positive Radon measure $\sigma_V\perp\mu_V$.
Since any $\varphi\in\mathfrak{X}(\Om)$ can be decomposed as $\varphi=\varphi^T+\varphi^\perp$ where $\varphi^T\in\mathfrak{X}_t(\Om)$ and $\varphi^\perp\in\mathfrak{X}^\perp(\Om)$, we can therefore compare \eqref{eq:1st-variation-Radon-Niko} with \eqref{defn:1st-vairation-intro2} and \eqref{eq:1st-variation-perp} to obtain that
\eq{\label{eq:widehat-H}
\widehat H\cdot\mu_V\llcorner\overline\Om
=\mfH\cdot\mu_V\llcorner\overline\Om+\widetilde H\cdot\mu_V\llcorner S
}
as measure,
and
\eq{
\eta\cdot\sigma_V
=\nu^S\cdot\sigma_V^\perp+T_xS\left(\mfn_V(x)\right)\cdot\sigma_\Gamma
}
as measure supported on $S$, which proves \eqref{eq:1st-variation-V_beta}.

It is thus left to prove the estimate \eqref{esti:local-tangent}.
We assume in the following $\cos\beta(x_0)\neq0$, otherwise the estimate \eqref{esti:local-tangent} is trivial since LHS vanishes.

First note that by \eqref{condi:n_V} and \eqref{condi:n_W}, there exists some unit vector $\tau\in T_{x_0}S$ such that
\eq{
\left<\tau,\cos\beta(x_0)\mfn_W(x_0)\right>
\geq\ep\abs{\cos\beta(x_0)}
>0.
}
Define a $\sigma_\Gamma$-measurable function $g(x)\coloneqq\left<\tau,\cos\beta(x)\mfn_W(x)\right>$, then $g$ is approximate continuous at $x_0$ with respect to $\sigma_\Gamma$.
By virtue of the approximate continuity, we could find some $\rho_1>0$ such that for any $0<\rho<\rho_1$,
\eq{\label{eq:sigma_Gamma(G)}
\sigma_\Gamma(G_\rho)
\coloneqq
\sigma_\Gamma\left(\left\{x\in B_\rho(x_0):\abs{g(x)-g(x_0)}\geq\frac\ep2\abs{\cos\beta(x_0)}\right\}\right)
\leq\frac12\sigma_\Gamma(B_\rho(x_0)).
}

Now we consider a standard smooth cut-off function $\gamma$ on $\mfR$ with $\gamma\equiv1$ on $[0,1-t]$, $\gamma\equiv0$ on $[1,\infty)$, $-\frac2t\leq\gamma'\leq0$ on $\mfR$, for $t>0$ arbitrarily small.
After translation we may assume that $x_0=0$.
Letting $\varphi(x)=\gamma(\frac{\abs{x}}\rho)\tau$,
we have for every $0<\rho<\rho_1$
\eq{\label{eq:integral-g(x)}
&\int_S\left<\varphi(x),\cos\beta(x)\mfn_W(x)\right>\rd\sigma_\Gamma
=
\int_{B_\rho(0)}\gamma(\frac{\abs{x}}\rho)g(x)\rd\sigma_\Gamma(x)\\
=&\int_{G_\rho}\gamma(\frac{\abs{x}}\rho)g(x)\rd\sigma_\Gamma(x)+\int_{B_\rho(0)\setminus G_\rho}\gamma(\frac{\abs{x}}\rho)g(x)\rd\sigma_\Gamma(x).
}
On $B_\rho(0)\setminus G_\rho$ we have by construction $g(x)\geq\frac\ep2\abs{\cos\beta(0)}$ with
\eq{
\sigma_\Gamma(B_{\frac\rho2}(0)\setminus G_{\rho})
=\sigma_\Gamma(B_{\frac\rho2}(0)\setminus G_{\frac\rho2})
\geq\frac12\sigma_\Gamma(B_{\frac\rho2}(0)).
}
For $t$ sufficiently small we have $1-t>\frac12$,
thus we find
\eq{
\int_{B_\rho(0)\setminus G_\rho}\gamma(\frac{\abs{x}}\rho)g(x)\rd\sigma_\Gamma(x)
\geq\int_{B_{\frac{\rho}2(0)}\setminus G_\rho}\frac\ep2\abs{\cos\beta(0)}\rd\sigma_\Gamma
\geq\frac\ep2\abs{\cos\beta(0)}\sigma_\Gamma(B_\frac\rho2(0)),
}
and in turn
\eq{\label{ineq:rho-2rho}
\int_S\left<\varphi(x),\cos\beta(x)\mfn_W(x)\right>\rd\sigma_\Gamma
\geq&-\sigma_\Gamma(G_\rho)+\frac\ep2\abs{\cos\beta(0)}\sigma_\Gamma(B_\frac\rho2(0))\\
\geq&-\frac12\sigma_\Gamma(B_\rho(0))+\frac\ep2\abs{\cos\beta(0)}\sigma_\Gamma(B_\frac\rho2(0)),
}
where we have simply estimated $\gamma(\frac{\abs{x}}\rho)g(x)\geq-1$ on $G_\rho$.

Using such $\varphi$ in \eqref{eq:1st-variation-V_beta}, we obtain
\eq{\label{eq:integral-g(x)-1st-variation}
\int_S\left<\varphi(x),\cos\beta(x)\mfn_W(x)\right>\rd\sigma_\Gamma
=-\de V(\varphi)-\int_{\overline\Om}\left<\mfH+\widetilde H,\varphi\right>\rd\mu_V-\int_S\left<\nu^S,\varphi\right>\rd\sigma_V^\perp\\
=\int_{G_m(\overline\Om)}-\gamma'(\frac{\abs{x}}\rho)\frac1\rho\left<P(\frac{x}{\abs{x}}),\tau\right>-\gamma(\frac{\abs{x}}\rho){\rm div}_P(\tau)\rd V(x,P)\\
-\int_{\overline\Om}\gamma(\frac{\abs{x}}\rho)\left<\tau,\mfH+\widetilde H\right>\rd\mu_V-\int_S\gamma(\frac{\abs{x}}\rho)\left<\nu^S,\tau\right>\rd\sigma_V^\perp\\
\leq\frac3\rho\mu_V(B_\rho(0))+\int_{B_\rho(0)}(c+\abs{\mfH(x)})\rd\mu_V(x)+\sigma_V^\perp(B_\rho(0)),
}
where we have used the trivial fact that ${\rm div}_P(\tau)=0$ since $\tau$ is a constant vector field.
Taking \eqref{ineq:rho-2rho} and \eqref{esti:local-perp} into account, after further decreasing $\rho_1<\frac12\rho_S$, we obtain \eqref{esti:local-tangent}.

If
\eqref{defn:genera-capi-bdry-point-1} is satisfied by $x_0$ with the corresponding $\tau_{x_0}^V$, $\ep_{x_0},$ and $\rho_{x_0}$, then instead of \eqref{eq:sigma_Gamma(G)} we may directly define $g(x)\coloneqq\left<\tau^V_{x_0},\cos\beta(x)\mfn_W(x)\right>$, which satisfies by assumption
\eq{
g(x)\geq\ep_{x_0}>0,\quad
\sigma_\Gamma\text{-a.e. }x\in B_{\rho_{x_0}}(x_0).
}
So that for any $0<\rho<\min\{\rho_{x_0},\frac{\rho_S}2\}$, \eqref{eq:integral-g(x)} (with $\varphi(x)=\gamma(\frac{\abs{x}}\rho)\tau^V_{x_0}$) can be simply estimated:
\eq{\label{ineq:integral-g-lower-bdd}
\int_S\left<\varphi,\cos\beta\mfn_W\right>\rd\sigma_\Gamma
=&
\int_{B_\rho(0)}\gamma(\frac{\abs{x}}\rho)g(x)\rd\sigma_\Gamma(x)\\
\geq&\ep_{x_0}\int_S\gamma(\frac{\abs{x}}\rho)\rd\sigma_\Gamma(x)
\geq\ep_{x_0}\sigma_\Gamma(B_\frac\rho2(0)).
}
Putting this back into \eqref{eq:integral-g(x)-1st-variation} and taking \eqref{esti:local-perp} into account,
the claimed estimate \eqref{esti:local-tangent-improved} then follows, which completes the proof.
\end{proof}

\begin{corollary}
Let $\Om\subset\mfR^{n+1}$ be a bounded domain of class $C^2$ and $\beta\in C^1(S,(0,\pi))$.
For $V\in{\bf V}^m_\beta(\Om)$,
the following statements are equivalent:
\begin{enumerate}
    \item $V$ is an $m$-rectifiable varifold, that is, $V\in{\bf RV}^m(\overline\Om)$.
    \item $\Theta^m(\mu_V,x)>0$ for $\mu_V$-a.e. $x$.
\end{enumerate}
\end{corollary}
\begin{proof}
This follows directly from Proposition \ref{Prop:1st-variation} and the Rectifiability Theorem for varifolds, see e.g., \cite[Theorem 42.4]{Simon83}.
\end{proof}

\begin{corollary}\label{Cor:spt}
Let $\Om\subset\mfR^{n+1}$ be a bounded domain of class $C^2$ and $\beta\in C^1(S,(0,\pi))$.
Let $V\in{\bf V}^m_\beta(\Om)$ with boundary varifold $\Gamma$, such that $\mfH\in L^p(\mu_V)$ for some $p\in(1,\infty)$.
Then
\begin{enumerate}
    \item ${\rm spt}\,\sigma_V^\perp\subset{\rm spt}\,\mu_{V}\cap S$.
    \item Under the assumptions which ensure \eqref{esti:local-tangent-improved}, if $\cos\beta(x_0)\neq0$ then the implication holds:
\eq{
x_0\in{\rm spt}\sigma_\Gamma
\Rightarrow x_0\in{\rm spt}\mu_{V}\cap S.
}
Moreover, if the collection of all such points is dense in ${\rm spt}\sigma_\Gamma\cap\{x\in S:\cos\beta(x)\neq0\}$, then we have
\eq{
{\rm spt}\sigma_\Gamma\cap\{x\in S:\cos\beta(x)\neq0\}
\subset{\rm spt}\mu_{V}\cap S.
}
\end{enumerate}
In particular, if $\mu_V(S)=0$
then the above statements hold for $\mu_{V}=\mu_{V_I}$, where $V_I\coloneqq V\llcorner\Om$.
\end{corollary}
\begin{proof}
By virtue of \eqref{esti:local-perp} we have for any $x_0\in S$ and any $\rho<\rho_S$
\eq{
\sigma_V^\perp(B_\frac\rho2(x_0))
\leq\frac{c}\rho\mu_{V}(B_\rho(x_0))
+\norm{H}_{L^p(B_\rho(x_0),\mu_V)}\mu_{V}(B_\rho(x_0))^{1-\frac1p},
}
so that if $\mu_{V}(B_r(x_0))=0$ for some $r>0$, then we must have $\sigma_V^\perp(B_{\frac{r}2}(x_0))=0$.
Since this is true for all $x_0\in S$,  ({1}) is thus proved.

Now we show (2). The assertion concerning local property follows similarly by virtue of \eqref{esti:local-tangent-improved}.
The assertion concerning global property follows by a covering argument.
This completes the proof.
\end{proof}

\begin{remark}\label{Rem:G_pb-spt}
\normalfont
Recall Definition \ref{Defn:capillary-bdry-point},
the second conclusion of Corollary \ref{Cor:spt} can be then restated (in a stronger form) as:
Let $\Om\subset\mfR^{n+1}$ be a bounded domain of class $C^2$ and $\beta\in C^1(S,(0,\pi))$.
Let $V\in{\bf V}^m_\beta(\Om)$ with boundary varifold $\Gamma$, such that $\mfH\in L^p(\mu_V)$ for some $p\in(1,\infty)$.
Then the implication holds:
\eq{
x_0\in \mathscr{P}_{\rm cb}(\Gamma)
\Rightarrow x_0\in{\rm spt}\mu_{V}\cap S.
}
Moreover, if $\mathscr{P}_{\rm cb}(\Gamma)$ is dense in ${\rm spt}\sigma_\Gamma\cap\{x\in S:\cos\beta(x)\neq0\}$, then we have
\eq{
{\rm spt}\sigma_\Gamma\cap\{x\in S:\cos\beta(x)\neq0\}
\subset{\rm spt}\mu_{V}\cap S.
}
\end{remark}

\subsection{Examples}
We collect several examples that are helpful for understanding Definitions \ref{Defn:vfld-prescribed-bdry} and \ref{Defn:capillary-bdry-point}.
\begin{example}[Smooth submanifolds]\label{exam:submflds}
When $\Sigma\subset\Om$ is an $m$-dimensional submanifold with boundary $\p \Sigma$ supported on $S$ and intersecting $S$ with angle $\beta(x)$ at $x\in \p S$ in the following sense
\[
\langle \nu^S (x),n_\S(x)\rangle =\sin \beta(x), 
\]
where $n_\S$ is the inwards-pointing unit co-normal
of $\S$ along $\p\S$, then one can check that $V=\mathcal{H}^m\llcorner \Sigma \otimes \delta_{T_x\Sigma}$ has prescribed contact angle $\beta$ with $\Gamma =\mathcal {H}^{m-1} \llcorner \partial\Sigma\otimes \delta_{T_x\p\Sigma} $.
Precisely, for any $\varphi\in\mathfrak{X}_t(\Om)$, the first variation formula holds:
\eq{
\de V(\varphi)
=\int_\Om\left<\mfH,\varphi\right>\rd\mu_V-\int_{\p\S}\left<n_\S,\varphi\right>\rd\mcH^{m-1}.
}
    
\end{example}

\begin{example}[Unions of two planes in $G_{m,\beta_0}(\p\mfR^{n+1}_+)$ with common boundary]\label{exam:planes}
Let $\beta_0$ be a constant function on $\p\mfR^{n+1}_+$ with values in $(0,\pi)\setminus\{\frac\pi2\}$.
Let $P_\pm$ be two ``antipodal'' $m$-planes in $G_{m,\beta_0}(\p\mfR^{n+1}_+)$, in the sense that they intersect $\p\mfR^{n+1}_+$ along a common $(m-1)$-plane $L$,
with the inwards-pointing unit co-normal of $P_\pm\cap\mfR^{n+1}_+$ along $L$, denoted by $n_\pm$ respectively, satisfies
\eq{\label{condi:n_++n_-}
\left<n_\pm,e_{n+1}\right>=\sin\beta_0,\quad
\frac12(n_++n_-)=\sin\beta_0e_{n+1}.
}
One can directly check that $V=\mcH^m\llcorner(P_+\cap\mfR^{n+1}_+)\otimes\frac{\de_{P_+}}2+\mcH^m\llcorner(P_-\cap\mfR^{n+1}_+)\otimes\frac{\de_{P_-}}2$ has constant prescribed contact angle $\beta_0$ with boundary varifold $\Gamma=\mcH^{m-1}\llcorner L\otimes\frac{\de_{P_+}+\de_{P_-}}2$.
In fact, for any $\varphi\in\mathfrak{X}(\mfR^{n+1}_+)$ with compact support, one can check that
\eq{
\de V(\varphi)
=-\int_L\left<\frac{n_++n_-}2,\varphi\right>\rd\mcH^{m-1}
=-\int_{G_{m,\beta_0}(\p\mfR^{n+1}_+)}\left<\mfn(x,P),\varphi(x)\right>\rd\Gamma(x,P).
}
Moreover, for any $x\in L$, by \eqref{defn:n_V-intro} $\mfn_V(x)=\sin\beta_0e_{n+1}$, while by \eqref{condi:n_++n_-} we have $\de V(\varphi)=0$ for any compactly supported tangential variation $\varphi\in\mathfrak{X}_t(\mfR^{n+1}_+)$ and also $\mfn_W(x)=0$.
Comparing the first variation in \eqref{eq:1st-variation-V_beta}, we find
\eq{
\sigma_\Gamma
=\mcH^{m-1}\llcorner L,
\quad\sigma_V^\perp=\sin\beta_0\mcH^{m-1}\llcorner L.
}
See Figure \ref{Fig-1} (in co-dimension-1 case) for illustration.
\begin{figure}[H]
	\centering
	\includegraphics[width=14cm]{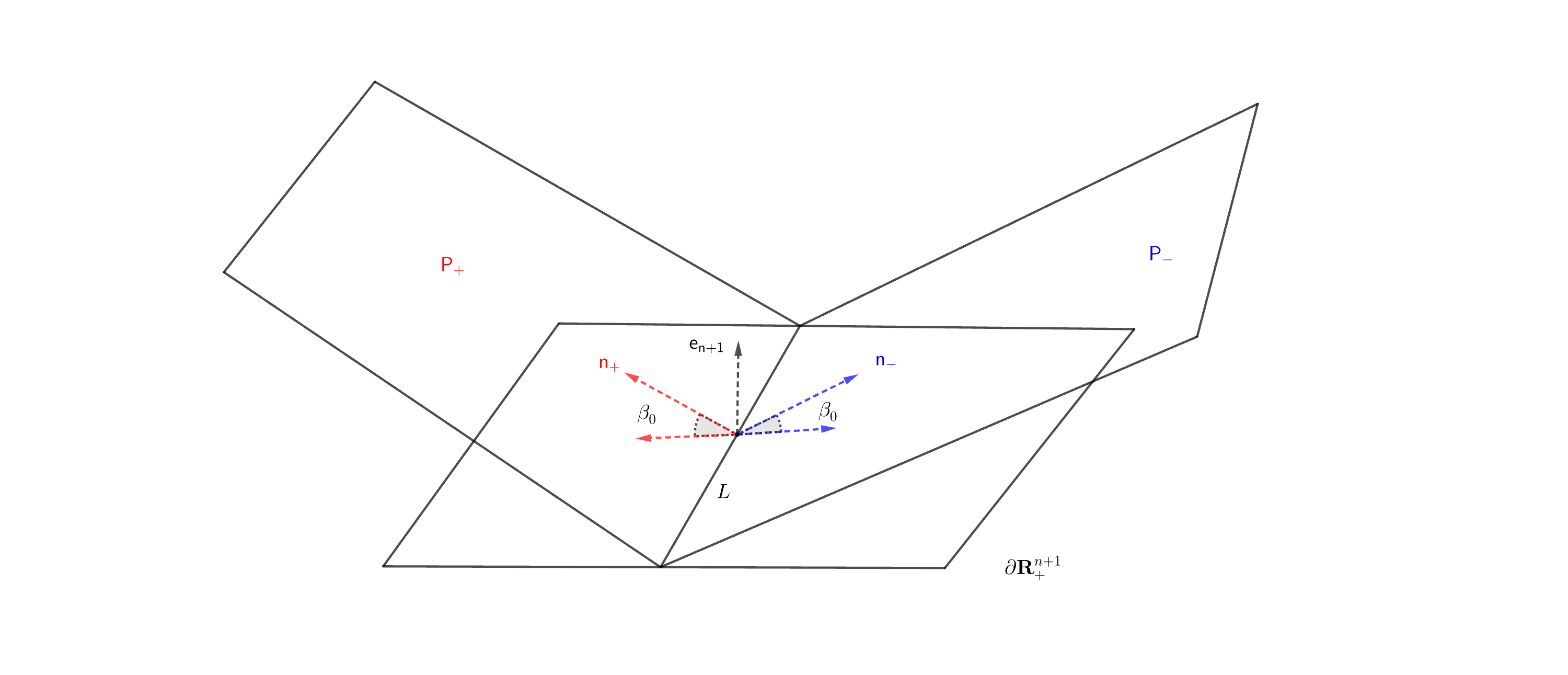}
	\caption{Unions of two planes in $G_{n,\beta_0}(\p\mfR^{n+1}_+)$ with common boundary}
	\label{Fig-1}
\end{figure}

\end{example}

This is an important example to understand the fact that {degenerate capillary phenomenon}
might occur in Definition \ref{Defn:vfld-prescribed-bdry}.
More importantly, it serves as a degenerate limit in compactness result, which can be seen as follows.
\begin{example}[Example \ref{exam:planes} as convergence limit]\label{exam:planes-limit}
Under the notations in Example \ref{exam:planes}, require further that $\beta_0\neq\frac\pi2$.
Put
\eq{
\bar n_{\pm}
=\frac{T_x\p\mfR^{n+1}_+(n_\pm)}{\Abs{T_x\p\mfR^{n+1}_+(n_\pm)}},
}
which are well-defined unit vectors in $\p\mfR^{n+1}_+$ since $\beta_0\in(0,\pi)\setminus\{\frac\pi2\}$.

For every $s>0$ fixed, moving $P_\pm$ and $L$ along the directions $\bar n_\pm$ with distance $s$ respectively, and write the corresponding $m$- and $(m-1)$-planes as $P^s_\pm$ and $L^s_\pm$.
Put
\eq{
V_s
&=\mcH^m\llcorner(P^s_+\cap\mfR^{n+1}_+)\otimes\frac{\de_{P^s_+}}2+\mcH^m\llcorner(P^s_-\cap\mfR^{n+1}_+)\otimes\frac{\de_{P^s_-}}2,\\
\Gamma_s
&=\frac12\mcH^{m-1}\llcorner L^s_+\otimes{\de_{P^s_+}}+\frac12\mcH^{m-1}\llcorner L^s_-\otimes{\de_{P^s_-}}.
}
It is easy to check that $V_s$ has constant prescribed contact angle $\beta_0$ with boundary varifold $\Gamma_s$, along $L^s_+$ and $L^s_+$,
$\mfn_{V_s}(x)=\sin\beta_0e_{n+1}$ and $\mfn_{W_s}(x)=\bar n_\pm$, respectively.
Moreover,
\eq{
V_s\ra V,\quad
\Gamma_s\wsc\Gamma\text{ as }s\searrow0.
}
\end{example}
The next example indicates the fact that $\Gamma$ as a measure-theoretic capillary boundary, might have a complicated structure, in the sense that the Radon probability measure resulting from disintegration at some point could contain information of multiple choices of possible tangent planes.

\begin{example}[Unions of two planes in $G_{m,\beta_0}(\p\mfR^{n+1}_+)$ with distinct boundaries]\label{exam:planes-distinct}
Let $\beta_0$ be a constant function on $\p\mfR^{n+1}_+$ with values in $(0,\pi)\setminus\{\frac\pi2\}$.
Let $P_1,P_2$ be two $m$-planes in $G_{m,\beta_0}(\p\mfR^{n+1}_+)$, such that $L_1\neq L_2$ are the $(m-1)$-planes at which they respectively intersect $\p\mfR^{n+1}_+$ with,
and denote the inwards-pointing unit co-normal of $P_i\cap\mfR^{n+1}_+$ along $L_i$ by $n_i$ for $i=1,2$.

Following the computations conducted in Example \ref{exam:planes},
one can directly check that $V=\mcH^m\llcorner(P_1\cap\mfR^{n+1}_+)\otimes{\de_{P_1}}+\mcH^m\llcorner(P_2\cap\mfR^{n+1}_+)\otimes{\de_{P_2}}$ has constant prescribed contact angle $\beta_0$ with $\Gamma=\mcH^{m-1}\llcorner L_1\otimes{\de_{P_1}}+\mcH^{m-1}\llcorner L_2\otimes\de_{P_2}$.
Moreover,
\eq{
\sigma_\Gamma
=\mcH^{m-1}\llcorner(L_1\cup L_2),\quad
\sigma_V^\perp
=\sin\beta_0\mcH^{m-1}\llcorner(L_1\cup L_2).
}
Next we check the Radon probability measure $\Gamma^x$ resulting from disintegration of $\Gamma$.

Notice that $L_1\cap L_2$ is an $(m-2)$-plane in $\p\mfR^{n+1}_+$, and it is easy to see that for any $x\in L_1\setminus(L_1\cap L_2)$, $\Gamma^x=\de_{P_1}$; also for any $x\in L_2\setminus(L_1\cap L_2)$, $\Gamma^x=\de_{P_2}$.
Now for any $x\in L_1\cap L_2$, by using \cite[Lemma 38.4]{Simon83} for $\Gamma$ at $x$, we find that $\Gamma^x=\frac{\de_{P_1}+\de_{P_2}}2$.

See Figure \ref{Fig-2} (in co-dimension-1 case) for illustration.
\begin{figure}[H]
	\centering
	\includegraphics[width=14cm]{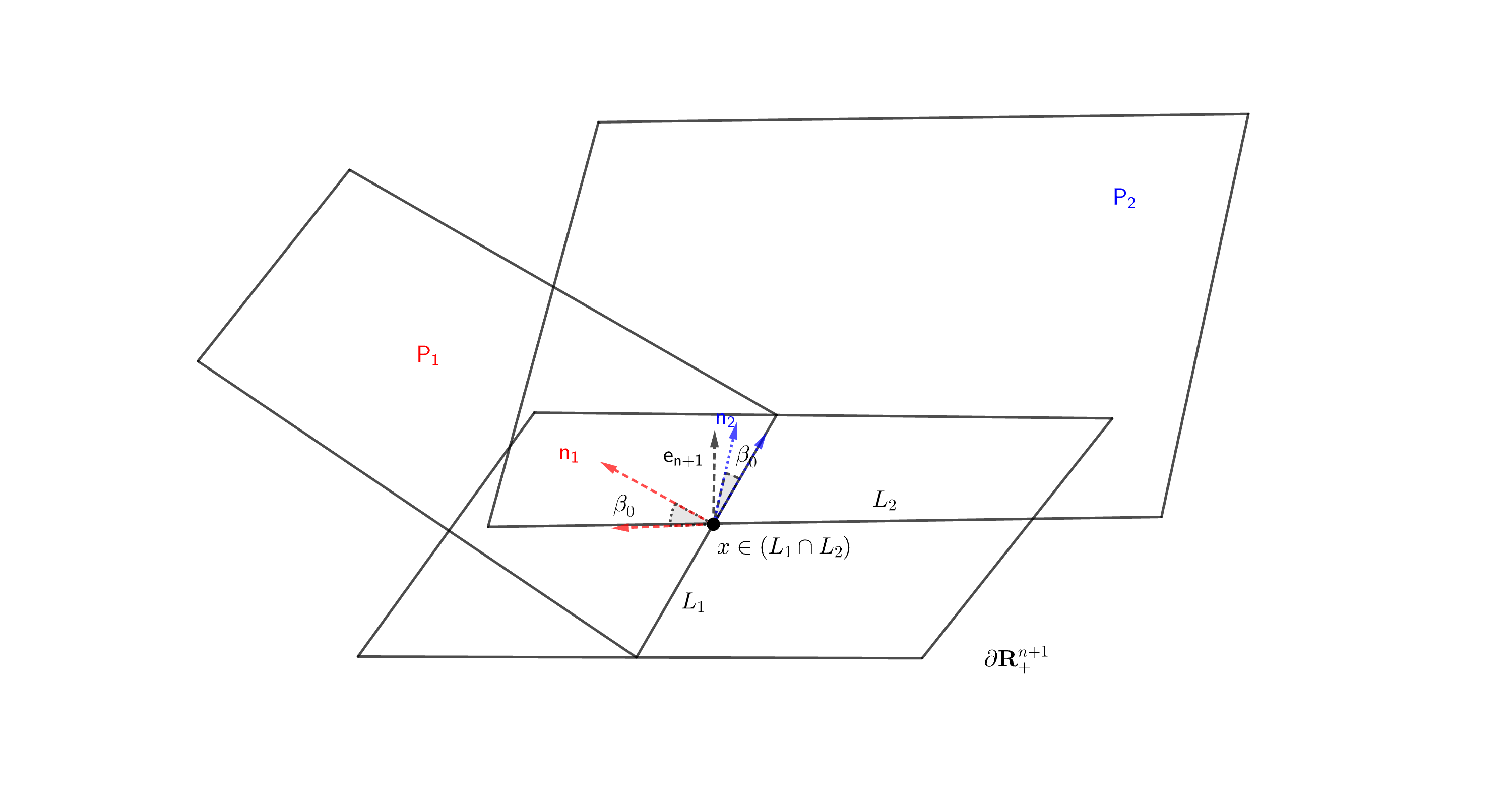}
	\caption{Unions of two planes in $G_{n,\beta_0}(\p\mfR^{n+1}_+)$ with distinct boundaries}
	\label{Fig-2}
\end{figure}
\end{example}

The last example reveals the fact that Definition \ref{Defn:vfld-prescribed-bdry} is weak, in the sense that it admits as many as free boundary components (with possibly non-trivial intersection with the capillary components), since they do not contribute to the first variation when restricted to tangential variation.
\begin{example}[Union of the mutually intersecting $\theta$-cap and free boundary cap]\label{exam:caps}
Let $\beta_0$ in $(0,\pi)\setminus
\{\frac\pi2\}$ and
 $C_{\beta_0}(o_1)$  a $\beta_0$-cap given by
\eq{
C_{\beta_0}(o_1)
\coloneqq\{x\in\overline{\mfR^{n+1}_+}:\Abs{x-\left(o_1-\cos\beta_0e_{n+1}\right)}=1\}.
}
Consider another free boundary cap $C_{\frac\pi2}(o_2)$ which has possibly non-trivial intersection with $C_{\beta_0}(o_1)$.

One can directly check that $V=\mcH^n\llcorner C_{\beta_0}(o_1)\otimes\de_{T_xC_{\beta_0}(o_1)}+\mcH^n\llcorner C_{\frac\pi2}(o_2)\otimes\de_{T_xC_{\frac\pi2}(o_2)}$ has constant prescribed contact angle $\beta_0$ with $\Gamma=\mcH^{n-1}\llcorner\left(\p C_{\beta_0}(o_1)\right)\otimes\de_{T_x\p C_{\beta_0}(o_1)}$.
In fact,
let $n_1,n_2$ denote respectively the inwards pointing unit co-normals to $C_{\beta_0}(o_1), C_{\frac\pi2}(o_2)$ along their boundaries,
computing the first variation and then comparing with \eqref{eq:1st-variation-V_beta}, one finds
\eq{
\mfn_V=n_1,\quad
\sigma_\Gamma
=\mcH^{n-1}\llcorner\p C_{\beta_0}(o_1),
\quad\sigma_V^\perp=\sin\beta_0\mcH^{m-1}\llcorner\p C_{\beta_0}(o_1)+\mcH^{m-1}\llcorner\p C_{\frac\pi2}(o_2).
}
\end{example}

\subsection{Capillary boundary point}

\begin{remark}\label{Rem:G_pb}
\normalfont
In Definition \ref{Defn:capillary-bdry-point},
({C1}) describes the non-degeneracy and is justified by Example \ref{exam:planes}, see also Remark \ref{Rem:Item-1}.
Note, however, that a small perturbation of Example \ref{exam:planes}  satisfies the definition.
More precisely, under the notations in Example \ref{exam:planes}, for  any small  $\ep>0$, one can check that the varifold
\eq{
V
\coloneqq\mcH^m\llcorner P_+\otimes\frac{(1+\ep)\de_{P_+}}2+\mcH^m\llcorner P_-\otimes\frac{(1-\ep)\de_{P_-}}2
}
has constant prescribed contact angle $\beta_0$ with $\Gamma=\mcH^{m-1}\llcorner L\otimes\frac{(1+\ep)\de_{P_+}+(1-\ep)\de_{P_-}}2$.
Moreover, for any $x\in L$, $\mfn_V(x)=\frac{(1+\ep)n_++(1-\ep)n_-}2=\sin\beta_0e_{n+1}+\ep\bar n_+$, where $\bar n_+$ is the projection of $n_+$ onto $\p\mfR^{n+1}_+$, hence $\cos\beta_0\mfn_W(x)=\ep\bar n_+$, and of course $\abs{\mfn_W(x)}=\ep$ for any $x\in L$.
This shows that Definition \ref{Defn:capillary-bdry-point} is not too restrictive.

Now we discuss ({C2}).
If $\p\S$ is induced by the boundary of a $C^{1,1}$-submanifold $\S$ which intersects $S$ transversally
then ({C2}) always holds.
More precisely, assume $x_0=0\in\p\S$, it is direct to check that
\eq{
\Abs{\na^{\p\S}\left<\frac{x}{\abs{x}},\mfn_W(x)\right>}
\leq \tilde c_{x_0}=\tilde c_{x_0}(m,\p\S,S),
}
where $\na^{\p\S}$ denotes the tangential gradient with respect to $\p\S$ and $\mfn_W$ is the unit co-normal along $\p\S$ (with respect to $S$) and is a Lipschitz map since $\p\S$ is a hypersurface in $S$ of class $C^{1,1}$.
On the other hand, since $\lim_{x\ra x_0=0}\frac{x}{\abs{x}}$ lives in the tangent space $T_{x_0}\p\S$, we have
\eq{
\lim_{x\ra x_0=0}\left<\frac{x}{\abs{x}},\mfn_W(x)\right>=0.
}
Since $\beta\in C^1$ we can assume $0<\abs{\cos\beta(x)}\leq2\abs{\cos\beta(x_0)}$ locally and therefore obtain \eqref{defn:genera-capi-bdry-point-2} for any $x_0\in\p\S$ with $c_{x_0}=c_{x_0}(m,\p\S,S)$.
From this example we see that ({C2}) may relate to the local $C^{1,1}$-property of ${\rm spt}\sigma_\Gamma$.

%
\end{remark}


\subsection{Co-dimension-1}

\begin{proposition}\label{Prop:RV-codim-1}
In Definition \ref{Defn:vfld-co-dim-1}, if $U$ is a set of finite perimeter in $S$, then $V\in{\bf RV}^n_\beta(\Om)$ with multiplicity one rectifiable boundary in the sense of Definition \ref{Defn:rectifiable-bdry}.
\end{proposition}
\begin{proof}
For any compactly supported $\varphi\in\mathfrak{X}_t(\Om)$
by De Giorgi's structure theorem (see e.g., \cite[Theorem 15.9]{Mag12}) we could write
\eq{
-\int_U{\rm div}_S(\cos\beta\varphi)\rd\mcH^n
=\int_{\p^\ast U}\left<\cos\beta(x)\nu_U(x),\varphi(x)\right>\rd\mcH^{n-1}(x),
}
where $\p^\ast U$ is the reduced boundary of $U$, which is locally $\mcH^{n-1}$-rectifiable;
and $\nu_U(x)\in T_xS$ is the (measure-theoretic) inner unit normal, which is perpendicular to the approximate tangent space (an $(n-1)$-affine space) $T_{x}\p^\ast U\subset T_xS$ for $\mcH^{n-1}$-a.e. $x\in\p^\ast U$.
At any such point,
put
\eq{
P_U(x)
\coloneqq T_x\p^\ast U\oplus\left(\sin\beta(x)\nu^S(x)+\cos\beta(x)\nu_U(x)\right).
}
Since $\beta(x)\in(0,\pi)$ and $\nu_U(x)\in T_xS$, it is easy to see that $P_U(x)\in G_{n,\beta}(x)$ with
\eq{
\mfn(x,P_U(x))
=\sin\beta(x)\nu^S(x)+\cos\beta(x)\nu_U(x).
}
Define
$\Gamma=\mcH^{n-1}\llcorner\p^\ast U\otimes\de_{P_U(x)}$ we then find:
for any compactly supported $\varphi\in\mathfrak{X}_t(\Om)$
\eq{
\de V(\varphi)
=&-\int_{\overline\Om}\left<\mfH,\varphi\right>\rd\mu_V-\int_{\p^\ast U}\left<\cos\beta(x)\nu_U(x),\varphi(x)\right>\rd\mcH^{n-1}(x)\\
=&-\int_{\overline\Om}\left<\mfH,\varphi\right>\rd\mu_V-\int_{G_{n,\beta}(S)}\left<\mfn(x,P),\varphi(x)\right>\rd\Gamma(x,P)
}
as desired.
\end{proof}
From the proof of this Proposition we observe that $\mfn_W$, the generalized inwards pointing co-normal with respect to $S$, is exactly $\nu_U$.


\section{Monotonicity inequalities}\label{Sec:4}

\subsection{Interior case}
\begin{proposition}[Monotonicity inequality: interior case]\label{Prop-Mono-Interior}
Let $\Om\subset\mfR^{n+1}$ be a bounded domain of class $C^2$ and $\beta\in C^1(S,(0,\pi))$.
Let $V\in{\bf V}^m_\beta(\Om)\cap{\bf RV}^m(\overline\Om)$ with $\mfH\in L^p(\mu_V)$ for some $p\in[1,\infty)$.
Let $h$ be a non-negative $C^1$-function on $\overline\Om$.
For any $\xi\in\Om$, and for either every $0<r_1<r_2<\infty$ when $h_{\mid_S}\equiv0$, or every $0<r_1<r_2<d_S(\xi)$, there holds
\eq{
\frac{1}{r_1^m}\int_{B_{r_1}(\xi)}h\rd\mu_V
\leq&\frac{1}{r_2^m}\int_{B_{r_2}(\xi)}h\rd\mu_V
+\int_{r_1}^{r_2}\frac1{\rho^m}\int_{B_\rho(\xi)}\left(h\abs{\mfH}+\abs{\na^Vh}\right)\rd\mu_V\rd\rho\\
&-\int_{B_{r_2}(\xi)\setminus B_{r_1}(\xi)} h(x)\frac{\abs{(x-\xi)^\perp}^2}{\abs{x-\xi}^{m+2}}\rd\mu_V\\
\leq&\frac{1}{r_2^m}\int_{B_{r_2}(\xi)}h\rd\mu_V
+\int_{r_1}^{r_2}\frac1{\rho^m}\int_{B_\rho(\xi)}\left(h\abs{\mfH}+\abs{\na^Vh}\right)\rd\mu_V\rd\rho.
}
\end{proposition}
\begin{proof}
We use $\varphi(x)=h(x)\gamma(\frac{x-\xi}\rho)(x-\xi)$ to test the first variation \eqref{eq:1st-variation-V_beta}.
Note that in both cases we have $\varphi_{\mid_S}\equiv0$, thus the first variation simply reads
\eq{
\int_{\overline\Om}{\rm div}_V(\varphi)\rd\mu_V
=-\int_\Om\left<\mfH,\varphi\right>\rd\mu_V.
}
A direct computation gives
\eq{
{\rm div}_V(\varphi)
=&\gamma'(\frac{x-\xi}\rho)h(x)\left(\frac{\abs{x-\xi}}\rho-\frac{\abs{(x-\xi)^\perp}^2}{\rho\abs{x-\xi}}\right)\\
&+m\gamma(\frac{x-\xi}\rho)h(x)+\gamma(\frac{x-\xi}\rho)h(x)\left<\na^Vh(x),x-\xi\right>,
}
and hence
\eq{
\int h\left(\gamma'\frac{\abs{x-\xi}}\rho+m\gamma\right)\rd\mu_V
=&\int h\left(\gamma'\frac{\abs{(x-\xi)^\perp}^2}{\rho\abs{x-\xi}}-\gamma\left<\na^Vh(x),x-\xi\right>\right)\rd\mu_V\\
&-\int_\Om h\gamma\left<\mfH,x-\xi\right>\rd\mu_V\\
\leq&\int h\gamma'\frac{\abs{(x-\xi)^\perp}^2}{\rho\abs{x-\xi}}\rd\mu_V+\rho\int\gamma\left(h\abs{\mfH}+\abs{\na^Vh}\right)\rd\mu_V.
}
Define
\eq{
I(\rho)
=\int h(x)\gamma(\frac{\abs{x-\xi}}\rho)\rd\mu_V,\text{ }
J(\rho)
=\int h(x)\gamma(\frac{\abs{x-\xi}}\rho)\frac{\abs{(x-\xi)^\perp}^2}{\abs{x-\xi}^2}\rd\mu_V,
}
we then have the differential inequality
\eq{
-\rho I'(\rho)+mI(\rho)
\leq-\rho J'(\rho)+\rho\int\gamma\left(h\abs{\mfH}+\abs{\na^Vh}\right)\rd\mu_V,
}
which can be rewritten as
\eq{
I'(\rho)-\frac{m}\rho I(\rho)
\geq J'(\rho)-\int\gamma\left(h\abs{\mfH}+\abs{\na^Vh}\right)\rd\mu_V,
}
or further
\eq{
\frac{\rd}{\rd\rho}\left(\rho^{-m}I(\rho)\right)
\geq\rho^{-m}J'(\rho)-\rho^{-m}\int\gamma\left(h\abs{\mfH}+\abs{\na^Vh}\right)\rd\mu_V.
}
Integrating this differential inequality from $r_1$ to $r_2$ then letting $\gamma$ increase to the indicator function $\chi_{[0,1]}$, also taking into account that $\gamma'\leq0$, we then obtain the desired inequality.

\end{proof}

\begin{remark}\label{Rem-Mono-Interior}
\normalfont
By virtue of Proposition \ref{Prop-Mono-Interior}, when $p>m$ we have for any interior points $\xi\in\Om$ the monotonicity formula and consequently all the nice properties as in \cite[Section 17]{Simon83}.
\end{remark}

\subsection{Boundary case}
\begin{proposition}[Monotonicity inequality: boundary case]\label{Prop:Mono-bdry}
Let $\Om\subset\mfR^{n+1}$ be a bounded domain of class $C^2$ and $\beta\in C^1(S,(0,\pi))$.
Let $V\in{\bf V}^m_\beta(\Om)$ with boundary varifold $\Gamma$, such that $\mfH\in L^p(\mu_V)$ for some $p\in[1,\infty)$.

Let $x_0$ be a capillary boundary point in the sense of
Definition \ref{Defn:capillary-bdry-point}
with corresponding factors $\ep_0,\rho_0,c_0,\tau_{x_0}^V$.
Then there exists a constant $C=C(m,p,S,\ep_0,c_0,\cos\beta(x_0))>0$, such that for any $0<\rho<\min\{\rho_0,\frac{\rho_S}2\}$ (resulting from the last assertion of Proposition \ref{Prop:1st-variation}), there holds
\eq{\label{formu:Monotone-ineq}
&(1+C\rho)\frac{\rd}{\rd\rho}\left(\frac1{\rho^m}\int_{\overline\Om}\gamma(\frac{\abs{x-x_0}}\rho)\rd\mu_V\right)^\frac1p\\
\geq&-\rho^{-\frac{m}p}\left(\frac1p+C\rho\right)\left(\int_{B_\rho(x_1)}\Abs{\mfH+\widetilde H}^p\rd\mu_V\right)^\frac1p
&-C(1+\rho)\left(\frac1{\rho^m}\int_{\overline\Om}\gamma(\frac{\abs{x-x_0}}\rho)\rd\mu_V\right)^\frac1p,
}
where $\gamma$ is a standard smooth cut-off function on $\mfR$ with $\gamma\equiv1$ on $[0,1-t]$, $\gamma\equiv0$ on $[1,\infty)$, $-\frac2t\leq\gamma'\leq0$ on $\mfR$, for $t>0$ arbitrarily small.
    
\end{proposition}
\begin{proof}
After translation we may assume that $x_0=0$.
Our goal is to bound the following derivative from below:
\eq{\label{eq:Monoto-derivative-1}
\frac{\rd}{\rd\rho}\left(\frac1{\rho^m}\int_{\overline\Om}\gamma(\frac{\abs{x}}\rho)\rd\mu_V\right)^\frac1p
=\frac1p\frac{\rd}{\rd\rho}\left(\frac1{\rho^m}\int_{\overline\Om}\gamma(\frac{\abs{x}}\rho)\rd\mu_V\right)\left(\frac1{\rho^m}\int_{\overline\Om}\gamma(\frac{\abs{x}}\rho)\rd\mu_V\right)^\frac{1-p}p.
}
Note that
\eq{
\frac{\rd}{\rd\rho}\left(\frac1{\rho^m}\int_{\overline\Om}\gamma(\frac{\abs{x}}\rho)\rd\mu_V\right)
=-\frac1{\rho^{m+1}}\int_{G_m(\overline\Om)}\left(m\gamma(\frac{\abs{x}}\rho)+\frac{\abs{x}}\rho\gamma'(\frac{\abs{x}}\rho)\right)\rd V(x,P).
}
Consider the test vector field $\varphi(x)=\gamma(\frac{\abs{x}}\rho)x$, a direct computation gives
\eq{
{\rm div}_P\varphi(x)
=m\gamma(\frac{\abs{x}}\rho)+\frac{\abs{x}}\rho\gamma'(\frac{\abs{x}}\rho)\Abs{P(\frac{x}{\abs{x}})}^2,
}
and in turn
\eq{\label{eq:monoto-identity-0}
\frac{\rd}{\rd\rho}\left(\frac1{\rho^m}\int_{\overline\Om}\gamma(\frac{\abs{x}}\rho)\rd\mu_V\right)
=&-\frac1{\rho^{m+1}}\de V(\varphi)\\
&-\frac1{\rho^{m+1}}\int_{G_m(\overline\Om)}\frac{\abs{x}}\rho\gamma'(\frac{\abs{x}}\rho)\Abs{P^\perp(\frac{x}{\abs{x}})}^2\rd V(x,P).
}
Using this $\varphi$ in \eqref{eq:1st-variation-V_beta}, the above equality can be further written as
\eq{\label{eq:monoto-identity}
\frac{\rd}{\rd\rho}\left(\frac1{\rho^m}\int_{\overline\Om}\gamma(\frac{\abs{x}}\rho)\rd\mu_V\right)
=&\frac1{\rho^{m+1}}\int_{\overline\Om}\left<\mfH+\widetilde H,\varphi\right>\rd\mu_V\\
&-\frac1{\rho^{m+1}}\int_S\left<\nu^S,\varphi\right>\rd\sigma_V^\perp-\frac1{\rho^{m+1}}\int_S\left<\cos\beta\mfn_W,\varphi\right>\rd\sigma_\Gamma\\
&-\frac1{\rho^{m+1}}\int_{G_m(\overline\Om)}\frac{\abs{x}}\rho\gamma'(\frac{\abs{x}}\rho)\Abs{P^\perp(\frac{x}{\abs{x}})}^2\rd V(x,P).
}
Again, we wish to bound from below the above equality.
Note that $\gamma'\leq0$, thus we could neglect the last integral, also note that by cut-off $\abs{x}\leq\rho$, therefore
\eq{\label{eq:Monoto-derivative-2}
\frac{\rd}{\rd\rho}&\left(\frac1{\rho^m}\int_{\overline\Om}\gamma(\frac{\abs{x}}\rho)\rd\mu_V\right)
\geq-\frac1{\rho^m}\int_{\overline\Om}\gamma(\frac{\abs{x}}\rho)\Abs{\mfH+\widetilde H}\rd\mu_V\\
&-\frac1{\rho^m}\int_S\gamma(\frac{\abs{x}}\rho)\Abs{\left<\frac{x}{\abs{x}},\nu^S\right>}\rd\sigma_V^\perp-\frac1{\rho^{m+1}}\int_S\gamma(\frac{\abs{x}}\rho)\Abs{\left<{x},\cos\beta(x)\mfn_W(x)\right>}\rd\sigma_\Gamma.
}
For the mean curvature term, since $\mfH\in L^p(\mu_V)$ and $\widetilde H\in L^\infty(\mu_V)$, thus $\mfH+\widetilde H\in L^p(\mu_V)$ and hence by H\"older inequality and the fact that $\gamma\leq1$ we obtain
\eq{\label{ineq:integral-H+tilde-H}
\frac1{\rho^m}\int_{\overline\Om}\gamma(\frac{\abs{x}}\rho)\Abs{\mfH+\widetilde H}\rd\mu_V
\leq\frac1{\rho^m}\norm{\mfH+\widetilde H}_{L^p(B_\rho(0),\mu_V)}\left(\int_{\overline\Om}\gamma(\frac{\abs{x}}\rho)\rd\mu_V\right)^{1-\frac1p}\\
\leq\rho^{-\frac{m}p}\left(\int_{B_\rho(0)}\Abs{\mfH+\widetilde H}^p\rd\mu_V\right)^\frac1p\left(\frac1{\rho^{m}}\int_{\overline\Om}\gamma(\frac{\abs{x}}\rho)\rd\mu_V\right)^{1-\frac1p}.
}
For the term concerning tangential part of the boundary measure, since $x_0=0$ is a capillary boundary point of $V$, 
by condition ({C2})
\eq{\label{ineq:bdry-monoto-tang-estimate-1}
\frac1{\rho^{m+1}}\int_S\gamma(\frac{\abs{x}}\rho)\Abs{\left<{x},\cos\beta(x)\mfn_W(x)\right>}\rd\sigma_\Gamma
\leq\frac{c_0}{\rho^{m-1}}\abs{\cos\beta(0)}\int_S\gamma(\frac{\abs{x}}\rho)\rd\sigma_\Gamma.
}
Then as in the proof of Proposition \ref{Prop:1st-variation}, we test \eqref{eq:1st-variation-V_beta} with $\varphi_1(x)=\gamma(\frac{\abs{x}}\rho)\tau_{x_0}^V$.
In view of \eqref{eq:integral-g(x)-1st-variation} and \eqref{ineq:integral-g-lower-bdd}, we get
\eq{
\frac{\ep_{0}}{\rho^{m-1}}
\int_S\gamma(\frac{\abs{x}}\rho)\rd\sigma_\Gamma
\leq&-\frac1{\rho^{m-1}}
\int_{G_m(\overline\Om)}\gamma'(\frac{\abs{x}}\rho)\frac1\rho\left<P(\frac{x}{\abs{x}}),\tau^V_{x_0}\right>\rd V(x,P)\\
&-\frac1{\rho^{m-1}}\int_{\overline\Om}\gamma(\frac{\abs{x}}\rho)\left<\tau^V_{x_0},\mfH+\widetilde H\right>\rd\mu_V\\
&-\frac1{\rho^{m-1}}\int_S\gamma(\frac{\abs{x}}\rho)\left<\nu^S,\tau^V_{x_0}\right>\rd\sigma_V^\perp.
}
Note that for the term involving $\gamma'$, by definition of $\gamma$ we only consider the points satisfying $1\geq\frac{\abs{x}}\rho\geq(1-t)>\frac12$, hence we further write
\eq{
\frac{\ep_{0}}{\rho^{m-1}}
\int_S\gamma(\frac{\abs{x}}\rho)\rd\sigma_\Gamma
\leq&-\frac2{\rho^{m-1}}
\int_{G_m(\overline\Om)}\gamma'(\frac{\abs{x}}\rho)\frac{\abs{x}}{\rho^2}\rd V(x,P)\\
&+\frac1{\rho^{m-1}}\int_{\overline\Om}\gamma(\frac{\abs{x}}\rho)\Abs{\mfH+\widetilde H}\rd\mu_V+\frac1{\rho^{m-1}}\int_S\gamma(\frac{\abs{x}}\rho)\rd\sigma_V^\perp\\
\leq&\frac2{\rho^{m-1}}\frac{\rd}{\rd\rho}\left(\int_{\overline\Om}\gamma(\frac{\abs{x}}\rho)\rd\mu_V\right)+\frac1{\rho^{m-1}}\int_S\gamma(\frac{\abs{x}}\rho)\rd\sigma_V^\perp\\
&+\rho^\frac{p-m}p\left(\int_{B_\rho(0)}\Abs{\mfH+\widetilde H}^p\rd\mu_V\right)^\frac1p\left(\frac1{\rho^{m}}\int_{\overline\Om}\gamma(\frac{\abs{x}}\rho)\rd\mu_V\right)^{1-\frac1p},
}
where we have used \eqref{ineq:integral-H+tilde-H} in the second inequality.

Observe also that
\eq{
&\frac2{\rho^{m-1}}\frac{\rd}{\rd\rho}\left(\int_{\overline\Om}\gamma(\frac{\abs{x}}\rho)\rd\mu_V\right)-\frac{2(m-1)}{\rho^m}\int_{\overline\Om}\gamma(\frac{\abs{x}}\rho)\rd\mu_V\\
=&2\frac{\rd}{\rd\rho}\left(\frac1{\rho^{m-1}}\int_{\overline\Om}\gamma(\frac{\abs{x}}\rho)\rd\mu_V\right)
=\frac2{\rho^m}\int_{\overline\Om}\gamma(\frac{\abs{x}}\rho)\rd\mu_V+2\rho\frac{\rd}{\rd\rho}\left(\frac1{\rho^m}\int_{\overline\Om}\gamma(\frac{\abs{x}}\rho)\rd\mu_V\right),
}
and in turn
\eq{\label{ineq:bdry-monoto-tang-estimate-2}
\frac{\ep_{0}}{\rho^{m-1}}
\int_S\gamma(\frac{\abs{x}}\rho)\rd\sigma_\Gamma
\leq2\rho\frac{\rd}{\rd\rho}\left(\frac1{\rho^m}\int_{\overline\Om}\gamma(\frac{\abs{x}}\rho)\rd\mu_V\right)
+\frac{2m}{\rho^m}\int_{\overline\Om}\gamma(\frac{\abs{x}}\rho)\rd\mu_V\\
+\rho^\frac{p-m}p\left(\int_{B_\rho(0)}\Abs{\mfH+\widetilde H}^p\rd\mu_V\right)^\frac1p\left(\frac1{\rho^{m}}\int_{\overline\Om}\gamma(\frac{\abs{x}}\rho)\rd\mu_V\right)^{1-\frac1p}
+\frac1{\rho^{m-1}}\int_S\gamma(\frac{\abs{x}}\rho)\rd\sigma_V^\perp.
}
For the normal part of the boundary measure, we test \eqref{eq:1st-variation-V_beta} with the vector field $\varphi_2(x)=\gamma(\frac{\abs{x}}\rho)\na d_S(x)$, as computed in \cite[(4.13), (4.16)]{DeMasi21} one obtains
\eq{\label{eq:Monoto-derivative-3}
\frac1{\rho^m}\int_S\gamma(\frac{\abs{x}}\rho)\Abs{\left<\frac{x}{\abs{x}},\nu^S\right>}\rd\sigma_V^\perp
\leq\frac{c}{\rho^{m-1}}\int_S\gamma(\frac{\abs{x}}\rho)\rd\sigma_V^\perp\\
\leq\frac{2mc+c^2\rho}{\rho^m}\int_{\overline\Om}\gamma(\frac{\abs{x}}\rho)\rd\mu_V
+2c\rho\frac{\rd}{\rd\rho}\left(\frac1{\rho^{m}}\int_{\overline\Om}\gamma(\frac{\abs{x}}\rho)\rd\mu_V\right)\\
+c\rho^\frac{p-m}p\left(\int_{B_\rho(0)}\Abs{\mfH+\widetilde H}^p\rd\mu_V\right)^\frac1p\left(\frac1{\rho^{m}}\int_{\overline\Om}\gamma(\frac{\abs{x}}\rho)\rd\mu_V\right)^{1-\frac1p},
}
which is the desired estimate for the normal part of the boundary measure.
We can then obtain the desired estimate for the tangential part of the boundary measure:
\eq{\label{eq:Monoto-derivative-4}
&\frac1{\abs{\cos\beta(0)}\rho^{m+1}}\int_S\gamma(\frac{\abs{x}}\rho)\Abs{\left<{x},\cos\beta(x)\mfn_W(x)\right>}\rd\sigma_\Gamma
\leq\frac{c_0}{\rho^{m-1}}\int_S\gamma(\frac{\abs{x}}\rho)\rd\sigma_\Gamma\\
\leq&\frac{c_0}{\ep_0}(2+2c)\rho\frac{\rd}{\rd\rho}\left(\frac1{\rho^m}\int_{\overline\Om}\gamma(\frac{\abs{x}}\rho)\rd\mu_V\right)
+\frac{c_0}{\ep_0}\frac{(2m+2mc+c^2\rho)}{\rho^m}\int_{\overline\Om}\gamma(\frac{\abs{x}}\rho)\rd\mu_V\\
&+\frac{c_0}{\ep_0}(c+1)\rho^\frac{p-m}p\left(\int_{B_\rho(0)}\Abs{\mfH+\widetilde H}^p\rd\mu_V\right)^\frac1p\left(\frac1{\rho^{m}}\int_{\overline\Om}\gamma(\frac{\abs{x}}\rho)\rd\mu_V\right)^{1-\frac1p},
}
where we have used \eqref{ineq:bdry-monoto-tang-estimate-1} for the first inequality; and \eqref{ineq:bdry-monoto-tang-estimate-2} together with \eqref{eq:Monoto-derivative-4} for the last one.

Finally, substituting \eqref{eq:Monoto-derivative-4}, \eqref{eq:Monoto-derivative-3}, \eqref{ineq:integral-H+tilde-H} into \eqref{eq:Monoto-derivative-2}, and note that
\eq{
\frac1p\left(\frac1{\rho^m}\int_{\overline\Om}\gamma(\frac{\abs{x}}\rho)\rd\mu_V\right)^\frac{1-p}p
\frac{\rd}{\rd\rho}\left(\frac1{\rho^m}\int_{\overline\Om}\gamma(\frac{\abs{x}}\rho)\rd\mu_V\right)
=\frac{\rd}{\rd\rho}\left(\frac1{\rho^m}\int_{\overline\Om}\gamma(\frac{\abs{x}}\rho)\rd\mu_V\right)^\frac1p,
}
we can then estimate \eqref{eq:Monoto-derivative-1} as follows:
\eq{
\frac{\rd}{\rd\rho}\left(\frac1{\rho^m}\int_{\overline\Om}\gamma(\frac{\abs{x}}\rho)\rd\mu_V\right)^\frac1p
=&\frac1p\frac{\rd}{\rd\rho}\left(\frac1{\rho^m}\int_{\overline\Om}\gamma(\frac{\abs{x}}\rho)\rd\mu_V\right)\left(\frac1{\rho^m}\int_{\overline\Om}\gamma(\frac{\abs{x}}\rho)\rd\mu_V\right)^\frac{1-p}p\\
\geq&-\rho^{-\frac{m}p}\left(\frac1p+C\rho\right)\left(\int_{B_\rho(0)}\Abs{\mfH+\widetilde H}^p\rd\mu_V\right)^\frac1p\\
-C(1+\rho)&\left(\frac1{\rho^m}\int_{\overline\Om}\gamma(\frac{\abs{x}}\rho)\rd\mu_V\right)^\frac1p
-C\rho\frac{\rd}{\rd\rho}\left(\frac1{\rho^m}\int_{\overline\Om}\gamma(\frac{\abs{x}}\rho)\rd\mu_V\right)^\frac1p.
}
Rearranging then we obtain the claimed Monotonicity inequality.
\end{proof}

\subsection{Consequences of monotonicity inequalities}

\begin{corollary}\label{Coro-monoto-p>m}
Let $\Om\subset\mfR^{n+1}$ be a bounded domain of class $C^2$ and $\beta\in C^1(S,(0,\pi))$.
Let $V\in{\bf V}^m_\beta(\Om)$ with boundary varifold $\Gamma$, such that $\mfH\in L^p(\mu_V)$ for some $p\in(m,\infty)$.

Let $x_0$ be a capillary boundary point in the sense of
Definition \ref{Defn:capillary-bdry-point}
with corresponding factors $\ep_0,\rho_0,c_0,\tau_{x_0}^V$.
Then there exists $\Lambda=\Lambda(m,p,S,\norm{\mfH}_{L^p(\mu_V)},\ep_0,c_0,\cos\beta(x_0))>0$ such that the function
\eq{
\rho\mapsto e^{\Lambda\rho}\left(\left(\frac{\mu_V(B_\rho(x_0))}{\rho^m}\right)^\frac1p+\Lambda\rho^\frac{p-m}p\right)
}
is monotone increasing on $[0,\min\{\rho_{0},\frac{\rho_S}2\})$ (resulting from the last assertion of Proposition \ref{Prop:1st-variation}).

Moreover, there exists an increasing function $g_{x_0}:\mfR^+\ra\mfR^+$ such that for any $0<r<\rho<\min\{\rho_{0},\frac{\rho_S}2\}$
\eq{\label{ineq-mu_V-B_r-B_rho}
\frac{\mu_V(B_r(x_0))}{r^m}
\leq\frac{\mu_V(B_{\rho}(x_0))}{\rho^m}+g_{x_0}(\rho),
}
where
\eq{
\lim_{\rho\searrow0}g_{x_0}(\rho)=0.
}
\end{corollary}
\begin{proof}
Rearranging the Monotonicity inequality \eqref{formu:Monotone-ineq} we see that there exists $\Lambda>0$, depending on the stated factors, such that
\eq{
\frac{\rd}{\rd\rho}e^{\Lambda\rho}\left(\left(\frac1{\rho^m}\int_{\overline\Om}\gamma(\frac{\abs{x-x_1}}\rho)\rd\mu_V\right)^\frac1p+\Lambda\rho^{1-\frac{m}p}\right)
\geq0.
}
Letting $\gamma$ increase to the indicator function $\chi_{[0,1]}$ we see that
\eq{
\rho\mapsto e^{\Lambda\rho}\left(\frac{\mu_V(B_\rho(x_1))}{\rho^m}\right)^\frac1p+\Lambda e^{\Lambda\rho}\rho^\frac{p-m}p
}
is increasing as desired.
Since $e^r\leq 1+cr$ on $r\leq r_0$ for some $c=c(r_0)$,
\eqref{ineq-mu_V-B_r-B_rho} then follows as a by-product.

\end{proof}

\begin{corollary}
Let $\Om\subset\mfR^{n+1}$ be a bounded domain of class $C^2$ and $\beta\in C^1(S,(0,\pi))$.
Let $V\in{\bf V}^m_\beta(\Om)$ with boundary varifold $\Gamma$, such that $\mfH\in L^\infty(\mu_V)$.

Let $x_0$ be a capillary boundary point in the sense of
Definition \ref{Defn:capillary-bdry-point}
with corresponding factors $\ep_0,\rho_0,c_0,\tau_{x_0}^V$.
Then there exists $\Lambda=\Lambda(m,p,S,\norm{\mfH}_{L^\infty(\mu_V)},\ep_0,c_0,\cos\beta(x_0))>0$ such that the function
\eq{
\rho\mapsto e^{\Lambda\rho}\frac{\mu_V(B_\rho(x_0))}{\rho^m}
}
is monotone increasing on $[0,\min\{\rho_{0},\frac{\rho_S}2\})$ (resulting from the last assertion of Proposition \ref{Prop:1st-variation}).
\end{corollary}
\begin{proof}
Since $\mfH\in L^\infty(\mu_V)$ the term \eqref{ineq:integral-H+tilde-H} can be simply estimated as follows:
\eq{\label{ineq:integral-H+tilde-H-infty}
\frac1{\rho^m}\int_{\overline\Om}\gamma(\frac{\abs{x}}\rho)\Abs{\mfH+\widetilde H}\rd\mu_V
\leq\frac{\norm{\mfH+\widetilde H}_{L^\infty(\mu_V)}}{\rho^m}\int_{\overline\Om}\gamma(\frac{\abs{x}}\rho)\rd\mu_V.
}
Substituting \eqref{eq:Monoto-derivative-4}, \eqref{eq:Monoto-derivative-3} (with $p=1$),  and \eqref{ineq:integral-H+tilde-H-infty} into \eqref{eq:Monoto-derivative-2}, we obtain
\eq{
\frac{\rd}{\rd\rho}&\left(\frac1{\rho^m}\int_{\overline\Om}\gamma(\frac{\abs{x}}\rho)\rd\mu_V\right)
\geq-(1+C\rho)\norm{\mfH+\widetilde H}_{L^\infty(\mu_V)}\left(\frac1{\rho^m}\int_{\overline\Om}\gamma(\frac{\abs{x}}\rho)\rd\mu_V\right)\\
&-C(1+\rho)\left(\frac1{\rho^m}\int_{\overline\Om}\gamma(\frac{\abs{x}}\rho)\rd\mu_V\right)-C\rho\frac{\rd}{\rd\rho}\left(\frac1{\rho^m}\int_{\overline\Om}\gamma(\frac{\abs{x}}\rho)\rd\mu_V\right),
}
that is,
\eq{
(1+C\rho)\frac{\rd}{\rd\rho}\left(\frac1{\rho^m}\int_{\overline\Om}\gamma(\frac{\abs{x}}\rho)\rd\mu_V\right)
\geq-C(1+\rho)\left(\frac1{\rho^m}\int_{\overline\Om}\gamma(\frac{\abs{x}}\rho)\rd\mu_V\right),
}
rearranging and we get
\eq{
\frac{\rd}{\rd\rho}\left(\frac1{\rho^m}\int_{\overline\Om}\gamma(\frac{\abs{x}}\rho)\rd\mu_V\right)
+\frac{C(1+\rho)}{1+C\rho}\left(\frac1{\rho^m}\int_{\overline\Om}\gamma(\frac{\abs{x}}\rho)\rd\mu_V\right)
\geq0,
}
and hence there exists some $\Lambda>0$, depending on the stated quantities, such that
\eq{
\frac{\rd}{\rd\rho}\left(\frac1{\rho^m}\int_{\overline\Om}\gamma(\frac{\abs{x}}\rho)\rd\mu_V\right)
+\Lambda\left(\frac1{\rho^m}\int_{\overline\Om}\gamma(\frac{\abs{x}}\rho)\rd\mu_V\right)
\geq0.
}
Note that the integrating factor is $e^{\Lambda\rho}$, the assertion then follows easily.
\end{proof}

\begin{corollary}[Existence of density]\label{Cor:density-existence}
Let $\Om\subset\mfR^{n+1}$ be a bounded domain of class $C^2$ and $\beta\in C^1(S,(0,\pi))$.
Let $V\in{\bf V}^m_\beta(\Om)$ with boundary varifold $\Gamma$, such that $\mfH\in L^p(\mu_V)$ for some $p\in(m,\infty)$.
\begin{enumerate}
    \item For any $x_0\in \mathscr{P}_{\rm cb}(\Gamma)$, $\Theta^m(\mu_V,x_0)$ exists.
    \item If $V\in{\bf V}^m_\beta(\Om)\cap{\bf RV}^m(\overline\Om)$, then for any $x\in\Om$, $\Theta^m(\mu_V,x)$ exists.
\end{enumerate}
\end{corollary}
\begin{proof}
The assertions follow from Corollary \ref{Coro-monoto-p>m} and Proposition \ref{Prop-Mono-Interior}.
\end{proof}

\section{Blow-up at capillary boundary points}\label{Sec:5}

In this section we consider the blow-up at boundary points, the framework is as follows:
Let $\Om\subset\mfR^{n+1}$ be a bounded domain of class $C^2$ and $\beta\in C^1(S,(0,\pi))$.
Let $V\in{\bf V}^m_\beta(\Om)$ with boundary varifold $\Gamma$.
For a fixed $x_0\in S$ and a fixed sequence of numbers $r_j\searrow0$ as $j\ra\infty$, we use the notations:
\eq{
\Om_j
\coloneqq\bseta_{x_0,r_j}(\Om),\quad
S_j
\coloneqq\bseta_{x_0,r_j}(S),\\
V_j
\coloneqq(\bseta_{x_0,r_j})_\#V\in{\bf V}^m(\overline\Om_j),
}
where $(f)_\#V$ is the push-forward of varifold through $f$.
Moreover, the induced prescribed contact angle function on $S_j$ is given by
\eq{
\beta_j(x)
\coloneqq\beta(x_0+r_jx),\quad x\in S_j.
}
And the induced subspace of $G_m(S_j)$ is denoted by
\eq{
G_{m,\beta_j}(S_j).
}
In fact, we have the identifications for any $x=\bseta_{x_0,r_j}(y)\in S_j$:
\eq{\label{eq:identification-G(x)-G(y)}
G_{m,\beta_j}(x)
=G_{m,\beta}(\bseta_{x_0,r_j}^{-1}(x))
=G_{m,\beta}(y),
}
and
\eq{\label{eq:identification-mfn}
\mfn(y,P_0)
=\mfn(x,P_0),
}
where $P_0\in G_{m,\beta_j}(x)=G_{m,\beta}(y)$.

For any $x\in S_j$, we denote by $\nu^{S_j}(x)$ the inwards-pointing unit normal to $\Om_j$ at $x\in S_j$.
As $j\ra\infty$ we have in the sense of $C^2$-topology ($C^1$-topology for the function $\beta_j$) the convergence
\eq{\label{eq:blow-up-settings}
\overline\Om_j\ra T_{x_0}\overline\Om,\quad
S_j\ra T_{x_0}S,\quad
\beta_j\ra\beta(x_0),
}
where $\beta(x_0)$ is understood as the constant function of value $\beta(x_0)$ on ${T_{x_0}S}$.
In particular, this yields the following convergences:
\eq{\label{eq:blow-up-settings-bundle}
G_m(S_j)\ra G_m(T_{x_0}S),\quad
G_{m,\beta_j}(S_j)
\ra G_{m,\beta(x_0)}(T_{x_0}S).
}

We define the push-forward of $\Gamma$ through $\bseta_{x_0,r_j}$, denoted by $\Gamma_j$, to be the Radon measure on $G_{m,\beta_j}(S_j)$ such that $\forall\phi\in C_c(G_{m,\beta_j}(S_j))$
\eq{\label{defn:Gamma_j}
\int_{G_{m,\beta_j}(S_j)}\phi(x,P)\rd\Gamma_j(x,P)
\coloneqq&r_j\int_{G_{m,\beta}(S)}{\rm J}_P(\bseta_{x_0,r_j})(y)\phi(\bseta_{x_0,r_j}(y),{\rm d}\bseta_{x_0,r_j}(P))\rd\Gamma(y,P)\\
=&r_j^{1-m}\int_{G_{m,\beta}(S)}\phi(\frac{y-x_0}{r_j},P)\rd\Gamma(y,P).
}
We have the standard disintegration
\eq{
\Gamma_j=\sigma_{\Gamma_j}\otimes\Gamma_j^x,
}
where $\sigma_{\Gamma_j}=(\pi)_\ast\Gamma_j$ and $\Gamma_j^x$ is the Radon probability measure on $G_{m,\beta_j}(x)$ for $\sigma_{\Gamma_j}$-a.e. $x$.
On the other hand,
in light of \eqref{defn:n_V-intro} we define for $\sigma_{\Gamma_j}$-a.e. $x$
\eq{
\mfn_{V_j}(x)
\coloneqq\int_{G_{m,\beta_j}(x)}\mfn(x,P)\rd\Gamma_j^x(P),\text{ and }
\cos\beta(x)\mfn_{W_j}(x)
\coloneqq T_xS\left(\mfn_{V_j}(x)\right).
}

\subsection{Properties of blow-up limits}

\begin{lemma}\label{Lem:Gamma_j-properties}
Under the above notations, the following properties hold:
\begin{enumerate}
    \item the measure $\sigma_{\Gamma_j}$ is given by
\eq{\label{eq:sigma_Gamma_j}
\sigma_{\Gamma_j}
=\frac1{r_j^{m-1}}(\bseta_{x_0,r_j})_\ast\sigma_\Gamma.
}
    \item For $\sigma_{\Gamma_j}$-a.e. $x=\bseta_{x_0,r_j}(y)\in S_j$, the Radon probability measure $\Gamma_j^x$ is given by
\eq{
\Gamma_j^x
=\Gamma^y,
}
as measures on $G_{m,\beta_j}(x)=G_{m,\beta}(y)$.
    \item For $\sigma_{\Gamma_j}$-a.e. $x=\bseta_{x_0,r_j}(y)\in S_j$, the generalized inwards pointing co-normal to $V_j$ is given by
\eq{
\mfn_{V_j}(x)
=\mfn_V(y).
}
\end{enumerate}
\end{lemma}
\begin{proof}

For $\phi(x,P)=\varphi(x)$,
\eqref{defn:Gamma_j} simplifies as
\eq{
\int_{S_j}\varphi(x)\rd\sigma_{\Gamma_j}(x)
=r_j^{1-m}\int_S\varphi(\frac{y-x_0}{r_j})\rd\sigma_\Gamma(y).
}
Changing of variables we see
\eq{
\int_S\varphi(\frac{y-x_0}{r_j})\rd\sigma_\Gamma(y)
=\int_{S_j}\varphi(x)\rd\left((\bseta_{x_0,r_j})_\ast\sigma_\Gamma\right)(x),
}
which proves ({\bf 1}).

For $\phi(x,P)$ in \eqref{defn:Gamma_j}, we use disintegration, \eqref{eq:identification-G(x)-G(y)}, and change of variables to find
\eq{\label{eq:convergence-gamma_infty}
\int_{G_{m,\beta_j}(S_j)}\phi(x,P)\rd\Gamma_j(x,P)
=&r_j^{1-m}\int_{G_{m,\beta}(S)}\phi(\frac{y-x_0}{r_j},P)\rd\Gamma(y,P)\\
=&r_j^{1-m}\int_S\int_{G_{m,\beta}(y)}\phi(\frac{y-x_0}{r_j},P)\rd\Gamma^y(P)\rd\sigma_\Gamma(y)\\
=&r_j^{1-m}\int_{S_j}\int_{G_{m,\beta_j}(x)}\phi(x,P)\rd\Gamma^{\bseta_{x_0,r_j}^{-1}(x)}(P)\rd\left((\bseta_{x_0,r_j})_\ast\sigma_\Gamma\right)(x)\\
=&\int_{S_j}\int_{G_{m,\beta_j}(x)}\phi(x,P)\rd\Gamma^{\bseta_{x_0,r_j}^{-1}(x)}(P)\rd\sigma_{\Gamma_j}(x),
}
where we have used \eqref{eq:sigma_Gamma_j} in the last equality.
In particular, this shows that the following two measures are the same:
\eq{
\Gamma_j
=\sigma_{\Gamma_j}\otimes\Gamma_j^x,\quad
\tilde\Gamma_j\coloneqq\sigma_{\Gamma_j}\otimes\Gamma^{\bseta_{x_0,r_j}^{-1}(x)},
}
which proves ({\bf 2}).


({\bf 3}) is a direct consequence of ({\bf 2}), \eqref{eq:identification-G(x)-G(y)}, and the definition \eqref{defn:n_V-intro}, the proof is completed.
\end{proof}

Next we prove that each $V_j\in {\bf V}^m_{\beta_j}(\Om_j)$, with boundary varifold given by $\Gamma_j$ defined above, so that by virtue of Proposition \ref{Prop:1st-variation} there exist boundary measures
\eq{
\sigma_j^T
\coloneqq\sigma_{\Gamma_j},\quad
\sigma_j^\perp
\coloneqq\sigma_{V_j}^\perp.
}

\begin{lemma}[Blow-up sequence]\label{Lem:blow-up-seq}
Under the notations above, for $\sigma_\Gamma$-a.e. $x_0\in \mathscr{P}_{\rm cb}(\Gamma)$, the following statements hold:
\begin{enumerate}
    \item
    For $j\in\mbN$, $V_j\in{\bf V}^m_{\beta_j}(\Om_j)$ with boundary varifold $\Gamma_j$ and satisfies
\eq{
\mfH_j(x)
=r_j\mfH(x_0+r_jx),\quad
\widetilde H_j(x)
=r_j\widetilde H(x_0+r_jx),\\
\sigma_j^\perp
\coloneqq\sigma_{V_j}^\perp
=\frac1{r_j^{m-1}}(\bseta_{x_0,r_j})_\ast\sigma_V^\perp,\quad
\sigma_j^T
\coloneqq\sigma_{\Gamma_j}
=\frac1{r_j^{m-1}}(\bseta_{x_0,r_j})_\ast\sigma_\Gamma.
}
    \item Passing to a subsequence $\{j_k\}_{k\in\mbN}$, there exist $\mcC\in{\bf V}^m(\mfR^{n+1})$ and a Radon measure $\Gamma_\infty$ on $G_m(\mfR^{n+1})$, such that
\eq{
V_{j_k}\ra\mcC,\quad
\Gamma_{j_k}\wsc\Gamma_\infty.
}
Moreover, $\mcC$ is supported on $T_{x_0}\overline\Om$, $\Gamma_\infty$ is supported on $G_{m,\beta(x_0)}(T_{x_0}S)$, where $\beta(x_0)$ is understood as the constant function of value $\beta(x_0)$ on $T_{x_0}S$.
    \item $\mcC$ has prescribed contact angle $\beta(x_0)$ intersecting $T_{x_0}S$ along $\Gamma_\infty$ in the sense of Definition \ref{Defn:vfld-prescribed-bdry}, and satisfies \eqref{eq:1st-variation-V_beta} with
\eq{\label{eq:H_C}
\mfH_\mcC=0,\quad
\widetilde H_\mcC=0.
}
The disintegration $\Gamma_\infty=\sigma_{\Gamma_\infty}\otimes\Gamma_\infty^x$ is characterized by:
\eq{\label{eq:limit-sigma_jk}
\sigma_{\Gamma_\infty}
=\lim_{k\ra\infty}\sigma_{\Gamma_{j_k}}
}
as measures;
for $\sigma_{\Gamma_\infty}$-a.e. $x$,
\eq{\label{eq:limit-gamma^x_infty}
\Gamma_\infty^x
=\Gamma^{x_0}
}
as measures on $G_{m,\beta(x_0)}(x)$, where
\eq{
G_{m,\beta(x_0)}(x)\cong G_{m,\beta(x_0)}(0)\cong G_{m,\beta}(x_0),\quad
\forall x\in T_{x_0}S.
}
Moreover, for $\sigma_{\Gamma_\infty}$-a.e. $x$
\eq{\label{eq:n_mcC}
\mfn_\mcC(x)
=\mfn_V(x_0).
}
For the normal part of boundary measure, one has
\eq{\label{eq:limit-sigma_j^perp}
\sigma_\mcC^\perp
=\lim_{k\ra\infty}\sigma_j^\perp.
}
    \item For $\sigma_{\Gamma_\infty}$-a.e. $x$, there holds
\eq{\label{eq:x-n_W(x_0)}
\left<x,\mfn_W(x_0)\right>=0.
}
    \item If in addition $\Theta^m(\mu_V,x)\geq a>0$ for $\mu_V$-a.e. $x\in\overline\Om$, then $\mcC$ is a rectifiable cone with
\eq{
\Theta^m(\mu_\mcC,x)\geq a,\text{ for }\mu_C\text{-a.e. }x,
}
    and $\sigma_{\Gamma_\infty}$ is scaling invariant in the sense that for every $\rho>0$
\eq{
\frac1{\rho^{m-1}}(\bseta_{0,\rho})_\ast\sigma_{\Gamma_\infty}
=\sigma_{\Gamma_\infty}.
}

\end{enumerate}
\end{lemma}
\begin{proof}
After translation and rotation we may assume that $x_0=0\in S$, $T_{x_0}\overline\Om=\{x_{n+1}\geq0\}$ is the upper half-space, which we denote by $\overline{\mfR^{n+1}_+}$, and $T_{x_0}S=\{x_{n+1}=0\}=\p\mfR^{n+1}_+$.
We also assume that $\Gamma^{x_0}$ is a well-defined Radon probability measure resulting from disintegration, which is true for $\sigma_\Gamma$-a.e. $x$.

({\bf 1}):
Observe that for every $\varphi\in\mathfrak{X}_t(\Om_j)$ the vector field $\tilde \varphi(y)=\varphi(\frac{y}{r_j})\in\mathfrak{X}_t(\Om)$.
Testing \eqref{defn:1st-vairation-intro} with such $\tilde\varphi$ we obtain
\eq{
\de V_j(\varphi)
=&\int_{G_m(\overline\Om_j)}{\rm div}_P\varphi(x)\rd V_j(x,P)
=\frac1{r_j^{m-1}}\int_{G_m(\overline\Om)}{\rm div}_P\varphi(\frac{y}{r_j})\rd V(y,P)\\
=&-\frac1{r_j^{m-1}}\int_{\overline\Om}\left<\mfH(y),\varphi(\frac{y}{r_j})\right>\rd\mu_{V}(y)-\frac1{r_j^{m-1}}\int_{G_{m,\beta}(S)}\left<\mfn(y,P),\varphi(\frac{y}{r_j})\right>\rd\Gamma(y,P).
}
Changing of variables and recalling Lemma \ref{Lem:Gamma_j-properties} ({\bf 1})({\bf 2}) we get
\eq{\label{eq:1st-variation-V_j-Gamma_j}
\de V_j(\varphi)
=-\int_{\overline\Om_j}\left<\varphi(x),r_j\mfH(r_jx)\right>\rd\mu_{V_j}(x)-\int_{G_{m,\beta_j}(S_j)}\left<\mfn(x,P),\varphi(x)\right>\rd\Gamma_j(x,P),
}
which shows that $V_j\in{\bf V}^m_{\beta_j}(\Om_j)$ with $\mfH_j(x)=r_j\mfH(r_jx)$ and boundary $\Gamma_j$.
By Proposition \ref{Prop:1st-variation} each $V_j$ has bounded first variation and there are correspondingly
\eq{
\widetilde H_j\in L^\infty(\mu_{V_j}),\quad
\sigma_j^\perp
\coloneqq\sigma_{V_j}^\perp,\quad
\sigma_j^T
\coloneqq\sigma_{\Gamma_j},
}
and it is easy to see that
\eq{
\widetilde H_j(x)
=r_j\widetilde H(r_jx),\quad
\mu_{V_j}\text{-a.e. } x\in S_j.
}
In particular, for any test function $\varphi\in\mathfrak{X}(\Om_j)$ we have
\eq{
\de V_j(\varphi)
=&\frac1{r_j^{m-1}}\int_{G_m(\overline\Om)}{\rm div}_P\varphi(\frac{y}{r_j})\rd V(y,P)\\
=&-\frac1{r_j^{m-1}}\int_{\overline\Om}\left<\mfH(y)+\widetilde H(y),\varphi(\frac{y}{r_j})\right>\rd\mu_V(x)\\
&-\frac1{r_j^{m-1}}\int_S\left<\nu^S(y),\varphi(\frac{y}{r_j})\right>\rd\sigma_V^\perp
-\frac1{r_j^{m-1}}\int_S\left<\mfn_V(y),\varphi^T(\frac{y}{r_j})\right>\rd\sigma_\Gamma\\
=&-\int_{\overline\Om_j}\left<\mfH_j+\widetilde H_j,\varphi\right>\rd\mu_{V_j}\\
&-\frac1{r_j^{m-1}}\int_{S_j}\left<\nu^{S_j},\varphi\right>\rd(\bseta_{0,r_j})_\ast\sigma_V^\perp
-\frac1{r_j^{m-1}}\int_{S_j}\left<\mfn_{V_j},\varphi^T\right>\rd(\bseta_{0,r_j})_\ast\sigma_\Gamma,
}
where we have used Lemma \ref{Lem:Gamma_j-properties} ({\bf 3}).
This shows that
\eq{
\sigma_j^\perp
\coloneqq\sigma_{V_j}^\perp
=\frac1{r_j^{m-1}}(\bseta_{0,r_j})_\ast\sigma_V^\perp,\quad
\sigma_j^T
\coloneqq\sigma_{\Gamma_j}
=\frac1{r_j^{m-1}}(\bseta_{0,r_j})_\ast\sigma_\Gamma.
}

Next we study the blow-up limits at capillary boundary points.

({\bf 2}): By Corollary \ref{Coro-monoto-p>m}, especially \eqref{ineq-mu_V-B_r-B_rho}, there exists some $C>0$ such that for $j$ large (or for $r_j$ sufficiently small)
\eq{\label{ineq:mu_V_j-bound}
\mu_{V_j}(B_1(0))
=\norm{(\bseta_{0,r_j})_\#V}(B_1(0))
=\frac{\mu_{V}(B_{r_j}(0))}{r_j^m}
\leq C.
}
Using compactness of Radon measures we see that there exist a subsequence $\{r_{j_k}\}$ and some $\mcC\in{\bf V}^m(\mfR^{n+1})$ such that
\eq{
V_{j_k}\ra\mcC
}
as varifolds.
Moreover, recalling \eqref{eq:blow-up-settings} we have that ${\rm spt}\mu_\mcC\subset\overline{\mfR^{n+1}_+}$.

On the other hand, in view of \eqref{eq:sigma_Gamma_j} and \eqref{esti:local-tangent-improved} we get
\eq{\label{ineq:sigma_Gamma-ratio}
\sigma_{\Gamma_j}(B_1(0))
=&\frac{\sigma_\Gamma(B_{r_j}(0))}{r_j^{m-1}}
\leq\frac{\max\{C_0,1\}}{\ep_{x_0}}\left(\frac{\mu_V(B_{4r_j}(0))}{r_j^m}+\int_{B_{4r_j}(0)}\frac{\abs{\mfH}}{r_j^{m-1}}\rd\mu_V\right)\\
\leq&C\left(\frac{\mu_V(B_{4r_j}(0))}{r_j^m}+r_j^{1-\frac{m}p}\left(\int_{B_{4r_j}(0)}\abs{H}^p\rd\mu_V\right)^\frac1p\left(\frac{\mu_V(B_{4r_j}(0))}{(4r_j)^m}\right)^\frac{p-1}p\right)\\
\leq&C,
}
where we have used again \eqref{ineq-mu_V-B_r-B_rho} in the last inequality, which holds for $j$ sufficiently large.
Again, by compactness of Radon measures, after passing to a subsequence, still denoted by ${r_{j_k}}$, there exists some Radon measure $\Gamma_\infty$ on $G_m(\mfR^{n+1})$ such that
\eq{
\Gamma_{j_k}
\wsc\Gamma_\infty
}
as measures.
Moreover, recalling \eqref{eq:blow-up-settings-bundle} we have that ${\rm spt}\Gamma_\infty\subset G_{m,\beta(0)}(\p\mfR^{n+1}_+)$.

Then we test the first variation of $\mcC$ by any vector field $\varphi\in\mathfrak{X}_t(\mfR^{n+1}_+)$ with compact support.
By \eqref{eq:blow-up-settings} there exists a sequence of vector field $\varphi_{j_k}\in\mathfrak{X}_t(\Om_{j_k})$ with compact support such that $\varphi_{j_k}\ra\varphi$ in the sense of $C^1$-topology.
In particular, using \eqref{eq:1st-variation-V_j-Gamma_j} and recalling \eqref{eq:blow-up-settings-bundle} we obtain
\eq{\label{eq:first-variation-C}
&\de\mcC(\varphi)
=\lim_{k\ra\infty}\de V_{j_k}(\varphi_{j_k})\\
=&\lim_{k\ra\infty}\left(\int_{\overline\Om_{j_k}}\left<\varphi_{j_k},\mfH_{j_k}\right>(x)\rd\mu_{V_{j_k}}(x)-\int_{G_{m,\beta_{j_k}}(S_{j_k})}\left<\mfn(x,P),\varphi_{j_k}(x)\right>\rd\Gamma_{j_k}(x,P)\right)\\
=&0-\int_{G_{m,\beta(0)}(\p\mfR^{n+1}_+)}\left<\mfn(x,P),\varphi(x)\right>\rd\Gamma_\infty(x,P),
}
where we have used in the last equality \eqref{eq:blow-up-settings-bundle} and the fact that as $k\ra\infty$
\eq{
\norm{\mfH_{j_k}}_{L^1(B_1(0),\mu_{V_{j_k}})}
\leq&\left(\int_{B_1(0)}\abs{\mfH_{j_k}}^p\rd\mu_{V_{j_k}}\right)^\frac1p\left(\mu_{V_{j_k}}(B_1(0))\right)^\frac{p-1}p\\
\overset{\eqref{ineq:mu_V_j-bound}}{\leq}&C\left(r_{j_k}^{p-m}\int_{B_{r_{j_k}}(0)}\abs{\mfH}^p\rd\mu_{V}\right)^\frac1p\ra0.
}
In particular, this proves that $\mcC\in{\bf V}^m_{\beta(0)}(\mfR^{n+1}_+)$ in the sense of Definition \ref{Defn:vfld-prescribed-bdry} with boundary varifold given by $\Gamma_\infty$.
By disintegration we write
\eq{
\Gamma_\infty
=\sigma_{\Gamma_\infty}\otimes\Gamma_\infty^x,
}
where $\Gamma_\infty^x$ is the Radon probability measure on $G_{m,\beta_0}(x)$ for $\sigma_{\Gamma_\infty}$-a.e. $x\in\p\mfR^{n+1}_+$.
Note that the inwards pointing unit normal field on $\p\mfR^{n+1}_+$ is the constant vector field $e_{n+1}$, therefore $G_{m,\beta_0}(x)= G_{m,\beta_0}(0)$ for any $x\in\p\mfR^{n+1}_+$.

{({\bf 3})}:
Now we show that $\sigma_{\Gamma_\infty}$ and $\Gamma_\infty^x$ are blow-up limits as well.
First, for any test function $\phi\in C_c(\mfR^{n+1}\times\mfR^{(n+1)^2})$, we have by virtue of $\Gamma_{j_k}\wsc\Gamma_\infty$ that
\eq{
\int_{G_{m,\beta(0)}(\p\mfR^{n+1}_+)}\phi(x,P)\rd\Gamma_\infty(x,P)
=\lim_{k\ra\infty}\int_{G_{m,\beta_{j_k}}(S_j)}\phi(y,P)\rd\Gamma_{j_k}(y,P).
}
For $\phi(x,P)=\varphi(x)$, by disintegrations this simplifies as
\eq{
\int_{\p\mfR^{n+1}}\varphi(x)\rd\sigma_{\Gamma_\infty}(x)
=\lim_{k\ra\infty}\int_{S_{j_k}}\varphi(y)\rd\sigma_{\Gamma_{j_k}}(y),
}
which proves \eqref{eq:limit-sigma_jk}.

Then we show that \eqref{eq:limit-gamma^x_infty} holds for $\sigma_\Gamma$-a.e. $x$, which is inspired by \cite[Proposition 9]{DeLOF03}, see also \cite[Lemma 2.3]{DePDeRG18}.

Precisely, we first fix an arbitrary $\rho>0$.
Let $\mathscr{S}\subset C_c^0(\mfR^{(n+1)^2})$ be a countable family of functions which is dense in $C_c^0(\mfR^{(n+1)^2})$.
For every $\psi\in\mathscr{S}$, define the function $f_\psi\in L^1(\mfR^{n+1})$ by
\eq{
f_\psi(x)\coloneqq\int\psi(P)\rd\Gamma^x(P),
}
where $\Gamma^x$ is resulting from the disintegration $\Gamma=\sigma_\Gamma\otimes\Gamma^x$, and is a well-defined Radon probability measure on $G({m,n+1})$ for $\sigma_\Gamma$-a.e. $x$.
Then we put $S\coloneqq\bigcap_{\psi\in\mathscr{S}}S_\psi$, where
\eq{
S_\psi
\coloneqq\left\{x\in\mfR^{n+1}:x\text{ is a }\text{Lebesgue point for }f_\psi\text{ with respect to }\sigma_\Gamma\right\}.
}
Clearly, $\sigma_\Gamma(\mfR^{n+1}\setminus S)=0$.
Thus, to prove that \eqref{eq:limit-gamma^x_infty} holds for $\sigma_\Gamma$-a.e. $x_0$, it suffices to show that \eqref{eq:limit-gamma^x_infty} holds for every $x_0\in S$.
To this end, we fix an arbitrary $x_0\in S$, and after translation we may assume that $x_0=0$.
For every $\phi\in C_c(\mfR^{n+1}\times\mfR^{(n+1)^2})$ and $r>0$, we define
\eq{
\mathscr{F}(\phi,r)
\coloneqq r^{1-m}\left\{\int_{B_{r\rho}(0)}\int\phi(\frac{x}{r\rho},P)\rd\Gamma^0(P)\rd\sigma_\Gamma(x)-\int_{B_{r\rho}(0)}\int\phi(\frac{x}{r\rho},P)\rd\Gamma^x(P)\rd\sigma_\Gamma(x)\right\}.
}
For test function $\phi(x,P)=\varphi(x)\psi(P)$, where $\varphi\in C_0(\mfR^{n+1})$, and $\psi\in\mathscr{S}$, we have
\eq{
\mathscr{F}(\varphi\psi,r)
=r^{1-m}\int_{B_{r\rho}(0)}\varphi(\frac{x}{r\rho})\left(f_\psi(0)-f_\psi(x)\right)\rd\sigma_\Gamma(x),
}
and hence for $\{j_k\}_{k\in\mbN}$ resulting from ({\bf 2}), we have (recall that $x_0=0$ is a capillary boundary point, and hence we write accordingly $\ep_{x_0},\rho_{x_0},c_{x_0}$.
For $k$ sufficiently large we have $r_{j_k}\rho<\rho_{x_0}$)
\eq{
\Abs{\mathscr{F}(\varphi\psi,r_{j_k})}
\leq&\abs{\varphi}_{L^\infty(\mfR^{n+1})}r_{j_k}^{1-m}\int_{B_{r_{j_k}\rho}(0)}\abs{f_\psi(0)-f_\psi(x)}\rd\sigma_\Gamma(x)\\
=&\abs{\varphi}_{L^\infty(\mfR^{n+1})}\rho^{m-1}\frac{\sigma_\Gamma(B_{r_{j_k}\rho}(0))}{(r_{j_k}\rho)^{m-1}}\frac1{\sigma_\Gamma(B_{r_{j_k}\rho}(0))}\int_{B_{r_{j_k}\rho}(0)}\abs{f_\psi(0)-f_\psi(x)}\rd\sigma_\Gamma(x)\\
\leq&C\rho^{m-1}\abs{\varphi}_{L^\infty(\mfR^{n+1})}\frac1{\sigma_\Gamma(B_{r_{j_k}\rho}(0))}\int_{B_{r_{j_k}\rho}(0)}\abs{f_\psi(0)-f_\psi(x)}\rd\sigma_\Gamma(x),
}
where we have used \eqref{ineq:sigma_Gamma-ratio} for the last inequality.
Since $x_0=0$ is a Lebesgue point for $f_\psi$ with respect to $\sigma_\Gamma$, this infers that $\lim_{k\ra\infty}\mathscr{F}(\varphi\psi,r_{j_k})=0$.
Thereby, for any function $\phi\in C_c(\mfR^{n+1}\times\mfR^{(n+1)^2})$ of the form
\eq{
\phi
=\sum_i\varphi_i\psi_i,\quad\text{where }\varphi_i\in C_c(\mfR^{n+1}),\text{ and }\psi_i\in\mathscr{S},
}
we have $\lim_{k\ra\infty}\mathscr{F}(\phi,r_{j_k})=0$.
Moreover, as these functions are dense in $C_c(\mfR^{n+1}\times\mfR^{(n+1)^2})$, and thanks again to \eqref{ineq:sigma_Gamma-ratio}, it is easy to see that
\eq{
\Abs{\mathscr{F}(\phi,r_{j_k})-\mathscr{F}(\tilde\phi,r_{j_k})}
\leq2\rho^{m-1}C\Abs{\phi-\tilde\phi}_{L^\infty(\mfR^{n+1})},\quad\forall\phi,\tilde\phi\in C_c(\mfR^{n+1}\times\mfR^{(n+1)^2}),
}
we thus find
\eq{\label{eq:lim-mathscr-F=0}
\lim_{k\ra\infty}\mathscr{F}(\phi,r_{j_k})=0,\quad\forall\phi\in C_c(\mfR^{n+1}\times\mfR^{(n+1)^2}).
}

Now consider the standard cut-off function $\gamma$ on $\mfR$ with $\gamma\equiv1$ on $[0,1-t]$, $\gamma\equiv0$ on $[1,\infty)$, $-\frac2t\leq\gamma'\leq0$ on $\mfR$ for $t>0$ arbitrarily small.
Observe that for any $\phi\in C_c(\mfR^{n+1}\times\mfR^{(n+1)^2})$,
if we put $\tilde\gamma(x,P)=\gamma(\abs{x})$, $\tilde\phi(x,P)=\phi(\rho x,P)$, then
\eq{
\mathscr{F}(\tilde\gamma\tilde\phi,r_{j_k})
=&r_{j_k}^{1-m}\int_{S}\gamma(\frac{\abs{x}}{r_{j_k}\rho})\int\phi(\frac{x}{r_{j_k}},P)\rd\Gamma^0(P)\rd\sigma_\Gamma(x)\\
&-r_{j_k}^{1-m}\int_S\gamma(\frac{\abs{x}}{r_{j_k}\rho})\int\phi(\frac{x}{r_{j_k}},P)\rd\Gamma^x(P)\rd\sigma_\Gamma(x)\\
=&r_{j_k}^{1-m}\int_{S_{j_k}}\gamma(\frac{\abs{y}}\rho)\int\phi(y,P)\rd\Gamma^0(P)\rd\left((\eta_{0,r_{j_k}})_\ast\sigma_\Gamma\right)(y)\\
&-r_{j_k}^{1-m}\int_{G_{m,\beta}(S)}\gamma(\frac{\abs{x}}{r_{j_k}\rho})\phi(\frac{x}{r_{j_k}},P)\rd\Gamma(x,P)\\
=&\int_{S_{j_k}}\gamma(\frac{\abs{y}}\rho)\int\phi(y,P)\rd\Gamma^0(P)\rd\sigma_{\Gamma_{j_k}}(y)
-\int_{G_{m,\beta_{j_k}}(S_{j_k})}\gamma(\frac{\abs{y}}{\rho})\phi(y,P)\rd\Gamma_{j_k}(y,P),
}
where we have used changing of variables in the second equality and \eqref{eq:sigma_Gamma_j}, \eqref{defn:Gamma_j} in the last equality.
In particular, letting $k\ra\infty$ in the above expression, we have by virtue of \eqref{eq:lim-mathscr-F=0} and the convergences $\Gamma_{j_k}\wsc\Gamma_\infty$, $\sigma_{\Gamma_{j_k}}\wsc\sigma_{\Gamma_\infty}$ that
\eq{
\int_{\p\mfR^{n+1}_+}\gamma(\frac{\abs{y}}\rho)\int\phi(y,P)\rd\Gamma^0(P)\rd\sigma_{\Gamma_\infty}(y)
=\int_{G_{m,\beta(0)}(\p\mfR^{n+1}_+)}\gamma(\frac{\abs{y}}{\rho})\phi(y,P)\rd\Gamma_\infty(y,P).
}
Since $\gamma(\frac{\abs{y}}\rho)$ converges to $\chi_{B_\rho(0)}(y)$ as $t\searrow0$, by the arbitrariness of $\rho$ we finally deduce: for any $\phi\in C_c(\mfR^{n+1}\times\mfR^{(n+1)^2})$
\eq{
\int_{\p\mfR^{n+1}_+}\int\phi(y,P)\rd\Gamma^0(P)\rd\sigma_{\Gamma_\infty}(y)
=\int_{G_{m,\beta(0)}(\p\mfR^{n+1}_+)}\phi(y,P)\rd\Gamma_\infty(y,P),
}
which shows that $\Gamma_\infty=\sigma_{\Gamma_\infty}\otimes\Gamma_\infty^x=\sigma_{\Gamma_\infty}\otimes\Gamma^0$, and hence proves \eqref{eq:limit-gamma^x_infty}.

Recall that $T_{x_0}S=\p\mfR^{n+1}_+$, and hence
for any $(x,P)\in G_{m,\beta(0)}(\p\mfR^{n+1}_+)$,
$\mfn(x,P)=\frac{P(e_{n+1})}{\abs{P(e_{n+1})}}$ is constant on $x\in\p\mfR^{n+1}_+$ due to \eqref{defn:n(x,P)}.
Therefore \eqref{eq:n_mcC} follows from \eqref{eq:limit-gamma^x_infty} and \eqref{defn:n_V-intro}.


The assertion concerning $\sigma_\mcC^\perp$, \eqref{eq:limit-sigma_j^perp}, follows directly from \cite[(5.11)]{DeMasi21}.

{({\bf 4})}:
To prove ({\bf 4}) it suffices to prove the following claim.

{\bf Claim.} For any fixed $\rho>0$ there always holds
\eq{
\int_{\p\mfR^{n+1}_+}\Abs{\left<\cos\beta(0)\mfn_W(0),\gamma(\frac{\abs{x}}\rho)x\right>}\rd\sigma_{\Gamma_\infty}(x)
=0,
}
where $\gamma$ is the standard cut-off function on $\mfR$ with $\gamma\equiv1$ on $[0,1-t]$, $\gamma\equiv0$ on $[1,\infty)$, $-\frac2t\leq\gamma'\leq0$ on $\mfR$ for $t>0$ arbitrarily small.

To see this, observe that for every $k$
\eq{\label{eq:n_w-Gamma_Gamma_jk}
&\frac1{r_{j_k}^{m-1}}\int_S{\left<\cos\beta(y)\mfn_W(y),\gamma(\frac{\abs{y}}{r_{j_k}\rho})\frac{y}{r_{j_k}}\right>}\rd\sigma_\Gamma(y)\\
=&\frac1{r_{j_k}^{m-1}}\int_S\left<\cos\beta(y)\mfn_W(y)-\cos\beta(0)\mfn_W(0),\gamma(\frac{\abs{y}}{r_{j_k}\rho})\frac{y}{r_{j_k}}\right>\rd\sigma_\Gamma(y)\\
&+\frac1{r_{j_k}^{m-1}}\int_{S_{j_k}}{\left<\cos\beta(0)\mfn_W(0),\gamma(\frac{\abs{x}}\rho)x\right>}\rd\left((\eta_{0,r_{j_k}})_\ast\sigma_\Gamma\right)(x)\\
=&\frac1{r_{j_k}^{m-1}}\int_S\left<\cos\beta(y)\mfn_W(y)-\cos\beta(0)\mfn_W(0),\gamma(\frac{\abs{y}}{r_{j_k}\rho})\frac{y}{r_{j_k}}\right>\rd\sigma_\Gamma(y)\\
&+\int_{S_{j_k}}{\left<\cos\beta(0)\mfn_W(0),\gamma(\frac{\abs{x}}\rho)x\right>}\rd\sigma_{\Gamma_{j_k}}(x)
\eqqcolon({\bf I})+({\bf II}),
}
where we simply changed of variables in the first equality and used \eqref{eq:sigma_Gamma_j} in the second equality.


Next, note that $\abs{\cos\beta\mfn_W}<1$, and hence $\cos\beta\mfn_W\in L^1_{loc}(\mfR^{n+1},\sigma_{\Gamma_\infty})$.
At this stage we further assume that $x_0$ is a Lebesgue point of the function $\cos\beta\mfn_W$, which is true for $\sigma_\Gamma$-a.e. $x_0\in \mathscr{P}_{\rm cb}(\Gamma)$ with
\eq{\label{eq:Lebesgue-point-cosbeta-n_W}
\lim_{r\searrow0}\frac1{\sigma_\Gamma(B_r(x_0))}\int_{B_r(x_0)}\Abs{\cos\beta(y)\mfn_W(y)-\cos\beta(x_0)\mfn_W(x_0)}\rd\sigma_\Gamma(y)=0,
}
thanks to Lebesgue points Theorem.
From this we obtain
(recall that we have assumed $x_0=0$,
and that $x_0$ is a capillary boundary point, and hence we write accordingly $\ep_{x_0},\rho_{x_0},c_{x_0}$.
For $k$ sufficiently large we have $r_{j_k}\rho<\rho_{x_0}$)
\eq{
&\frac1{r_{j_k}^{m-1}}\int_S\left<\cos\beta(y)\mfn_W(y)-\cos\beta(0)\mfn_W(0),\gamma(\frac{\abs{y}}{r_{j_k}\rho})\frac{y}{r_{j_k}}\right>\rd\sigma_\Gamma(y)\\
\leq&\frac\rho{r_{j_k}^{m-1}}\int_{ B_{r_{j_k}\rho}(0)}\Abs{\cos\beta(y)\mfn_W(y)-\cos\beta(0)\mfn_W(0)}\rd\sigma_\Gamma(y)\\
=&\rho^m\frac{\sigma_\Gamma(B_{r_{j_k}\rho}(0))}{(r_{j_k}\rho)^{m-1}}\frac1{\sigma_\Gamma(B_{r_{j_k}\rho}(0))}\int_{ B_{r_{j_k}\rho}(0)}\Abs{\cos\beta(y)\mfn_W(y)-\cos\beta(0)\mfn_W(0)}\rd\sigma_\Gamma(y)\\
\leq&C\rho^m\frac1{\sigma_\Gamma(B_{r_{j_k}\rho}(0))}\int_{ B_{r_{j_k}\rho}(0)}\Abs{\cos\beta(y)\mfn_W(y)-\cos\beta(0)\mfn_W(0)}\rd\sigma_\Gamma(y),
}
where we have used \eqref{ineq:sigma_Gamma-ratio} for the last inequality.
In particular, this infers that the term $({\bf I})$ in \eqref{eq:n_w-Gamma_Gamma_jk} converges to $0$ as $k\ra\infty$, thanks to \eqref{eq:Lebesgue-point-cosbeta-n_W}.

On the other hand, by \eqref{eq:limit-sigma_jk} we have
\eq{
\lim_{k\ra\infty}\int_{S_{j_k}}{\left<\cos\beta(0)\mfn_W(0),\gamma(\frac{\abs{x}}\rho)x\right>}\rd\sigma_{\Gamma_{j_k}}(x)
=\int_{\p\mfR^{n+1}_+}{\left<\cos\beta(0)\mfn_W(0),\gamma(\frac{\abs{x}}\rho)x\right>}\rd\sigma_{\Gamma_\infty}(x).
}
Letting $k\ra\infty$ in \eqref{eq:n_w-Gamma_Gamma_jk}, we thus find
\eq{
&\int_{\p\mfR^{n+1}_+}{\left<\cos\beta(0)\mfn_W(0),\gamma(\frac{\abs{x}}\rho)x\right>}\rd\sigma_{\Gamma_\infty}(x)\\
=&\lim_{k\ra\infty}\frac1{r_{j_k}^{m-1}}\int_S{\left<\cos\beta(y)\mfn_W(y),\gamma(\frac{\abs{y}}{r_{j_k}\rho})\frac{y}{r_{j_k}}\right>}\rd\sigma_\Gamma(y).
}
The claim is then proved by
the following estimate: 
\eq{
&\frac1{r_{j_k}^{m-1}}\int_S\Abs{\left<\cos\beta(y)\mfn_W(y),\gamma(\frac{\abs{y}}{r_{j_k}\rho})\frac{y}{r_{j_k}}\right>}\rd\sigma_\Gamma(y)\\
=&\frac1{r_{j_k}^{m-1}}\int_{B_{r_{j_k}\rho}(0)}\Abs{\left<\cos\beta(y)\mfn_W(y),\gamma(\frac{\abs{y}}{r_{j_k}\rho})\frac{y}{r_{j_k}}\right>}\rd\sigma_\Gamma(y)\\
\leq&\frac{c_{x_0}r_{j_k}^2\rho^2\abs{\cos\beta(0)}}{r_{j_k}}\frac{\sigma_\Gamma(B_{r_{j_k}\rho}(0))}{r_{j_k}^{m-1}}
\leq C\rho^{m+1} r_{j_k}
\ra0\text{ as }k\ra\infty,
}
where we have used \eqref{defn:genera-capi-bdry-point-2} in the first inequality and \eqref{ineq:sigma_Gamma-ratio} for the last inequality.

{({\bf 5})}:
First we show that $\mcC$ is rectifiable, which is based on the following facts:
For $k$ sufficiently large,
\begin{enumerate}
    \item $x_0=0\in{\rm spt}\mu_V\cap S$ thanks to Remark \ref{Rem:G_pb-spt}, and hence $0\in{\rm spt}\mu_{V_{j_k}}\cap S_{j_k}$ for every $k$.
    \item $\Theta^m(\mu_{V_{j_k}},x)\geq a$ for $\mu_{V_{j_k}}$-a.e. $x$ since $\Theta^m(\mu_V,x)\geq a$ for $\mu_V$-a.e. $x$;
    \item $\sup_{k\geq1}\{\mu_{V_{j_k}}(B_1(0))\}<\infty$ thanks to \eqref{ineq:mu_V_j-bound};
    \item $\{V_{j_k}\}_{k\in\mbN}$ have locally uniformly bounded first variation thanks to \eqref{ineq:mu_V_j-bound} and \eqref{ineq:sigma_Gamma-ratio}. 
\end{enumerate}
An immediate consequence of these properties is that each $V_{j_k}$ is rectifiable, by virtue of the Rectifiabilty Theorem \cite[Theorem 42.4]{Simon83}.
The fact that $\mcC$ is rectifiable then follows from Allard's compactness theorem, see \cite[Theorem 42.7]{Simon83}.
As a by-product, we have
\eq{
\Theta^m(\mu_\mcC,x)\geq a, \mu_V\text{-a.e. }x.
}

Next we prove that $\mcC$ is a cone.
Observe that
for $\mcC$ we also have the differential equality \eqref{eq:monoto-identity}, which reads, thanks to \eqref{eq:H_C}, ({\bf 4}), and the fact that $\p\mfR^{n+1}_+$ is flat, as follows
\eq{
\frac{\rd}{\rd\rho}\left(\frac1{\rho^m}\int_{\overline{\mfR^{n+1}_+}}\gamma(\frac{\abs{x}}\rho)\rd\mu_\mcC\right)
=-\frac1{\rho^{m+1}}\int_{G_m(\overline{\mfR^{n+1}_+})}\frac{\abs{x}}\rho\gamma'(\frac{\abs{x}}\rho)\Abs{P^\perp(\frac{x}{\abs{x}})}^2\rd\mcC(x,P).
}
Integrating over $(r_1,r_2)$ we get
\eq{
&\frac1{r_2^m}\int_{\overline{\mfR^{n+1}_+}}\gamma(\frac{\abs{x}}{r_2})\rd\mu_\mcC(x)-\frac1{r_1^m}\int_{\overline{\mfR^{n+1}_+}}\gamma(\frac{\abs{x}}{r_1})\rd\mu_\mcC(x)\\
=&\frac1{r_2^m}\int_{G_m(\overline{\mfR^{n+1}_+})}\gamma(\frac{\abs{x}}{r_2})\Abs{P^\perp(\frac{x}{\abs{x}})}^2\rd\mcC(x,P)-\frac1{r_1^m}\int_{G_m(\overline{\mfR^{n+1}_+})}\gamma(\frac{\abs{x}}{r_1})\Abs{P^\perp(\frac{x}{\abs{x}})}^2\rd\mcC(x,P)\\
&+\int_{G_m(\overline{\mfR^{n+1}_+})}\Abs{P^\perp(\frac{x}{\abs{x}})}^2\left(\int_{r_1}^{r_2}\frac{m}{\rho^{m+1}}\gamma(\frac{\abs{x}}\rho)\rd\rho\right)\rd\mcC(x,P).
}
Letting $\gamma$ increase to $\chi_{[0,1]}$ and note that by an approximation argument, for every $\rho>0$ one has
\eq{
\frac{\mu_\mcC(B_\rho(0))}{\rho^m}
=\lim_{k\ra\infty}\frac{\mu_{V_{j_k}}(B_\rho(0))}{\rho^m}
=\lim_{k\ra\infty}\frac{\mu_V(B_{r_{j_k}\rho}(0))}{r_{j_k}^m\rho^m}
=\Theta^m(\mu_V,0),
}
we find
\eq{
0
=\frac{\mu_\mcC(B_{r_2}(0))}{r_2^m}
-\frac{\mu_\mcC(B_{r_1}(0))}{r_1^m}
=\int_{G_m(B_{r_2}(0)\setminus B_{r_1}(0))}\frac{\Abs{P^\perp(\frac{x}{\abs{x}})}^2}{\abs{x}^m}\rd\mcC(x,P).
}
Since this holds for arbitrary $0<r_1<r_2$,
it follows that
\eq{\label{eq:C-cone-null-set}
\mcC\left((x,P):P(x)\neq x\right)=0.
}
Finally, for any $0$-homogeneous function $h\in C^1(\mfR^{n+1})$, one has
\eq{
\abs{h}_{C^0(\mfR^{n+1})}
=\abs{h}_{C^0(\mfS^n)}
\leq C_h,\quad\left<\na h(x),x\right>=0,
}
and from \eqref{eq:C-cone-null-set} we infer
\eq{\label{eq:nabla-h}
\left<\na h(x),P(x)\right>
=\left<\na h(x),x\right>
=0,\quad\text{for }\mcC\text{-a.e. }(x,P).
}
Going back to the differential equality \eqref{eq:monoto-identity-0} again but with $\varphi(x)=\gamma(\frac{\abs{x}}\rho)h(x)x$ this time, we deduce
\eq{
&\frac{\rd}{\rd\rho}\left(\frac1{\rho^m}\int_{\overline{\mfR^{n+1}_+}}\gamma(\frac{\abs{x}}\rho)h(x)\rd\mu_\mcC(x)\right)\\
=&-\frac1{\rho^{m+1}}\int_{\p\mfR^{n+1}}\left<\cos\beta(0)\mfn_W(0),\gamma(\frac{\abs{x}}\rho)h(x)x\right>\rd\sigma_{\Gamma_\infty}(x)\\
&+\frac1{\rho^{m+1}}\int_{G_m(\overline{\mfR^{n+1}_+})}\gamma(\frac{\abs{x}}\rho)\left<\na h(x),P(x)\right>\rd\mcC(x,P)\\
=&0,
}
where we have used ({\bf 4}) and \eqref{eq:nabla-h} to derive the last equality.
In particular, this is sufficient to show that $\mcC$ is a cone in the sense that
\eq{\label{eq:dilation-mcC}
(\bseta_{0,\rho})_\#\mcC=\mcC,\quad\forall\rho>0,
}
see e.g., \cite[(5.17)]{DeMasi21}.
Moreover, as the first variation of $\mcC$ with respect to $\varphi\in\mathfrak{X}_t(\mfR^{n+1}_+)$ is given by
\eq{
\de\mcC(\varphi)
=\int_{G_m(\overline{\mfR^{n+1}_+})}{\rm div}_P\varphi(x)\rd\mcC(x,P)
=\int_{\p\mfR^{n+1}_+}\left<\cos\beta(0)\mfn_W(0),\varphi(x)\right>\rd\sigma_{\Gamma_\infty}(x),
}
the claimed fact that $\sigma_{\Gamma_\infty}$ is scaling invariant then follows from \eqref{eq:dilation-mcC}.
More precisely, for any $\varphi\in\mathfrak{X}_t(\mfR^{n+1}_+)$ the vector field $\tilde\varphi(y)=\varphi(\frac{y}\rho)$ is also tangent to $\p\mfR^{n+1}_+$.
Testing the first variation of $(\bseta_{0,\rho})_\#\mcC$ with $\varphi$ we obtain by virtue of \eqref{eq:dilation-mcC}
\eq{
\int_{G_m(\overline{\mfR^{n+1}_+})}{\rm div}_P\varphi(x)\rd\mcC(x,P)
=&\int_{G_m(\overline{\mfR^{n+1}_+})}{\rm div}_P\varphi(x)\rd\left((\bseta_{0,\rho})_\#\mcC\right)(x,P)\\
=&\frac1{\rho^{m-1}}\int_{G_m(\overline{\mfR^{n+1}_+})}{\rm div}_P\varphi(\frac{y}\rho)\rd\mcC(y,P),
}
which gives
\eq{
\int_{\p\mfR^{n+1}_+}\left<\cos\beta(0)\mfn_W(0),\varphi(x)\right>\rd\sigma_{\Gamma_\infty}&(x)
=\frac1{\rho^{m-1}}\int_{\p\mfR^{n+1}_+}\left<\cos\beta(0)\mfn_W(0),\varphi(\frac{y}\rho)\right>\rd\sigma_{\Gamma_\infty}(y)\\
=&\frac1{\rho^{m-1}}\int_{\p\mfR^{n+1}_+}\left<\cos\beta(0)\mfn_W(0),\varphi(x)\right>\rd\left((\bseta_{0,\rho})_\ast\sigma_{\Gamma_\infty}\right)(x),
}
and proves the claimed fact.
\end{proof}

The characterization of $\sigma_{\Gamma_\infty}$ is divided into the following Lemmatum.
In what follows
we denote by $L_\infty$ the $(n-1)$-dimensional linear subspace of $T_{x_0}S\cong\mfR^n$
which is perpendicular to $\mfn_W(x_0)$.

\begin{lemma}[Linear subspace]\label{Lem:linear-subspace}
Under the assumptions and notations in Lemma \ref{Lem:blow-up-seq},
let $x_0$ be resulting from Lemma \ref{Lem:blow-up-seq}, and
assume after translation and rotation that $x_0=0$ and $T_{x_0}^+\overline\Om=\overline{\mfR^{n+1}_+}$.
Define the set
\eq{
D_\mcC
\coloneqq\{y\in L_\infty: \Theta^m(\mu_\mcC,y)=\Theta^m(\mu_\mcC,0)\}.
}
Then $D_\mcC$ is a linear subspace of $L_\infty$.
Moreover, $(\bstau_{-y})_\#\mcC=\mcC$ and $(\bstau_{-y})_\ast\sigma_{\Gamma_\infty}=\sigma_{\Gamma_\infty}$ for any $y\in D_\mcC$.
\end{lemma}
\begin{proof}
By definition of $L_\infty$, we have in virtue of Lemma \ref{Lem:blow-up-seq} ({\bf 4}) that for $\sigma_{\Gamma_\infty}$-a.e. $x$ and any $y\in L_\infty$,
\eq{
x-y\in L_\infty,
}
and hence for any $\rho>0$ there always holds
\eq{
\int_{\p\mfR^{n+1}_+}\Abs{\left<\cos\beta(0)\mfn_W(0),\gamma(\frac{\abs{x-y}}\rho)(x-y)\right>}\rd\sigma_{\Gamma_\infty}(x)
=0.
}
As what is done in the proof of Lemma \ref{Lem:blow-up-seq}, for the rectifiable cone $\mcC$ we go back to the differential equality \eqref{eq:monoto-identity-0} but with $\varphi(x)=\gamma(\frac{\abs{x-y}}{\rho})(x-y)$.
Thanks to the observation above, we get
\eq{
&\frac{\rd}{\rd\rho}\left(\frac1{\rho^m}\int_{\overline{\mfR^{n+1}_+}}\gamma(\frac{\abs{x-y}}\rho)\rd\mu_\mcC(x)\right)\\
=&-\frac1{\rho^{m+1}}\int_{G_m(\overline{\mfR^{n+1}_+})}\frac{\abs{x-y}}\rho\gamma'(\frac{\abs{x-y}}\rho)\Abs{P^\perp(\frac{x-y}{\abs{x-y}})}^2\rd\mcC(x,P)
\geq0,
}
so that
\eq{
\Theta^m(\mu_\mcC,y)
\leq\lim_{r\ra\infty}\frac{\mu_\mcC(B_r(y))}{\om_mr^m}.
}
Moreover, recall that $\mcC$ is a rectifiable cone thanks to Lemma \ref{Lem:blow-up-seq} ({\bf 5}), thus
\eq{\label{ineq:Theta^m-mcC-0-y}
\om_m\Theta^m(\mu_\mcC,0)
=&\lim_{r\ra\infty}\frac{\mu_\mcC(B_r(0))}{r^m}
\geq\lim_{r\ra\infty}\frac{\mu_\mcC(B_{r-\abs{y}}(y))}{(r-\abs{y})^m}\frac{(r-\abs{y})^m}{r^m}\\
=&\lim_{r\ra\infty}\frac{\mu_\mcC(B_r(y))}{r^m}
\geq\om_m\Theta^m(\mu_\mcC,y),\quad\forall y\in L_\infty.
}
Now by virtue of the construction of $D_\mcC$ we know that every inequality is in fact equality in the above argument.
In particular,
\eq{
\frac{\mu_\mcC(B_{r_2}(y))}{r_2^m}-\frac{\mu_\mcC(B_{r_1}(y))}{r_1^m}
\equiv0,\quad\forall0<r_1<r_2,
}
which is sufficient to show that $\mcC$ is a cone with respect to $y$, as shown in the proof of Lemma \ref{Lem:blow-up-seq}.
Now observe that if $\mcC$ is a cone with respect to vertex $v$, then for any $z\in\overline{\mfR^{n+1}_+}$ we have
\eq{
\Theta^m(\mu_\mcC,z)
=\Theta^m(\mu_\mcC,v+\frac{z-v}2)
=\Theta^m(\mu_\mcC,\frac{v+z}2).
}
As $\mcC$ is a cone with respect to both $0$ and $y$, we deduce for any $z\in\overline{\mfR^{n+1}_+}$
\eq{\label{eq:mu_C,z-y+z}
\Theta^m(\mu_\mcC,z)
=\Theta^m(\mu_\mcC,y+z),
}
which shows that $(\bstau_{-y})_\#\mcC=\mcC$.
The translation invariance of $\sigma_{\Gamma_\infty}$ follows
similarly to the last part of the proof of Lemma \ref{Lem:blow-up-seq}.

Finally, we show that $D_\mcC$ is in fact a linear subspace, that is, for every $y,z\in D_\mcC$, we have $y+z\in D_\mcC$ and $\lambda y\in D_\mcC, \forall\lambda\in\mfR$.
The first property follows from \eqref{eq:mu_C,z-y+z} and the last property is a direct consequence of the fact that $\mcC$ is a cone with respect to both $0$ and $y$.
The proof is thus completed.
\end{proof}
\begin{lemma}\label{Lem:sigma_Gamma(G_pb)}
Let $V\in{\bf V}^m_\beta(\Om)$ with boundary varifold $\Gamma$, such that $\mfH\in L^p(\mu_V)$ for some $p\in(m,\infty)$, and $\Theta^m(\mu_V,x)\geq a>0$ for $\mu_V$-a.e. $x$.

Then for any $\sigma_\Gamma$-measurable set $A\subseteq \mathscr{P}_{\rm cb}(\Gamma)$, 
there exists a set $F\subset S$ with
\eq{\label{eq:sigma_V-G_pb}
\sigma_\Gamma(A\setminus F)=0,
}
such that the dichotomy holds: for every $x_0\in A\cap F$ and for every $\mcC\in {\rm VarTan}(V,x_0)$, either
\eq{
\sigma_{\Gamma_\infty}=0,
}
or
\eq{\label{eq:mu_C-y-0}
\sigma_{\Gamma_\infty}\neq0,\text{ with }
\Theta^m(\mu_\mcC,y)=\Theta^m(\mu_\mcC,0),\quad\forall y\in{\rm spt}\sigma_\mcC\cap L_\infty.
}
In particular, ${\rm spt}\sigma_{\Gamma_\infty}=D_\mcC$.
\end{lemma}
\begin{proof}
\noindent{\bf Step 1. We construct the set $F$.}

First note that the collection of points $x\in A\subset \mathscr{P}_{\rm cb}\subset{\rm spt}\sigma_\Gamma\cap S$ at which
\begin{enumerate}
    \item $x$ is a point of density $1$ for $A$ with respect to $\sigma_\Gamma$, i.e.,
\eq{
\lim_{\rho\ra0}\frac{\sigma_\Gamma(A\cap B_\rho(x))}{\sigma_\Gamma(B_\rho(x))}
=1.
}
   \item The density function $\Theta^m(\mu_V,\cdot)$ (which is a well-defined function on $\mathscr{P}_{\rm cb}$ thanks to Corollary \ref{Cor:density-existence}) is approximate continuous at $x$ with respect to $\sigma_\Gamma$,
\end{enumerate}
denoted by $E_\Gamma$, is of full measure with respect to $\sigma_\Gamma$ by virtue of Lebesgue differentiation Theorem, i.e.,
\eq{\label{eq:sigma_V-E_Gamma}
\sigma_\Gamma(A\setminus E_\Gamma)=0.
}
Moreover,
note that Lemma \ref{Lem:blow-up-seq} holds for $\sigma_\Gamma$-a.e. $x\in \mathscr{P}_{\rm cb}$, up to modifying a $\sigma_\Gamma$-negligible set, we could assume that every $x\in E_\Gamma$ satisfies the conclusions of Lemma \ref{Lem:blow-up-seq}.

Since $p>m$, by \eqref{ineq-mu_V-B_r-B_rho} for every $x\in E_\Gamma$ we know that the resulting functions $g_x(\rho)$ are increasing on $\rho$ and converge to $0$ point-wisely as $\rho\searrow0$.
By Egoroff's Theorem, for every $i\in\mbN$ we could find a set $F_i\subset E_\Gamma$ with $\sigma_\Gamma(E_\Gamma\setminus F_i)\leq\frac1i$, such that $g_x(\rho)$ uniformly converges to $0$ on $F_i$ as $\rho\searrow0$.

Now we define the desired $F$ as
\eq{
F
\coloneqq E_\Gamma\cap\left(\bigcup_{i\in\mbN}F_i\right).
}
By construction of $F$ and \eqref{eq:sigma_V-E_Gamma} we have
\eq{
0
=\sigma_\Gamma(E_\Gamma\setminus F)
=\sigma_\Gamma(A\setminus F),
}
which proves \eqref{eq:sigma_V-G_pb}.

To prove the dichotomy we recall that by Lemma \ref{Lem:blow-up-seq}, for every $x_0\in F$ and every $\mcC\in{\rm VarTan}(V,x_0)$ there exists a sequence $\{r_j\searrow0\}_{j\in\mbN}$ such that the corresponding $V_j\ra\mcC$ as varifolds.
In particular, $\mcC$ is a non-trivial rectifiable cone with $\Theta^m(\mu_\mcC,x)\geq a$ for $\mu_\mcC$-a.e. $x$.
Again,
by Lemma \ref{Lem:blow-up-seq} ({\bf 1}), each $V_j\in{\bf V}^m_{\beta_j}(\Om_j)$ with $\mfH_j(x)=r_j\mfH(x_0+r_jx)$ belongs to $L^p$ as well, and hence we could apply \eqref{ineq-mu_V-B_r-B_rho} to each $V_j$ and get increasing functions $g^j_{r_jx}(\rho)$.
Moreover, we have the scaling property
\eq{
g^j_x(\rho)
=g_{r_jx}(r_j\rho).
}

We are now ready to prove the dichotomy, and it suffices to consider $\sigma_{\Gamma_\infty}\neq0$.
Before that, after translation we may assume $x_0=0\in S$, $T_{x_0}\overline\Om=\{x_{n+1}\geq0\}$ is the upper half-space, which we denote by $\overline{\mfR^{n+1}_+}$, and $T_{x_0}S=\{x_{n+1}=0\}=\p\mfR^{n+1}_+$.
We also recall the definitions of the $(n-1)$-dimensional linear subspace $L_\infty$ and also $D_\mcC$ introduced in Lemma \ref{Lem:linear-subspace}.

\noindent{\bf Step 2. We show that: ${\rm spt}\sigma_{\Gamma_\infty}\cap L_\infty=D_\mcC$, on which \eqref{eq:mu_C-y-0} holds.}

First by Lemma \ref{Lem:blow-up-seq} ({\bf 5}), $\sigma_{\Gamma_\infty}$ is scaling invariant, since $\sigma_{\Gamma_\infty}\neq0$ we must have $0\in{\rm spt}\sigma_{\Gamma_\infty}$.
On the other hand, by Lemma \ref{Lem:linear-subspace}, $\sigma_{\Gamma_\infty}$ is translation invariant along $D_\mcC$, we see that
\eq{
D_\mcC\subset{\rm spt}\sigma_{\Gamma_\infty}.
}
Note also that by definition of $D_\mcC$ we have $D_\mcC\subset L_\infty$, thus $D_\mcC\subset{\rm spt}\sigma_{\Gamma_\infty}\cap L_\infty$.

To show the opposite inclusion we argue by contradiction.
Recall that by \eqref{ineq:Theta^m-mcC-0-y}
\eq{
\Theta^m(\mu_\mcC,y)
\leq\Theta^m(\mu_\mcC,0),\quad\forall y\in L_\infty.
}
Thus we assume by contradiction that there exists some $y\in{\rm spt}\sigma_{\Gamma_\infty}\cap L_\infty$ and $\ep>0$ such that
\eq{
\Theta^m(\mu_\mcC,y)
<\Theta^m(\mu_\mcC,0)-\ep.
}
Thanks again to Lemma \ref{Lem:linear-subspace} after translation we may assume that $y\in B_{\frac12}(0)$.
By virtue of the construction of the sets $F$ and $F_i$, we can show there exists $J\in\mbN$ large and $0<r$ small such that
\eq{\label{inclu:B_r(y)-F_i}
B_r(y)\cap\frac1{r_j}F_i
\subseteq\left\{x\in B_1(0):\Abs{\Theta^m(\mu_{V_j},x)-\Theta^m(\mu_{V_j},0)}>\frac\ep2\right\},\quad\forall j>J,
}
see \cite[(5.35)]{DeMasi21}.
In particular, this violates the approximate continuity of the density function $\Theta^m(\mu_V,\cdot)$ with respect to $\sigma_\Gamma$ at $0$:
By definition of approximate continuity one must have
\eq{
0
=&\lim\sup_{j\ra\infty}\frac{\sigma_\Gamma\left(\left\{z\in B_{r_j}(0):\Abs{\Theta^m(\mu_{V},z)-\Theta^m(\mu_{V},0)}>\frac\ep2\right\}\right)}{\sigma_\Gamma(B_{r_j}(0))}\\
\overset{\eqref{inclu:B_r(y)-F_i}}{\geq}&\limsup_{j\ra\infty}\frac{\sigma_\Gamma(B_{rr_j}(r_jy)\cap F_i)}{\sigma_\Gamma(B_{r_j}(0))}
\overset{\eqref{eq:sigma_V-E_Gamma}}{=}\limsup_{j\ra\infty}\frac{\sigma_\Gamma(B_{rr_j}(r_jy)}{\sigma_\Gamma(B_{r_j}(0))}
=\limsup_{j\ra\infty}\frac{\sigma_j^T(B_r(y))}{\sigma_j^T(B_1(0))},
}
where we have used Lemma \ref{Lem:blow-up-seq} ({\bf 1}) in the last equality.
This leads to a contradiction since the last term is never vanishing, which can be inferred from the convergence \eqref{eq:limit-sigma_jk} and the fact that the points $0,y$ both belong to ${\rm spt}\sigma_{\Gamma_\infty}$.
In particular, the contradiction argument shows that ${\rm spt}\sigma_{\Gamma_\infty}\cap L_\infty\subset D_\mcC$, which completes the second step.

\noindent{\bf Step 3. We prove that ${\rm spt}\sigma_{\Gamma_\infty}=D_\mcC$.}

In view of {\bf Step 2}, we just need to prove that ${\rm spt}\sigma_{\Gamma_\infty}\subset L_\infty$.
To do so, recall \eqref{eq:x-n_W(x_0)} and the definition of $L_\infty$, we find $\sigma_{\Gamma_\infty}(\mfR^{n+1}\setminus L_\infty)=0$, i.e., $\sigma_{\Gamma_\infty}$ is concentrated on $L_\infty$, and hence ${\rm spt}\sigma_{\Gamma_\infty}\subset L_\infty$ as desired.
In particular, this completes the proof.
\end{proof}

\begin{lemma}\label{Lem:sigma_Gamma_infty-characterization}
Under the assumptions and notations in Lemma \ref{Lem:sigma_Gamma(G_pb)}.
For every $x_0\in A\cap F$ and for any $\mcC\in{\rm VarTan}(V,x_0)$, if
\eq{
\Theta^{m-1}_\ast(\sigma_\Gamma,x_0)>0,
}
then there hold: $D_\mcC$ is a $(m-1)$-dimensional linear subspace and there exists some $\theta_0>0$ such that
\eq{
\sigma_{\Gamma_\infty}
=\theta_0\mcH^{m-1}\llcorner D_\mcC.
}
\end{lemma}
\begin{proof}
After translation and rotation we may assume that $x_0=0\in S$, $T_{x_0}\overline\Om=\overline{\mfR^{n+1}_+}$, and $T_{x_0}S=\{x_{n+1}=0\}=\p\mfR^{n+1}_+$.
For $\mcC\in{\rm VarTan}(V,x_0)$ we have by definition a sequence $\{r_j\searrow0\}_{j\in\mbN}$ such that $V_j\ra\mcC$ as varifolds.
Now the condition gives
\eq{
\theta_0
\coloneqq\liminf_{j\ra\infty}\frac{\sigma_\Gamma(B_{r_j}(0))}{\om_{m-1}r_j^{m-1}}>0,
}
which together with Lemma \ref{Lem:blow-up-seq} ({\bf 1}) and the convergence \eqref{eq:limit-sigma_jk}, shows that
\eq{
\sigma_{\Gamma_\infty}\neq0.
}
Thus from dichotomy (Lemma \ref{Lem:sigma_Gamma(G_pb)}) we obtain
\eq{\label{eq:spt-sigma_L_infty-D_C}
{\rm spt}\sigma_{\Gamma_\infty}
=D_\mcC,
}
where $D_\mcC$ is a linear subspace of $L_\infty$ as shown in Lemma \ref{Lem:linear-subspace}.
Recall the definition of $L_\infty$ we know that the dimension of $D_\mcC$, say $k\in\mbN$, satisfies $k\leq n-1$.
After rotation we may assume that $D_\mcC=\{x:x_{k+1}=\ldots=x_{n+1}=0\}$.

For any $x\in D_\mcC$ and any $r>0$, consider the closed cube in $D_\mcC$ centered at $x$, of side length $r$, and with faces parallel to the coordinate vectors $e_1,\ldots,e_k$, say $Q_{D_\mcC}(x,r)$.
Note that $\sigma_{\Gamma_\infty}$ is invariant under rescaling (Lemma \ref{Lem:blow-up-seq} ({\bf 5})) and translation in $D_\mcC$ (Lemma \ref{Lem:linear-subspace}), we thus have some $\lambda>0$ such that
\eq{
\sigma_{\Gamma_\infty}(Q_{D_\mcC}(x,r))
=\lambda r^{m-1},\quad\forall x\in D_\mcC, r>0.
}
This is a sufficient condition to show that the dimension $k$ is in fact $m-1$, see \cite[(5.39)]{DeMasi21}.
In particular, again using the rescaling and translation invariance of $\sigma_{\Gamma_\infty}$, we obtain
\eq{
\frac{\sigma_{\Gamma_\infty}
(B_r(x))}{\mcH^{m-1}(D_\mcC\cap B_r(x))}
=\theta_0,\quad\forall x\in D_\mcC, r>0.
}
By Radon-Nikodym Theorem we have
\eq{
\sigma_{\Gamma_\infty}
=\theta_0\mcH^{m-1}\llcorner D_\mcC
}
as required.
This completes the proof.

\end{proof}
\subsection{Classification of tangent cones and Allard-type boundary regularity}
Now we classify tangent cones at the capillary boundary points.
\begin{proof}[Proof of Theorem \ref{Prop:Classify-VarTan}]
Let $F$ be the set resulting from Lemma \ref{Lem:sigma_Gamma(G_pb)}.

We claim that $\mcC$ is a non-trivial stationary integral $m$-cone satisfying
\eq{
\frac{\mu_\mcC(B_\rho(0))}{\om_m\rho^m}
=\Theta^m(\mu_\mcC,0)
=\Theta^m(\mu_V,x_0)\leq\frac12+\ep,\quad\forall\rho>0.
}
In fact, this is a direct consequence of Lemma \ref{Lem:blow-up-seq} ({\bf 5}) and its proof, where the fact that $\mcC$ is non-trivial and integral is due to $V\in{\bf IV}^m(\overline\Om)$.

To proceed, note that Lemmas \ref{Lem:linear-subspace}$\sim$\ref{Lem:sigma_Gamma_infty-characterization} imply that $\mcC$ has at least $(m-1)$-dimensions of translational symmetry, which, together with the claim, shows that $\mcC$ must be the induced varifold of some half-$m$-plane.

If $V\in{\bf RV}^m_\beta(\Om)$ then the disintegration of $\Gamma$ takes the form
\eq{
\Gamma
=\sigma_\Gamma\otimes\Gamma^x
=\mcH^{m-1}\llcorner M\otimes\de_{P_\Gamma(x)},
}
and at $x_0$ we have $\sigma_{\Gamma_\infty}=\mcH^{m-1}\llcorner T_{x_0}M$ by Lemma \ref{Lem:blow-up-seq} ({\bf 1}), \eqref{eq:limit-sigma_jk}, and the fact that $M$ is $(m-1)$-rectifiable;
while by definition $\Gamma^{x_0}=\de_{P_\Gamma(x_0)}$ for the $m$-plane $P_\Gamma(x_0)\in G_{m,\beta}(x_0)$. 
By \eqref{eq:limit-gamma^x_infty}
we have 
\eq{\label{eq:Gamma_infty}
\Gamma_\infty
=\sigma_{\Gamma_\infty}\otimes\Gamma^x_\infty
=\mcH^{m-1}\llcorner T_{x_0}M\otimes\de_{P_\Gamma(x_0)}.
}
Note also that $P_{\Gamma}(x_0)\cap T_{x_0}S=T_{x_0}M$ by Definition \ref{Defn:rectifiable-bdry}.

Since $\mcC$ is induced by some half-$m$-plane that lies in $T_{x_0}^+\overline\Om$, we could write the first variation of $\mcC$ explicitly,
taking \eqref{eq:first-variation-C}, \eqref{eq:Gamma_infty} into account, the assertion follows.
More precisely, say $\mcC$ is induced by the half-$m$-plane $\mbP$ that lies in $T_{x_0}\overline\Om\cong\overline{\mfR^{n+1}_+}$.
Let $\mfn_\mbP$ denote the inwards-pointing unit co-normal along $\mbP\cap\p\mfR^{n+1}_+$ with respect to to $\mbP$,
then we could write
\eq{
\mfn_\mbP
=\sin\alpha e_{n+1}+\cos\alpha\frac{\p\mfR^{n+1}_+(\mfn_\mbP)}{\Abs{\p\mfR^{n+1}_+(\mfn_\mbP)}},
}
where $\alpha\in[0,\frac\pi2]$ is the contact angle of $\mbP$ and $\p\mfR^{n+1}_+$.
It follows that for any $\varphi\in\mathfrak{X}(\mfR^{n+1}_+)$ with compact support,
\eq{\label{eq:de-mbH}
\de\mcC(\varphi)
=&0-\int_{\mbP\cap\p\mfR^{n+1}_+}\sin\alpha\left<e_{n+1},\varphi\right>\rd\mcH^{m-1}\\
&-\int_{\mbP\cap\p\mfR^{n+1}_+}\left<\Abs{\cos\beta(x_0)}\frac{\p\mfR^{n+1}_+(\mfn_\mbP)}{\Abs{\p\mfR^{n+1}_+(\mfn_\mbP)}},\varphi\right>\frac{\cos\alpha}{\Abs{\cos\beta(x_0)}}\rd\mcH^{m-1}.
}
On the other hand, by \eqref{eq:first-variation-C}, \eqref{eq:Gamma_infty}, for any $\varphi\in\mathfrak{X}_t(\mfR^{n+1}_+)$ with compact support,
\eq{
\de\mcC(\varphi)
=&-\int_{T_{x_0}M\cap\p\mfR^{n+1}_+}\left<\mfn(x,P_\Gamma(x_0)),\varphi(x)\right>\rd\mcH^{m-1}\\
=&-\int_{T_{x_0}M\cap\p\mfR^{n+1}_+}\left<\Abs{\cos\beta(x_0)}\frac{\p\mfR^{n+1}_+\left(\mfn(x,P_\Gamma(x_0))\right)}{\Abs{\p\mfR^{n+1}_+\left(\mfn(x,P_\Gamma(x_0))\right)}},\varphi(x)\right>\rd\mcH^{m-1}.
}
Combining these two variational formulas, we obtain
\eq{
\begin{cases}
\mbP\cap\p\mfR^{n+1}_+
=T_{x_0}M
=P_\Gamma(x_0)\cap T_{x_0}S
,\\
\frac{\p\mfR^{n+1}_+(\mfn_\mbP)}{\Abs{\p\mfR^{n+1}_+(\mfn_\mbP)}}
=\frac{\p\mfR^{n+1}_+\left(\mfn(x,P_\Gamma(x_0))\right)}{\Abs{\p\mfR^{n+1}_+\left(\mfn(x,P_\Gamma(x_0))\right)}}
,\quad\forall x\in T_{x_0}M,\\
\alpha
=\beta(x_0)\text{ if }\beta(x_0)\in(0,\frac\pi2);\text{ or }\pi-\beta(x_0)\text{ otherwise}.\label{eq-alpha-beta(x_0)}
\end{cases}
}
Therefore, $\mcC$ is induced by the half-$m$-plane $\mbP=P_\Gamma(x_0)\cap T_{x_0}\overline\Om$ as asserted.

\end{proof}
A direct consequence of Theorem \ref{Prop:Classify-VarTan} is that
\eq{
{\rm spt}\sigma_{\mcC}^\perp
={\rm spt}\sigma_{\Gamma_\infty},
}
since we have from \eqref{eq:1st-variation-V_beta}, Lemma \ref{Lem:sigma_Gamma_infty-characterization} and \eqref{eq:de-mbH} that
\eq{\label{eq:sigma_c^perp-sigma_Gamma-infty}
\sigma_\mcC^\perp
=\sin\alpha\mcH^{m-1}\llcorner(\mbP\cap\p\mfR^{n+1}_+),\quad
\sigma_{\Gamma_\infty}
=\theta_0\mcH^{m-1}\llcorner(\mbP\cap\p\mfR^{n+1}_+)\text{ for some }\theta_0>0.
}

If $\mu_V(S)=0$, then by construction of the blow-up sequence we have $\mcC=\mcC_I\coloneqq\mcC\llcorner\mfR^{n+1}_+$.
Moreover, again the first variation of $\mcC$ gives
\eq{
{\rm spt}\sigma_{\mcC}^\perp
={\rm spt}\sigma_{\Gamma_\infty}
={\rm spt}\mu_{\mcC_I}\cap\p\mfR^{n+1}_+,
}
which can be viewed as the limiting case of \cite[(3.9)]{DEGL24}.

Another application of Theorem \ref{Prop:Classify-VarTan}, as pointed out in the introduction, is as follows.
\begin{proof}[Proof of Theorem \ref{Thm:Allard}]
Conclusion ($i$) follows from Proposition \ref{Prop:RV-codim-1}.
In particular, we conclude from Proposition \ref{Prop:1st-variation} that $V$ has bounded first variation.

As for conclusion ($ii$), the existence of the set $F$ is ensured by Theorem \ref{Prop:Classify-VarTan}; the $C^{1,\alpha}$-regularity of ${\rm spt}\mu_V$ near $x_0$ together with the contact angle information is a direct consequence of \cite[Theorem 2.3, Corollary 2.4]{Wang24} and Theorem \ref{Prop:Classify-VarTan}.

\end{proof}

\subsection{Strong maximum principle}
Next we give a boundary strong maximum principle for the blow-up limits resulting from Theorem \ref{Prop:Classify-VarTan}.

Before we begin, we note that in Definition \ref{Defn:vfld-prescribed-bdry}, any $V\in{\bf V}^m_\beta(\Om)$ admits as many as free boundary components, as shown in Example \ref{exam:caps}.
Therefore, we do not expect a general $V$ to satisfy a boundary strong maximum principle like smooth submanifolds, see e.g., \cite[Lemma 1.13]{LZZ21}.

\begin{corollary}\label{Prop:bdry-SMP}
Let $\mcC\in{\rm VarTan}(V,x_0)$
be a tangent cone resulting from Theorem \ref{Prop:Classify-VarTan}, namely, $\mcC$ is induced by the half-$m$-plane $\mbP$ that lies in $\overline{\mfR^{n+1}_+}=T_{x_0}\overline\Om$ (up to translation and rotation).
Write $\alpha\in[0,\frac\pi2]$ as the contact angle of $\mbP$ and $\p\mfR^{n+1}_+$, in the sense that
the inwards-pointing unit co-normal along the boundary, denoted by $\mfn_{\mbP}$, satisfies
\eq{
\mfn_\mbP
=\sin\alpha e_{n+1}+\cos\alpha\frac{\p\mfR^{n+1}_+(\mfn_\mbP)}{\Abs{\p\mfR^{n+1}_+(\mfn_\mbP)}}
\eqqcolon\sin\alpha e_{n+1}+\cos\alpha\bar\mfn_\mbP,
}
where $\alpha\in[0,\frac\pi2]$ is the contact angle of $\mbP$ with $\p\mfR^{n+1}_+$.

If ${\rm spt}\mu_\mcC=\mbP$ is contained in the closed ($n+1$)-dimensional half-space 
\eq{
\mbH^-=\{x\in\mfR^{n+1}:\left<x,\nu_\mbH\right>\leq0\}
}
for some $\nu_\mbH\in\mfS^n$ having the expression
\eq{
\nu_\mbH
=-\sin\vartheta\bar\mfn_\mbP+\cos\vartheta e_{n+1},
}
where 
\eq{
\vartheta\coloneqq\arccos\left(\left<\nu_\mbH,e_{n+1}\right>\right)\in(0,\pi).
}
Then there holds
\eq{
\vartheta\geq\alpha.
}
Moreover, if equality holds we must have
\eq{
\mbP
\subseteq\p\mbH^-\cap\overline{\mfR^{n+1}_+}.
}
\end{corollary}
\begin{proof}
The proof follows easily from the fact that $\mcC$ is a multiplicity-$1$ half $m$-plane, and hence omitted.
\end{proof}
\begin{corollary}
Under the assumptions and notations in Corollary \ref{Prop:bdry-SMP}, assume in addition $V\in{\bf RV}^m_\beta(\overline\Om)$ with multiplicity one rectifiable boundary varifold $\Gamma$ in the sense of Definition \ref{Defn:rectifiable-bdry}.

Then the conclusions of Corollary \ref{Prop:bdry-SMP} are true with $\alpha$ replaced by $\beta(x_0)$ if $\beta(x_0)\in(0,\frac\pi2)$; or $\pi-\beta(x_0)$ if $\beta(x_0)\in(\frac\pi2,\pi)$.
\end{corollary}
\begin{proof}
This follows directly from the second assertion of Theorem \ref{Prop:Classify-VarTan}.
In particular, the third equality of \eqref{eq-alpha-beta(x_0)}.
\end{proof}
%

\section{Rectifiability of boundary measures}\label{Sec:6}

\subsection{Tangential part}

We record some basic properties resulting from the definition of $\sigma_{\ast\Gamma}$.
\begin{lemma}\label{Lem:sigma_astGamma<<H^m-1}
Let $\Om\subset\mfR^{n+1}$ be a bounded domain of class $C^2$ and $\beta\in C^1(S,(0,\pi))$.
Let $V\in{\bf V}^m_\beta(\Om)$ with boundary varifold $\Gamma$, such that $\mfH\in L^p(\mu_V)$ for some $p\in(m,\infty)$.
Then for any $\sigma_\Gamma$-measurable set $A\subseteq \mathscr{P}_{\rm cb}(\Gamma)$,
$\sigma_{\ast\Gamma}\llcorner A<<\mcH^{m-1}$.
\end{lemma}
\begin{proof}
Since $p>m$, we have by \eqref{ineq:sigma_Gamma-ratio}
\eq{
\Theta^{\ast(m-1)}(\sigma_{\ast\Gamma}\llcorner A,x)<\infty
}
for $\sigma_{\ast\Gamma}\llcorner A$-a.e. $x$ (in fact, for every $x\in A\subset \mathscr{P}_{\rm cb}$).
The claimed property then follows directly from the comparison Theorem for upper density, see e.g., \cite[Theorem 6.4]{Mag12}.
\end{proof}
At the points $x$ of density $1$ for the set $E_\ast\cap A$ ($E_\ast$ defined in \eqref{defn:E_ast}) with respect to the measure $\sigma_\Gamma$, using direct computations, see e.g., \cite[Lemma 5.3]{DeMasi21} (see also \cite[Remark 3.13]{DeLellis08}), one can show that
\eq{
{\rm Tan}^{m-1}(\sigma_{\ast\Gamma}\llcorner A,x)
={\rm Tan}^{m-1}(\sigma_{\Gamma},x).
}
By virtue of Lebesgue differentiation Theorem this holds for $\sigma_\Gamma$-a.e. $x\in E_\ast\cap A$, which implies in conjunction with Lemma \ref{Lem:blow-up-seq} the following fact.

\begin{lemma}\label{Lem:Tan^m-1-sigma_Gamma_infty}
Let $\Om\subset\mfR^{n+1}$ be a bounded domain of class $C^2$ and $\beta\in C^1(S,(0,\pi))$.
Let $V\in{\bf V}^m_\beta(\Om)$ with boundary varifold $\Gamma$, such that $\mfH\in L^p(\mu_V)$ for some $p\in(m,\infty)$.
Then for any $\sigma_\Gamma$-measurable set $A\subseteq \mathscr{P}_{\rm cb}(\Gamma)$, there holds
\eq{
{\rm Tan}^{m-1}(\sigma_{\ast\Gamma}\llcorner A,x)
=\{\sigma_{\Gamma_\infty}:\mcC\in{\rm VarTan}(V,x)\},\quad\text{for }\sigma_{\ast\Gamma}\llcorner A\text{-a.e. }x.
}
\end{lemma}
\begin{proof}[Proof of Theorem \ref{Thm:bdry-rectifiability}]
The idea is to apply the Marstrand-Mattila Rectifiability Criterion (Lemma \ref{Lem:M-M-Reciti-Crit}) to the measure $\sigma_{\ast\Gamma}\llcorner A$.

Since $p>m$, the upper density part of the first condition of Lemma \ref{Lem:M-M-Reciti-Crit} follows easily from \eqref{ineq:sigma_Gamma-ratio}.

To check the second condition of Lemma \ref{Lem:M-M-Reciti-Crit}, we first use Lemma \ref{Lem:Tan^m-1-sigma_Gamma_infty} and obtain for $\sigma_{\ast\Gamma}\llcorner A$-a.e. $x$ the characterization
\eq{
{\rm Tan}^{m-1}(\sigma_{\ast\Gamma}\llcorner A,x)
=\{\sigma_{\Gamma_\infty}:\mcC\in{\rm VarTan}(V,x)\}.
}
Thanks to \eqref{eq:sigma_V-G_pb} we just have to consider those points belonging to $A\cap F$, also recall that by definition of $\sigma_{\ast\Gamma}$ we have $\Theta_\ast^{m-1}(\sigma_\Gamma,x)>0$ for any such $x$.
It then follows from Lemma \ref{Lem:sigma_Gamma_infty-characterization} that for $\sigma_{\ast\Gamma}\llcorner A$-a.e. $x$, any $(m-1)$-blow-up of $\sigma_{\ast\Gamma}\llcorner A$ takes the form $\theta_0\mcH^{m-1}\llcorner L^{m-1}$ for some $(m-1)$-dimensional linear subspace $L^{m-1}\subset\mfR^{n+1}$, which verifies the second condition of Lemma \ref{Lem:M-M-Reciti-Crit}.
As a by-product, the lower density part of the first condition of Lemma \ref{Lem:M-M-Reciti-Crit} is also verified, which completes the proof.

\end{proof}

\begin{corollary}\label{Cor-n_V-perp-app.tangentspace}
Under the assumptions of Theorem \ref{Thm:bdry-rectifiability}, $\mfn_V$ is perpendicular to the approximate tangent space of the $(m-1)$-rectifiable measure $\sigma_{\ast\Gamma}\llcorner A$ for $\sigma_{\ast\Gamma}\llcorner A$-a.e. 
\end{corollary}
\begin{proof}
By Lemmas \ref{Lem:Tan^m-1-sigma_Gamma_infty}, \ref{Lem:sigma_Gamma_infty-characterization}, for $\sigma_{\ast\Gamma}\llcorner A$-a.e. $x_0$, the approximate tangent space of $\sigma_{\ast\Gamma}\llcorner A$ at $x_0$ is an $(m-1)$-dimensional linear subspace of $L_\infty$.
This implies, because of the construction of $L_\infty$, that $\mfn_W(x_0)$ is perpendicular to this approximate tangent space, and hence so is $\mfn_V(x_0)$.
\end{proof}

\subsection{Normal part}

To discuss rectifiability of the \textit{normal part} of the boundary measure, in view of the smooth examples (see e.g., Example \ref{exam:submflds}) it is natural to ask if $\sin\beta\sigma_V^\perp$ agrees with $\sigma_\Gamma$ as measures (or if ${\rm spt}\sigma_V^\perp$ agrees with ${\rm spt}\sigma_\Gamma$) so that Theorem \ref{Thm:bdry-rectifiability} already gives the desired result.
The answer is \textit{no}, as indicated from the examples given by the unions of mutually intersecting free boundary submanifolds and capillary submanifolds, see e.g., Example \ref{exam:caps}.
In this situation $\sigma_\Gamma$ is given by the boundary measure of the capillary submanifolds, yet $\sigma_V^\perp$ consists of not only the boundary measure of the capillary submanifolds but also that of the free boundary submanifolds.

In fact, we are able to show the following rectifiability of the \textit{normal part} of the boundary measure when it is restricted to $\{x\in S:\cos\beta(x)=0\}$.
Similar to $\sigma_{\ast\Gamma}$ we define the ``at most $(m-1)$-dimensional'' part of $\sigma_V^\perp$ to be the restriction of $\sigma_V^\perp$ to the points with strictly positive lower $(m-1)$-density, denoted by $\sigma_{\ast V}^\perp$.
\begin{theorem}[Boundary rectifiability: normal part]\label{Thm:bdry-rectifiability-normal}
Let $\Om\subset\mfR^{n+1}$ be a bounded domain of class $C^2$, and $\beta\in C^1(S,(0,\pi))$.
Let $V\in{\bf V}^m_\beta(\Om)$ with boundary varifold $\Gamma$, such that $\mfH\in L^p(\mu_V)$ for some $p\in(m,\infty)$, and $\Theta^m(\mu_V,x)\geq a>0$ for $\mu_V$-a.e. $x$.

Then $\sigma_{\ast V}^\perp\llcorner\left\{x\in S:\cos\beta(x)=0\right\}$ is $(m-1)$-rectifiable.
\end{theorem}
\begin{proof}
Fix any $x_0\in{\rm spt}\sigma_{\ast V}^\perp\cap\left\{x\in S:\cos\beta(x)=0\right\}$, after translation and rotation we assume that $x_0=0\in S$, $T_{x_0}\overline\Om=\overline{\mfR^{n+1}_+}$.

It follows from \cite[Lemma 5.4]{DeMasi21} that ${\rm VarTan}(V,x_0)$ is non-empty and any $\mcC\in{\rm VarTan}(V,x_0)$ is a stationary free boundary rectifiable cone in $\mfR^{n+1}_+$ such that the first variation of $\mcC$ with respect to compactly supported $\varphi\in\mathfrak{X}(\mfR^{n+1}_+)$ is given by
\eq{
\de\mcC(\varphi)
=-\int_{\p\mfR^{n+1}_+}\left<e_{n+1},\varphi\right>\rd\sigma_\mcC^\perp,
}
which guarantees us to repeat the proof of \cite[Lemmas 5.6$\sim$5.8]{DeMasi21}, thus verifying the second condition of the Marstrand-Mattila Rectifiability Criterion
for the measure $\sigma_{\ast V}^\perp\llcorner\left\{x\in S:\cos\beta(x)=0\right\}$.

To verify the first condition of Lemma \ref{Lem:M-M-Reciti-Crit}, note that the lower bound is a direct consequence of the definition of $\sigma_{\ast V}^\perp$, while the upper bound follows from a similar argument as \eqref{ineq:sigma_Gamma-ratio}.
More precisely, we could use Lemma \ref{Lem:blow-up-seq} ({\bf 1}), \eqref{esti:global-perp}, and $p>m$ to show that
\eq{
\Theta^{\ast(m-1)}(\sigma_V^\perp,x_0)<\infty,
}
which completes the proof.
\end{proof}

A follow-up consideration is whether we could talk about the rectifiability of $\sigma_V^\perp$ on $\mathscr{P}_{\rm cb}(\Gamma)$.
One possible solution is, in view of Example \ref{exam:caps}, to show that ${\rm spt}\sigma_V^\perp$ agrees with ${\rm spt}\sigma_\Gamma$ locally at any $x\in \mathscr{P}_{\rm cb}(\Gamma)$, provided certain density restrictions.
Once this is done, we could combine Theorem \ref{Thm:bdry-rectifiability} with Theorem \ref{Thm:bdry-rectifiability-normal} to obtain the $(m-1)$-rectifiability of the measure
\eq{
\sigma_{\ast V}^\perp\llcorner\left(A\cup\{x\in S:\cos\beta(x)=0\}\right),\quad\text{for any }\sigma_\Gamma\text{-measurable set }A\subseteq \mathscr{P}_{\rm cb}(\Gamma).
}
In certain cases it is true that ${\rm spt}\sigma_V^\perp$ agrees with ${\rm spt}\sigma_\Gamma$, 
see the discussion subsequent to the proof of Theorem \ref{Prop:Classify-VarTan}.

\section{Integral compactness}\label{Sec:7}

\begin{proof}[Proof of Theorem \ref{Thm:compactness}]

We divide the proof into two steps.

\noindent{\bf Step 1. We prove ($i$) and ($ii$).}

We claim that $\{V_j\}_{j\in\mbN}$ have uniformly bounded first variation.
By \eqref{eq:1st-variation-V_beta} and standard disintegration, we have for $j\in\mbN$
\eq{\label{eq:de-V_j(varphi)}
\de V_j(\varphi)
=&-\int_{\overline\Om}\left<\mfH_j,\varphi\right>\rd\mu_{V_j}-\int_S\left<\widetilde H_j,\varphi\right>\rd\mu_{V_j}\\
&-\int_{S}\left<\nu^S,\varphi\right>\rd\sigma_{V_j}^\perp-\int_{G_{m,\beta}(S)}\left<\mfn(x,P),\varphi^T(x)\right>\rd{\Gamma_j}(x,P),\quad\forall\varphi\in\mathfrak{X}(\Om).
}
Note that thanks to Proposition \ref{Prop:1st-variation}, $\{\norm{\widetilde H_{j}}_{L^\infty(\mu_{V_{j}})}\}_{j\in\mbN}$ are uniformly bounded.

In view of \eqref{condi:uniform-bdd-H-mass}, to show the claim we just have to show that $\{\sigma_{V_j}^\perp(S)\}_{j\in\mbN}$ and $\{\sigma_{\Gamma_j}(S)\}_{j\in\mbN}$ are uniformly bounded.
Thanks to \eqref{esti:global-perp}, $\{\sigma_{V_j}^\perp(S)\}_{j\in\mbN}$ are uniformly bounded.
The uniform bound on $\{\sigma_{\Gamma_j}(S)\}_{j\in\mbN}$ can be obtained by virtue of the local estimate \eqref{esti:local-tangent-improved} (which holds for a uniform scale $\tilde\rho$ depending on $S,\rho_0$ thanks to condition ({\bf 3})), together with a covering argument.
The claim is thus proved.
As a by-product of the uniform bound on $\{\sigma_{\Gamma_j}(S)\}_{j\in\mbN}$, we deduce that after passing to a subsequence $\Gamma_{j_k}\wsc\Gamma$.

Then we apply the classical Allard's integral compactness Theorem, see e.g., \cite[Theorem 42.7]{Simon83}, to deduce the subsequence convergent of $V_{j_k}\ra V$, $\de V_{j_k}\ra\de V$, and conclusions ($i$) and ($ii$).
In particular, if for each $j$, $V_j\in{\bf IV}^m(\overline\Om)$, then so is $V$.

As in \eqref{eq:widehat-H}, for $k\in\mbN$, we put $\widehat H_{j_k}\mu_{V_{j_k}}\coloneqq\mfH_{j_k}\mu_{V_{j_k}}\llcorner\overline\Om+\widetilde H_{j_k}\mu_{V_{j_k}}\llcorner S$,
then define the vector-valued Radon measure
$\vec{\mathcal{V}}_{j_k}=\widehat H_{j_k}\mu_{V_{j_k}}$.
Using H\"older inequality, we get
\eq{
\abs{\vec{\mathcal{V}}_{j_k}}(\mfR^{n+1})
=\int\abs{\widehat H_{j_k}}\rd\mu_{V_{j_k}}
\leq\left(\int\abs{\widehat H_{j_k}}^p\rd\mu_{V_{j_k}}\right)^\frac1p\mu_{V_{j_k}}(\overline\Om)^{1-\frac1p},
}
and hence by \eqref{condi:uniform-bdd-H-mass} and the uniform bound on $\{\norm{\widetilde H_{j_k}}_{L^\infty(\mu_{V_{j_k}})}\}_{k\in\mbN}$, we find
\eq{
\sup_{k\in\mbN}\left\{\abs{\vec{\mathcal{V}}_{j_k}}(\mfR^{n+1})\right\}
<\infty.
}
By weak-star compactness (see e.g., \cite[Corollary 4.34]{Mag12}), there exists a vector-valued Radon measure $\vec{\mathcal{V}}$, to which $\vec{\mathcal{V}}_{j_k}$ subsequentially converges to.
By the convergence $V_{j_k}\ra V$ and the fact that the weight measure of a varifold is in fact a push-forward measure, we have the convergence $\mu_{V_{j_k}}\wsc\mu_V$.

Then we argue as in the proof of \cite[Proposition 4.30]{Mag12}.
Consider any bounded open set $A\subset\mfR^{n+1}$, and let $A_t=\{x\in A:{\rm dist}(x,\p A)>t\}$ for $t>0$.
Let $\phi\in C_c^1(A,[0,1])$ be such that $\chi_{A_t}\leq\phi$, then we have (put $p'=\frac{p}{p-1}$, $\Lambda=\sup_{k\in\mbN}\left\{\int\abs{\hat H_{j_k}}^p\rd \mu_{V_{j_k}}\right\}<\infty$)
\eq{
\abs{\vec{\mathcal{V}}}(A_t)
\leq&\liminf_{k\ra\infty}\abs{\vec{\mathcal{V}}_{j_k}}(A_t)
\leq\liminf_{k\ra\infty}\int\phi\abs{\hat H_{j_k}}\rd\mu_{V_{j_k}}\\
\leq&\liminf_{k\ra\infty}\left(\int\abs{\hat H_{j_k}}^p\rd\mu_{V_{j_k}}\right)^\frac1p\left(\int\phi^{p'}\rd\mu_{V_{j_k}}\right)^\frac1{p'}
\leq\Lambda^\frac1p\liminf_{k\ra\infty}\left(\int\phi^{p'}\rd\mu_{V_{j_k}}\right)^\frac1{p'}\\
\leq&\Lambda^\frac1p\left(\int\phi^{p'}\rd\mu_V\right)^\frac1{p'}
\leq\Lambda^\frac1p\mu_V(A)^\frac1{p'},
}
and hence we get $\abs{\vec{\mathcal{V}}}(A)\leq\Lambda^\frac1p\mu_V(A)^\frac1{p'}$ for any Borel set $A$,
implying that ${\vec{\mathcal{V}}}$ is absolutely continuous with respect to $\mu_V$ and, by virtue of \cite[Corollary 5.11]{Mag12}, is of the type $\vec{\mathcal{V}}=\widehat H\mu_V$ for $\widehat H\in L^1(\mu_V)$.

Since for $\varphi\in\mathfrak{X}(\Om)$, the function $\left<\mfn(x,P),\varphi^T(x)\right>\in C^1(G_{m,\beta}(S))$, by the convergence $\Gamma_{j_k}\wsc\Gamma$ we have
\eq{
\lim_{k\ra\infty}\int_{G_{m,\beta}(S)}\left<\mfn(x,P),\varphi^T(x)\right>\rd{\Gamma_j}(x,P)
=\int_{G_{m,\beta}(S)}\left<\mfn(x,P),\varphi^T(x)\right>\rd{\Gamma}(x,P).
}
Letting $j_k\ra\infty$ in \eqref{eq:de-V_j(varphi)}, by the above convergences we obtain
\eq{
\de V(\varphi)
=&-\int_{\overline\Om}\left<\widehat H,\varphi\right>\rd\mu_{V}\\
&-\int_{S}\left<\nu^S,\varphi\right>\rd\sigma_{V}^\perp-\int_{G_{m,\beta}(S)}\left<\mfn(x,P),\varphi^T(x)\right>\rd{\Gamma}(x,P),\quad\forall\varphi\in\mathfrak{X}(\Om),
}
which shows that $V\in{\bf V}^m_\beta(\Om)$ with boundary varifold $\Gamma$.

\noindent{\bf Step 2. We prove ($iii$) and ($iv$).
}

We first use $\phi(x,P)=\varphi(x)$ to test $\Gamma_{j_k}\wsc\Gamma$, using disintegration this gives
\eq{
\lim_{k\ra\infty}\int_{S}\varphi(x)\rd\sigma_{\Gamma_{j_k}}(x)
=\int_S\varphi(x)\rd\sigma_\Gamma(x).
}
That is,
\eq{\label{eq:sigma_Gammajk-sigma_Gamma}
\sigma_{\Gamma_{j_k}}\wsc\sigma_\Gamma
}
as Radon measures on $S$.
Since $B_{\rho_{x_0}}(x_0)$ is compact, it follows immediately from condition ($4$) that
\eq{
\sigma_\Gamma(B_{\frac{\rho_{x_0}}2}(x_0))
\geq\limsup_{k\ra\infty}\sigma_{\Gamma_{j_k}}(B_{\frac{\rho_{x_0}}2}(x_0))
\geq a_{x_0},
}
which proves ($iii$).

Consider the standard smooth cut-off function $\gamma$ on $\mfR$ with $\gamma\equiv1$ on $[0,\frac12]$, $\gamma\equiv0$ on $[1,\infty)$, $-3\leq\gamma'\leq0$ on $\mfR$.
By conditions ($4$) and ($5$) we have
\eq{
&\int_S\gamma(\frac{\abs{x-x_0}}{\rho_{x_0}})\left<T_xS\left(\mfn_{V_j}(x)\right),\tau_{x_0}\right>\rd\sigma_{\Gamma_j}(x)\\
=&\int_S\gamma(\frac{\abs{x-x_0}}{\rho_{x_0}})\left<\cos\beta(x)\mfn_{W_j}(x),\tau_{x_0}\right>\rd\sigma_{\gamma_j}(x)
\geq\ep_{x_0}\sigma_{\Gamma_j}(B_{\frac{\rho_{x_0}}2}(x_0))\geq\ep_{x_0}a_{x_0},\quad\forall j.
}
Also notice that by disintegration and \eqref{defn:n_V-intro}
\eq{
&\int_{G_{m,\beta}(S)}\gamma(\frac{\abs{x-x_0}}{\rho_{x_0}})\left<T_xS\left(\mfn(x,P)\right),\tau_{x_0}\right>\rd\Gamma_{j}(x,P)\\
=&\int_S\gamma(\frac{\abs{x-x_0}}{\rho_{x_0}})\left<T_xS\left(\int_{G_{m,\beta}(x)}\mfn(x,P)\rd\Gamma_{j}^{x}(P)\right),\tau_{x_0}\right>\rd\sigma_{\Gamma_{j}}(x)\\
=&\int_S\gamma(\frac{\abs{x-x_0}}{\rho_{x_0}})\left<T_xS\left(\mfn_{V_j}(x)\right),\tau_{x_0}\right>\rd\sigma_{\Gamma_j}(x)\\
=&\int_S\gamma(\frac{\abs{x-x_0}}{\rho_{x_0}})\left<\cos\beta(x)\mfn_{W_j}(x),\tau_{x_0}\right>\rd\sigma_{\Gamma_j}(x).
}

Since $\gamma(\frac{\abs{x-x_0}}{\rho_{x_0}})\left<T_xS\left(\mfn(x,P)\right),\tau_{x_0}\right>$ is a continuous function on $G_{m,\beta}(S)$,
it follows from the convergence $\Gamma_{j_k}\wsc\Gamma$ that
\eq{
&\int_S\gamma(\frac{\abs{x-x_0}}{\rho_{x_0}})\left<T_xS\left(\mfn_{V}(x)\right),\tau_{x_0}\right>\rd\sigma_{\Gamma}(x)\\
=&\int_{G_{m,\beta}(S)}\gamma(\frac{\abs{x-x_0}}{\rho_{x_0}})\left<T_xS\left(\mfn(x,P)\right),\tau_{x_0}\right>\rd\Gamma(x,P)\\
=&\lim_{k\ra\infty}\int_{G_{m,\beta}(S)}\gamma(\frac{\abs{x-x_0}}{\rho_{x_0}})\left<T_xS\left(\mfn(x,P)\right),\tau_{x_0}\right>\rd\Gamma_{j_k}(x,P)\\
=&\lim_{k\ra\infty}\int_{S}\gamma(\frac{\abs{x-x_0}}{\rho_{x_0}})\left<\cos\beta(x)\mfn_{W_{j_k}(x)},\tau_{x_0}\right>\rd\sigma_{\Gamma_{j_k}}(x).
}
Combining the above facts, we get
\eq{
0<\ep_{x_0}a_{x_0}
\leq&\lim_{k\ra\infty}\int_S\gamma(\frac{\abs{x-x_0}}{\rho_{x_0}})\left<\cos\beta(x)\mfn_{W_{j_k}}(x),\tau_{x_0}\right>\rd\sigma_{\Gamma_{j_k}}(x)\\
=&\int_S\gamma(\frac{\abs{x-x_0}}{\rho_{x_0}})\left<T_xS\left(\mfn_{V}(x)\right),\tau_{x_0}\right>\rd\sigma_{\Gamma}(x)
\leq\int_{S\cap B_{\rho_{x_0}}(x_0)}\Abs{T_xS\left(\mfn_V(x)\right)}\rd\sigma_\Gamma(x),
}
which proves ($iv$) and completes the proof.

\end{proof}
\begin{remark}
\normalfont
As shown in the proof, condition ($3$) can be simply replaced by the stronger assumption that
\eq{
\sup_{j\in\mbN}\{\sigma_{\Gamma_j}(S)\}
<\infty.
}
\end{remark}

Theorem \ref{Thm:compactness} can be stated in the following more general form:
\begin{theorem}
Let $\{\Om_j\}_{j\in\mbN},\Om$ be bounded domains of class $C^2$ in $\mfR^{n+1}$ such that
\eq{
\overline\Om_j\ra\overline\Om\text{ in the sense of }C^2\text{-topology as }j\ra\infty.
}
For $j\in\mbN$, let $S_j$ denote the boundary of $\Om_j$, $\beta_j\in C^1(S_j,(0,\pi))$ with $\beta_j$ converges as $j\ra\infty$,
in the sense of $C^1$-topology, to $\beta\in C^1(S,(0,\pi))$, where $S=\p\Om$.
Suppose that $\{V_j\}_{j\in\mbN}$ is a sequence of rectifiable $m$-varifolds such that $V_j\in{\bf V}^m_{\beta_j}(\Om_j)$ with boundary varifolds $\{\Gamma_j\}_{j\in\mbN}$.
For a fixed $p\in(1,\infty)$, if for each $j$,
\begin{enumerate}
    \item There exists a universal constant $a>0$ such that
\eq{
\Theta^m(\mu_{V_j},x)\geq a>0,\text{ }\mu_{V_j}\text{-a.e. }x.
}
    \item $\mfH_j\in L^p(\mu_{V_j})$ and
\eq{
\sup_{j\in\mbN}\left(\mu_{V_j}(\overline\Om_j)+\norm{\mfH_j}_{L^p(\mu_{V_j})}\right)<\infty.
}
    \item There exist a dense set $E_j$ in ${\rm spt}\sigma_{\Gamma_j}$, universal constants $\ep_0,\rho_0>0$, such that Definition \ref{Defn:capillary-bdry-point} (C1) is satisfied by every $x\in E_j$ with corresponding unit vector $\tau^{V_j}_x\in T_{x}S_j$ locally in $B_{\rho_0}(x)$.
\end{enumerate}
Then the conclusions of Theorem \ref{Thm:compactness} hold.
\end{theorem}
\begin{proof}
Since $\Om_j\ra\Om$, $\beta_j\ra\beta$ as $j\ra\infty$, we could assume that $\{S_j\}_{j\in\mbN}$ have uniformly bounded geometry, namely, $\{\abs{h_{S_j}}_{C^0(S_j)}\}_{j\in\mbN}$ are uniformly bounded.
Furthermore, we have the convergences
\eq{
G_m(S_j)\ra G_m(S),\quad
G_{m,\beta_j}(S_j)
\ra G_{m,\beta}(S)\text{ as }j\ra\infty.
}
Repeating the proof of Theorem \ref{Thm:compactness} and note that when using local estimate \eqref{esti:local-tangent-improved}, we have a uniform upper bound of the constant $C_0=C_0(\Om_j)$ on the RHS for each $j$ thanks to the fact that $\{S_j\}_{j\in\mbN}$ have uniformly bounded geometry.
The rest of the proof is essentially the same.
\end{proof}

\section{Curvature varifolds with capillary boundary}\label{Sec:8}
\subsection{Basic properties}

\begin{lemma}
Let $\Om\subset\mfR^{n+1}$ be a bounded domain of class $C^2$, $\beta\in C^1(S, (0,\pi))$.
Then
the boundary varifold $\Gamma$ with respect to $V,\beta$ in the sense of Definition \ref{Defn:CV-capillary} is unique.
\end{lemma}
\begin{proof}
We argue by contradiction.
Suppose there are $\Gamma_1$, $\Gamma_2$ that satisfy the definition for the same $V,\beta$ and $B\in L^1(V)$.
Subtracting the identities \eqref{defn:CV-capillary}, we get for all $\phi\in C^1(\mfR^{n+1}\times\mfR^{(n+1)^2},\mfR^{n+1})$
\eq{\label{eq:mu_1-mu_2}
\int_{G_{m,\beta}(S)}\left<\mfn(x,P),\phi(x, P)\right>\rd(\Gamma_1-\Gamma_2)(x,P)
=0,
}
which also holds for all $\phi\in C^0(\mfR^{n+1}\times\mfR^{(n+1)^2},\mfR^{n+1})$ by approximation.

Recall that $\sin\beta(x)>0$ by definition,
for a given $\varphi\in C^0(\mfR^{n+1}\times\mfR^{(n+1)^2})$ we define $\phi(x,P)=\frac{\varphi(x,P)}{\sin\beta(x)}\nu^S(x)$.
Using suitable extension, we obtain a vector field $\phi\in C^0(\mfR^{n+1}\times\mfR^{(n+1)^2},\mfR^{n+1})$ such that for all $(x,P)\in G_{m,\beta}(S)$,
\eq{
\left<\mfn(x,P),\phi(x,P)\right>
=\varphi(x,P),
}
where we have used \eqref{eq:n(x,Q),nu^S}.

Testing \eqref{eq:mu_1-mu_2} with such $\phi$ we get a signed measure $\Lambda$ on $G_{m,\beta}(S)$ such that for all  $\varphi\in C^0(\mfR^{n+1}\times\mfR^{(n+1)^2})$,
\eq{
\int_{G_{m,\beta}(S)}\varphi(x,P)\rd\Lambda(x,P)
=0,
}
which
implies that $\Lambda=0$ and consequently the uniqueness of the boundary varifold.
\end{proof}

\begin{lemma}\label{Lem:basic-prop-CV^m_beta}
Let $\Om\subset\mfR^{n+1}$ be a bounded domain of class $C^2$, $\beta\in C^1(S, (0,\pi))$, and 
let $V\in{\bf CV}^m_\beta(\Om)$ in the sense of Definition \ref{Defn:CV-capillary}.
Then
$V$ has bounded first variation
and $\Theta^m(\mu_V,x)$ exists for $\mu_V$-a.e. $x$.
Moreover, if $\Theta^m(\mu_V,x)\geq a>0$ for $\mu_V$-a.e. $x$,
then $V\in{\bf RV}^m(\overline\Om)$.
\end{lemma}
\begin{proof}
For test functions $\phi(x,P)=\varphi(x)$ in \eqref{defn:CV-capillary}, the identity simplifies as follows
\eq{\label{eq:1stvariation-CV_beta}
\de V(\varphi)
&=\int_{G_m(\mfR^{n+1})}\left<\na\varphi(x),P\right>\rd V(x,P)\\
=-&\int_{G_m(\mfR^{n+1})}\left<{\rm tr}B(x,P),\varphi(x)\right>\rd V(x,P)-\int_{G_{m,\beta}(S)}\left<\mfn(x,P),\varphi(x)\right>\rd\Gamma(x,P).
}
As $\abs{{\rm tr}B}\leq\sqrt{m}\abs{B}$, we have
\eq{
\abs{\de V(\varphi)}
\leq\left(\sqrt{m}\norm{B}_{L^1(V)}+\sigma_\Gamma(S)\right)\norm{\varphi}_{C^0(\overline\Om)}.
}
Once we have shown that $\sigma_\Gamma(S)$ is controlled in terms of $\mu_V(\overline\Om)$ and the curvature energy, which we postpone to Lemma \ref{Lem:comparable-mass}, we then obtain
\eq{
\abs{\de V(\varphi)}
\leq C\left(\norm{B}_{L^1(V)}+\mu_V(\overline\Om)\right)\norm{\varphi}_{C^0(\overline\Om)}.
}
Note that $\mu_V(\overline\Om)<\infty$ as $\overline\Om$ is compact, thus $V$ has bounded first variation in $\overline\Om$.
As a by-product, we have the existence of $\Theta^m(\mu_V,x)$ for $\mu_V$-a.e. $x$, thanks to \cite[Lemma 40.5]{Simon83}.
The last statement follows from Rectifiability Theorem for varifolds, see \cite[Theorem 42.4]{Simon83}.
This completes the proof.
\end{proof}

\begin{proposition}
Let $\Om\subset\mfR^{n+1}$ be a bounded domain of class $C^2$, $\beta\in C^1(S, (0,\pi))$.
Let $V\in{\bf CV}^m_\beta(\Om)$ in the sense of Definition \ref{Defn:CV-capillary} with weak second fundamental form $B$ and boundary varifold $\Gamma$.
Then $V\in{\bf V}^m_\beta(\Om)$ with boundary varifold $\Gamma$ in the sense of Definition \ref{Defn:vfld-prescribed-bdry}.
In particular,
\eq{\label{inclu:CV_beta-V_beta}
{\bf CV}^m_\beta(\Om)
\subset{\bf V}^m_\beta(\Om).
}
\end{proposition}
\begin{proof}
By standard disintegration we write $V=\mu_V\otimes V^x$, where $\mu_V$ is the weight measure of $V$ and $V^x$ is the Radon probability measure on $G(m,n+1)$ for $\mu_V$-a.e. $x$.
Letting $\varphi\in\mathfrak{X}_t(\Om)$ in \eqref{eq:1stvariation-CV_beta}, we find
\eq{
\de V(\varphi)
=&-\int_{\overline\Om}\left<\int
{\rm tr}B(x,P)\rd V^x(P),\varphi(x)\right>\rd\mu_V(x)-\int_{G_{m,\beta}(S)}\left<\mfn(x,P),\varphi(x)\right>\rd\Gamma(x,P)\\
\eqqcolon&-\int_{\overline\Om}\left<\mfH(x),\varphi(x)\right>\rd\mu_V(x)-\int_{G_{m,\beta}(S)}\left<\mfn(x,P),\varphi(x)\right>\rd\Gamma(x,P).
}
Note also that $\mfH\in L^1(\mu_V)$ since
\eq{\label{ineq-CV-H-B-bound}
\int\abs{\mfH}\rd\mu_V
\leq\int_{\overline\Om}\int_{G(m,n+1)}\Abs{{\rm tr}B(x,P)}\rd V^x(P)\rd\mu_V(x)
\leq\sqrt{m}\int_{G_m(\overline\Om)}\Abs{B}\rd V
<\infty,
}
which shows that $V\in{\bf V}^m_\beta(\Om)$ with boundary varifold $\Gamma$ and generalized mean curvature $\mfH$, the proof is thus completed.
\end{proof}

Next we show that our definition is compatible with the definitions of Mantegazza \cite{Mantegazza96} and Kuwert-M\"uller \cite{KM22}, which we recall in Appendix \ref{Appen-1} for readers' convenience.

First let $\beta\equiv \frac \pi 2$, observe that by definition, any $(x,P)\in G_{m,\frac\pi2}(S)$ admits the unique decomposition
\eq{
P=\nu^S(x)\oplus Q
}
for some $(m-1)$-plane $Q\in T_xS$, which we write as $Q=P\ominus\nu^S(x)$.
Let $\mathbf{q}:G_{m,\frac\pi2}(S)\ra G_{m-1}(TS)$ denote the map
\eq{
\mathbf{q}(x,P)
=(x,P\ominus\nu^S(x)),
}
which is clearly a bijection.
On the other hand, thanks to \eqref{eq:n(x,Q),nu^S} we have for any such $(x,P)$ \eq{
\mfn(x,P)=\nu^S(x).
}
It is then easy to see that 
Definition
\ref{Defn:CV-capillary} is compatible with Definition \ref{Defn:CV-orthogonal} in the following sense:
If $\Gamma$ is the boundary varifold of $V$ with contact angle $\frac\pi2$ in the sense of Definition \ref{Defn:CV-capillary}, then the push-forward measure $\widetilde\Gamma={\bf q}_\ast(\Gamma)$ on $G_{m-1}(TS)$
is such that $V$ is orthogonal to $S$ along $\widetilde\Gamma$ in the sense of Definition \ref{Defn:CV-orthogonal}, vice versa.

Now we show that Definition \ref{Defn:CV-capillary} is compatible with Mantegazza's Curvature Varifolds with boundary:
\begin{lemma}\label{Lem:comptaible}
Let $\Om\subset\mfR^{n+1}$ be a bounded domain of class $C^2$, $\beta\in C^1(S, (0,\pi))$.
Let $V\in{\bf CV}^m_\beta(\Om)\cap{\bf IV}^m(\overline\Om)$
with boundary varifold $\Gamma$ in the sense of Definition \ref{Defn:CV-capillary}.
Then $V$ is a curvature varifold with boundary $\p V=\mfn\Gamma$, in the sense of Definition \ref{Defn:CV-Mantegazza} (with $U=\mfR^{n+1}$).
Moreover, the total variation of the boundary measure in this case is just
\eq{\label{eq-total-variation-measure}
\abs{\p V}
=\abs{\mfn\Gamma}
=\Gamma.
}
\end{lemma}
\begin{proof}
The fact that $\p V=\mfn\Gamma$ in the sense of Definition \ref{Defn:CV-Mantegazza} is a simple consequence of the fact that $G_{m,\beta}(S)\subset G_m(\mfR^{n+1})$.

To show \eqref{eq-total-variation-measure} we verify by definition: for any open $A\subset G_m(\mfR^{n+1})$,
\eq{
\abs{\mfn\Gamma}(A)
=&\sup\{\mfn\Gamma(\phi):\phi\in C_c^0(A;\mfR^{n+1}),\abs{\phi}\leq1\}\\
=&\sup\{\int_{G_{m,\beta}(S)}\left<\mfn(x,P),\phi(x,P)\right>\rd\Gamma(x,P):\phi\in C_c^0(A;\mfR^{n+1}),\abs{\phi}\leq1\}
=\Gamma(A),
}
which completes the proof.
\end{proof}
We write ${\bf ICV}^m_\beta(\Om)$ as the set of curvature varifolds discussed in Lemma \ref{Lem:comptaible}.
Thanks to this observation,
the following nice properties of $V\in{\bf ICV}^m_\beta(\Om)$ inherit directly from \cite{Mantegazza96}.

\begin{lemma}[Uniqueness]\label{Lem:uniqueness}
Let $V\in{\bf ICV}^m_\beta(\Om)$, then
the weak second fundamental form $B$ and the boundary varifold $\Gamma$ are uniquely determined by \eqref{defn:CV-capillary}.
\end{lemma}
\begin{proof}
This follows from \cite[Proposition 3.4]{Mantegazza96}, thanks to Lemma \ref{Lem:comptaible}.
\end{proof}

\begin{proposition}[Tangential properties]\label{Prop-CV-Tangent}
Let $V\in{\bf ICV}^m_\beta(\Om)$, let $P(x)$ be the corresponding approximate tangent space which exists $\mu_V$-a.e., and let $\{e_1,\ldots,e_{n+1}\}$ be the canonical orthonormal basis of $\mfR^{n+1}$.
Then for $\mu_V$-a.e. $x$, the components of the tangent space function $P_{jk}(x)$ are approximately differentiable with approximate gradients
\eq{
\na_{e^T_i}P_{jk}(x)
=B_{ijk}(x,P(x)),
}
where $e_i^T$ denotes the projection of $e_i$ onto the approximate tangent space.
Moreover,
for $\mu_V$-a.e.$x$ the following are true:
\begin{enumerate}
    \item For vectors $u,v,w\in\mfR^{n+1}$,
\eq{
\left<B\mid_{(x,P(x))}(u,v),w\right>
=\left<B\mid_{(x,P(x))}(u,w),v\right>.
}
    \item For any vector $v\in\mfR^{n+1}$
\eq{
\sum_{i=1}^{n+1}\left<B\mid_{(x,P(x))}(v,e_i),e_i\right>
=0.
}
    \item The functions $B_{ijk}(x,P)$ satisfy the relations:
\eq{
\sum_{l=1}^{n+1}P_{il}(x)B_{ljk}(x,P(x))
=B_{ijk}(x,P(x)),
}
also
\eq{
B_{ijk}(x,P(x))
=\sum_{l=1}^n\left(P_{jl}(x)B_{ilk}(x,P(x))+P_{lk}(x)B_{ijl}(x,P(x)\right),
}
and
\eq{
\sum_{l=1}^{n+1}P_{il}(x)H_{l}(x,P(x))=0,
}
where we write $H_l(x,P(x))=\sum_{j=1}^{n+1}B_{jlj}(x,P(x))$.
\end{enumerate}
\end{proposition}
\begin{proof}
The approximate differentiability follows from \cite[Theorem 5.4]{Mantegazza96}, thanks to Lemma \ref{Lem:comptaible}.
The rest of the statement follows as corollaries of the approximate differentiability, see \cite[Propositions 3.6, 3.7]{Mantegazza96}.
\end{proof}

\begin{remark}
\normalfont
The second order rectifiability of integral varifolds of locally bounded first variation shown by Menne in \cite{Menne13}
applies for the class ${\bf ICV}^m_\beta(\Om)$, as it is easy to check that for any $V\in{\bf ICV}^m_\beta(\Om)$, $\norm{\de V}$ is a Radon measure on $\mfR^{n+1}$.
\end{remark}

\begin{lemma}[Singularity of ${\bf ICV}_\beta^m(\Om)$]
Let $V\in{\bf ICV}^m_\beta(\Om)$ with boundary varifold $\Gamma$ in the sense of Definition \ref{Defn:CV-capillary}, then
$\Gamma$ has support included in the support of $V$.
Moreover, $\sigma_\Gamma\perp\mu_V$.
\end{lemma}
\begin{proof}
The assertion follows from \cite[Proposition 3.5]{Mantegazza96} in virtue of \eqref{eq-total-variation-measure}.
\end{proof}

\begin{lemma}[$(m-1)$-rectifiability of boundary measure]
Let $V\in{\bf ICV}^m_\beta(\Om)$ with boundary varifold $\Gamma$ in the sense of Definition \ref{Defn:CV-capillary}, then there exists a countably $(m-1)$-rectifiable set $\mcR_S\subset S$ together with a positive Borel function $\theta:\mcR_S\ra\mfR$ such that $\sigma_\Gamma=\theta\mcH^{m-1}\llcorner\mcR_S$.
\end{lemma}
\begin{proof}
The assertion follows from \cite[Theorem 7.1]{Mantegazza96} in virtue of \eqref{eq-total-variation-measure}.
\end{proof}

Summarizing the boundary properties, we have obtained:
\begin{proposition}\label{Prop-CV-bdry}
Let $V\in{\bf ICV}^m_\beta(\Om)$ with boundary varifold $\Gamma$ in the sense of Definition \ref{Defn:CV-capillary}, then
\begin{enumerate}
    \item There exists a countably $(m-1)$-rectifiable set $\mcR_S\subset S$ together with a positive Borel function $\theta:\mcR_S\ra\mfR$ such that
\eq{
\Gamma
=\theta\mcH^{m-1}\llcorner \mcR_S\otimes\Gamma^{x},
}
where $\Gamma^x$ is the Radon probability measure on $G_{m,\beta}(S)$ resulting from disintegration.
    \item For any $(x,P)\in{\rm spt}\Gamma\subset G_{m,\beta}(S)$, one also has
\eq{
(x,P)\in{\rm spt}V.
}
In particular, $\left<\mfn(x,P),\nu^S(x)\right>
=\sin\beta(x)$.
\end{enumerate}
\end{proposition}

\begin{remark}
\normalfont
Propositions \ref{Prop-CV-Tangent} and \ref{Prop-CV-bdry} extend to our class of curvature varifolds ${\bf ICV}^m_\beta(\Om)$ the nice geometric properties that are valid for the smooth capillary submanifolds.
\end{remark}

\subsection{Compactness and variational problems}
The following are capillary generalizations of the results by Kuwert-M\"uller \cite{KM22}.

\begin{lemma}[Comparable masses]\label{Lem:comparable-mass}
Let $\Om\subset\mfR^{n+1}$ be a bounded domain of class $C^2$, $\beta\in C^1(S, (0,\pi))$, $p\in[1,\infty]$.
Then there exists a positive constant $C=C(n,m,p,\beta,\Om)<\infty$, such that for any $V\in{\bf CV}^m_\beta(\Om)$ with boundary varifold $\Gamma$ in the sense of Definition \ref{Defn:CV-capillary}, there hold
\eq{
\mu_V(\overline\Om)
\leq C\left(\sigma_\Gamma(S)+\norm{B}_{L^p(V)}^p\right),
}
\eq{
\sigma_\Gamma(S)\leq
C\left(\mu_V(\overline\Om)+\norm{B}_{L^p(V)}^p\right).
}
\end{lemma}
\begin{proof}
Taking $\phi(x,P)=x-x_0$ for $x_0\in\Om$ in
\eqref{defn:CV-capillary} we obtain
\eq{
m\mu_V(\overline\Om)
\leq{\rm diam}(\Om)\left(\sigma_\Gamma(S)+\norm{B}_{L^1(V)}\right).
}
By Young's inequality we have for any $p\in(1,\infty)$ and $\ep>0$
\eq{
m\mu_V(\overline\Om)
\leq{\rm diam}(\Om)\left(\sigma_\Gamma(S)+\frac1p\ep^{-p}\norm{B}_{L^p(V)}^p+\frac{p-1}p\ep^\frac{p}{p-1}\mu_V(\overline\Om)\right).
}
Choosing $\ep=\ep(m,p,{\rm diam}(\Om))$ sufficiently small, we could absorb the last term on the Right and get
\eq{
\mu_V(\overline\Om)
\leq C\left(\sigma_\Gamma(S)+\norm{B}_{L^p(V)}\right).
}

On the other hand, we take $\phi=l\na d_S$, where $0\leq l\leq1$ is a cut-off function with $l_{\mid_S}=1$ and supported in $U^+_\de(S)$ for appropriate $\de=\de(\Om)\leq1$.
This yields
\eq{
\sigma_\Gamma(S)
\leq&\frac1{\min_{x\in S}\sin\beta(x)}\int_{G_{m,\beta}(S)}\left<\mfn(x,P),l\na d_S(x)\right>\rd\Gamma(x,P)\\
\overset{\eqref{defn:CV-capillary}}{=}&-\frac1{\min_{x\in S}\sin\beta(x)}\int_{G_m(\mfR^{n+1})}\left(\left<{\rm tr}B,l\na d_S\right>+\left<D_x(l\na d_S),P\right>\right)\rd V(x,P)\\
\leq&C(n,m,\beta,\Om)\left(\mu_V(\overline\Om)+\norm{B}_{L^1(V)}\right),
}
where we have used \eqref{eq:n(x,Q),nu^S} and $\abs{{\rm tr}B}\leq\sqrt{m}\abs{B}$,
the assertion for $p=1$ then holds.

By Young's inequality we have for any $p\in(1,\infty)$
\eq{
\norm{B}_{L^1(V)}
\leq\frac1p\ep^{-p}\norm{B}_{L^p(V)}^p+\frac{p-1}p\ep^\frac{p}{p-1}\mu_V(\overline\Om).
}
Taking for example $\ep=\frac12$,
the assertion for general $p$ then follows.
\end{proof}

\begin{proof}[Proof of Theorem \ref{Thm:compactness-CV}]
First, we obtain from Lemma \ref{Lem:comparable-mass} that $\{\sigma_{\Gamma_j}(S)\}_{j\in\mbN}$ are uniformly bounded.
Thus, using the compactness of Radon measures, after passing to a subsequence we have $V_{j_k}\ra V$ and $\Gamma_{j_k}\wsc\Gamma$.
Moreover, if $\{V_j\}_{j\in\mbN}$ are integral, then we wish to use Allard's integral compactness theorem to deduce that up to extracting a subsequence, $V_{j_k}$ converges as varifolds to an integral varifold $V$.
To this end, we just have to show that $\{V_j\}_{j\in\mbN}$ have uniformly bounded first variation.

Recall that for test functions $\phi(x,P)=\varphi(x)$ in \eqref{defn:CV-capillary}, we have by virtue of \eqref{eq:1stvariation-CV_beta}
\eq{
\abs{\de V_j(\varphi)}
\leq\left(\sqrt{m}\norm{B_j}_{L^1(V_j)}+\sigma_{\Gamma_j}(S)\right)\norm{\varphi}_{C^0(\overline\Om)}.
}
Note that
\begin{enumerate}
    \item $\{\norm{B_j}_{L^1(V_j)}\}_{j\in\mbN}$ are uniformly bounded by virtue of H\"older inequality, and the uniform bounds on $\{\mu_{V_j}(\overline\Om)\}_{j\in\mbN}$ and $\{\norm{B_j}_{L^p(V_j)}\}_{j\in\mbN}$;
    \item $\{\sigma_{\Gamma_j}(S)\}_{j\in\mbN}$ are uniformly bounded as shown in the beginning of the proof. 
\end{enumerate}
Thus we deduce that $\{V_j\}_{j\in\mbN}$ have uniformly bounded first variation as required, and hence $V$ is integral.

Then we define $\vec{\mathcal{V}}_{j_k}=B_{j_k}V_{j_k}$ as vector-valued Radon measures, where $B_{j_k}$ is the vector with $(n+1)^3$ entries $\left(\left<B_{j_k}(e_a,e_b),e_c\right>\right)_{1\leq a,b,c\leq n+1}$.
Using H\"older inequality, we get
\eq{
\abs{\vec{\mathcal{V}}_{j_k}}(\mfR^{n+1}\times\mfR^{(n+1)^2})
=\int\abs{B_{j_k}}\rd V_{j_k}
\leq\left(\int\abs{B_{j_k}}^p\rd V_{j_k}\right)^\frac1p\mu_{V_{j_k}}(\overline\Om)^{1-\frac1p},
}
and hence by assumptions
\eq{
\sup_{k\in\mbN}\left\{\abs{\vec{\mathcal{V}}_{j_k}}(\mfR^{n+1}\times\mfR^{(n+1)^2})\right\}
<\infty.
}
By weak-star compactness (see e.g., \cite[Corollary 4.34]{Mag12}), there exists a vector-valued Radon measure $\vec{\mathcal{V}}$, to which $\vec{\mathcal{V}}_{j_k}$ subsequentially converges to.
Recall also that $V_{j_k}$ converges to $V$ as Radon measures.

Using \cite[Proposition 4.30]{Mag12}, we find that
${\vec{\mathcal{V}}}$ is absolutely continuous with respect to $V$ and is of the type $\vec{\mathcal{V}}=BV$ for $B\in L^1(V)$.

To show that $V,\Gamma$ satisfy Definition \ref{Defn:CV-capillary} with the same $\beta$ and $B$ obtained above,
first notice that for any $\phi\in C^1(G_m(\mfR^{n+1}))$,
the functions
$\left<\mfn(x,P),\phi(x,P)\right>\in C^1(G_{m,\beta}(S))$, and hence
\eq{
\lim_{k\ra\infty}\int_{G_{m,\beta}(S)}\left<\mfn(x,P),\phi(x,P)\right>\rd\Gamma_{j_k}(x,P)
=\int_{G_{m,\beta}(S)}\left<\mfn(x,P),\phi(x,P)\right>\rd\Gamma(x,P).
}
The assertion follows after taking limit in
\eq{
\int
\left(D_P\phi\cdot B_{j_k}+\left<{\rm tr}B_{j_k},\phi\right>+\left<\na_x\phi,P\right>\right)\rd V_{j_k}(x,P)
=-\int_{G_{m,\beta}(S)}\left< 
\mfn(x,P),\phi(x,P)\right>\rd\Gamma_{j_k}(x,P).
}

Finally, \eqref{ineq:liminf-convex-curvature} follows by applying \cite[Theorem 9.6]{Mantegazza96} for the convergences $V_{j_k}\ra V$ and $B_{j_k}V_{j_k}\wsc BV$.
Taking $f(\xi)=\abs{\xi}^p$, we then have
\eq{
\norm{B}_{L^p(V)}^p
\leq\liminf_{k\ra\infty}\norm{B_{j_k}}^p_{L^p(V_{j_k})}
\leq\Lambda
}
as required,
which finishes the proof.
\end{proof}

This theorem can be used to find weak minima of functionals depending on the curvature among surfaces with capillary boundary, for example one can show that any capillary submanifold with normalized area and prescribed boundary information in a given container has uniform $L^p$-curvature lower bound.
\begin{theorem}[Finite $L^p$-curvature]\label{Thm:finite-L^p-curvature}
Let $\Om\subset\mfR^{n+1}$ be a bounded domain of class $C^2$, $\beta\in C^1(S, (0,\pi))$.
For any $p\in(1,\infty]$, define the $L^p$-curvature energy
\eq{
\kappa_\beta^{m,p}(\Om)
\coloneqq\inf\left\{\norm{B}^p_{L^p(V)}:V\in{\bf CV}^m_{\beta}(\Om), \mu_V(\overline\Om)=1\right\}.
}
Then $\kappa_\beta^{m,p}(\Om)$ is finite and the infimum is attained.
\end{theorem}

\begin{proof}
For any minimizing sequences $\{V_j\}_{j\in\mbN}$ we have by virtue of Lemma \ref{Lem:comparable-mass} that the corresponding boundary varifolds $\{\Gamma_j\}_{j\in\mbN}$ have uniformly bounded masses,
the assertion then follows from Theorem \ref{Thm:compactness-CV}.
\end{proof}

\appendix
\section{Curvature varifolds}\label{Appen-1}

Let us first record the classical computations for embedded submanifolds performed in \cite{KM22}, for local expressions see \cite{Hutchinson86,Mantegazza96}.

For an embedded $m$-dimensional submanifold $\S\subset\mfR^{n+1}$ with possibly non-empty boundary $\p\S$, its \textit{second fundamental form} is the symmetric bilinear form defined at every $x\in\S$ by
\eq{
A(x):T_x\S\times T_x\S\ra N_x\S,\quad A(x)(\tau,\eta)
=(\na_\tau\eta)^\perp,
}
where $N_x\S$ is the normal space to $\S$ at $x$ and $\na_\tau\eta$ denotes the covariant differentiation in $\mfR^{n+1}$.
The second fundamental form $A$ can be naturally extended to a symmetric bilinear form on all $\mfR^{n+1}$ with values in $\mfR^{n+1}$ by (the space of such forms is denoted by $BL(\mfR^{n+1}\times\mfR^{n+1},\mfR^{n+1})$)
\eq{
B(x)(v,w)
=A(x)(v^T,w^T)+\sum_{\alpha=1}^m\left<A(x)(v^T,\tau_\alpha),w^\perp\right>\tau_\alpha,
}
where $^T,^\perp$ denote respectively the projection on $T_x\S,N_x\S$, and $\tau_\alpha=\tau_\alpha(x)$ is an orthonormal basis of $T_x\S$, which can be extended to an orthonormal basis $e_i$ of $\mfR^{n+1}$ with $e_\alpha=\tau_\alpha$ for $\alpha=1,\ldots,m$.
The \textit{mean curvature vector} is the trace of $B(x)$, namely,
\eq{
\mfH(x)
=\sum_{\alpha=1}^mA(x)(\tau_\alpha,\tau_\alpha)
=\sum_{i=1}^{n+1}B(x)(e_i,e_i)
={\rm tr}(B(x)).
}

To define curvature varifolds we use $\phi(x,T_x\S)=X(x)$ in the first variation formula of $\S$:
\eq{
{\rm div}_\S(X)
={\rm div}_\S(X^T)-\left<{\rm tr}B,X\right>,\quad\forall X\in C^1(\S,\mfR^{n+1}).
}
Working out the computations and associate to $\S$ the $m$-varifold $V=\mcH^m\llcorner\S\otimes\de_{T_x\S}$, we get
\eq{
\int_{G_m(U)}\left(D_P\phi\cdot B+\left<{\rm tr}B,\phi\right>+\left<\na_x\phi,P\right>\right)\rd V(x,P)
=\int_\S{\rm div}_\S X^T\rd\mcH^m
=-\int_{\p\S}\left<X,\mfn\right>\rd\mcH^{m-1},
}
where $D_P\phi$ denotes the covariant differentiation of $\varphi$ with respect to the variable $P$, $\mfn$ denotes the inwards-pointing unit co-normal along $\p\S\subset\overline\S$, and
\eq{\label{eq:D_Pphi,B}
D_P\phi\cdot B_{\mid_{(x,P)}}
\coloneqq D_{P^k_j}\phi^i(x,P)B_{ij}^k(x,P)
=\left<\left(D_P\phi(x,P)\right)[D_{\tau_\alpha}P],\tau_\alpha\right>.
}

Plugging the designed boundary information of submanifolds on the RHS of the variational formula we reach the various definitions   for curvature varifolds as follows:

\begin{definition}[{Curvature varifolds (with empty boundary) \cite{Hutchinson86}}]
\normalfont
Let $V$ be an $m$-varifold in $U\subset\mfR^{n+1}$.
We say that \textit{$V$ has weak second fundamental form $B\in L_{\rm loc}^1(V)$}, where $B(x,P)\in BL(\mfR^{n+1}\times\mfR^{n+1},\mfR^{n+1})$, if for any $\phi\in C_c^1(U\times\mfR^{(n+1)^2},\mfR^{n+1})$
\eq{
\int_{G_m(U)}\left(D_P\phi\cdot B+\left<{\rm tr}B,\phi\right>+\left<\na_x\phi,P\right>\right)\rd V(x,P)
=0.
}
\end{definition}

\begin{definition}[{Curvature varifolds with boundary \cite{Mantegazza96}}]\label{Defn:CV-Mantegazza}
\normalfont
Let $V$ be an integral $m$-varifold in $U\subset\mfR^{n+1}$.
We say that \textit{$V$ is a curvature varifold with boundary} and has weak second fundamental form $B\in L_{\rm loc}^1(V)$, if there exists a Radon vector measure $\p V$ on $G_m(U)$ with values in $\mfR^{n+1}$ such that for all $\phi\in C_c^1(U\times\mfR^{(n+1)^2},\mfR^{n+1})$
\eq{
\int_{G_m(U)}\left(D_P\phi\cdot B+\left<{\rm tr}B,\phi\right>+\left<\na_x\phi,P\right>\right)\rd V(x,P)
=-\int_{G_m(U)}\left<\phi(x,P),\rd\p V(x,P)\right>.
}
\end{definition}

\begin{definition}[{Curvature varifolds with orthogonal boundary \cite{KM22}}]\label{Defn:CV-orthogonal}
\normalfont
Let $\Om\subset\mfR^{n+1}$ be a bounded domain of class $C^2$.
Let $V$ be an $m$-varifold on $\overline\Om$, and let $\Gamma$ be a Radon measure on $G_{m-1}(TS)$.
We say that \textit{V is a curvature varifold with weak second fundamental form $B\in L^1(V)$ and is orthogonal to $S$ along $\Gamma$} if for all $\phi\in C^1((\mfR^{n+1}\times\mfR^{(n+1)^2},\mfR^{n+1})$
\eq{
\int
\left(D_P\phi\cdot B+\left<{\rm tr}B,\phi\right>+\left<\na_x\phi,P\right>\right)\rd V(x,P)\\
=-\int_{G_{m-1}(TS)}\left<\nu^S(x),\phi(x,\nu^S(x)\oplus Q)\right>\rd\Gamma(x,Q).
}
\end{definition}

\bibliographystyle{amsplain}
\bibliography{BibTemplate.bib}

\end{document}